\documentclass[11pt, a4paper, titlepage]{article}

\usepackage{Stylefile}

\usepackage[T1]{fontenc}
\usepackage[utf8]{inputenc}
\usepackage{lmodern}
\usepackage{textcomp}
\usepackage{microtype}

\usepackage[english]{babel}
\usepackage{csquotes}

\usepackage[backend=biber,
			style=ext-numeric,
			isbn=false,
			doi=false,
            maxbibnames=99,
            sortcites,
            articlein=false
			]{biblatex}
\usepackage{graphicx}
\usepackage{setspace}
\usepackage{caption}
\usepackage{subcaption}

\usepackage{amsmath}
\usepackage{amsfonts}
\usepackage{amssymb}
\usepackage{amsthm}
\usepackage{multicol}
\usepackage{tikz}
\usetikzlibrary{arrows.meta, decorations.markings}
\usepackage{pgfplots}
\usepackage{pdfpages}
\pgfplotsset{compat=newest}
\usepackage{nicematrix}
\usepackage{mathtools}
\usepackage{xcolor}
\usepackage{datetime}

\numberwithin{equation}{section}

\newcommand*{\R}{\mathbb{R}}

\DeclareMathOperator{\trace}{tr}
\DeclareMathOperator{\sgn}{sgn}

\newcommand*{\eps}{\varepsilon}
\newcommand{\comment}[1]{\textcolor{red}{[\quad #1 \quad]}}
\newcommand{\lcomment}[1]{\textcolor{blue}{[\quad #1 \quad]}}

\usepackage{thmtools}

\declaretheorem[
	name=Theorem,
	numberwithin=section
	]{thm}
\declaretheorem[
	name=Lemma,
	sibling=thm,
	]{lem}
\declaretheorem[
	name=Proposition,
	sibling=thm,
	]{prop}

\declaretheorem[
	name=Remark,
	sibling=thm
	]{rem}

\usepackage[pdftex,pdfpagelabels,hidelinks]{hyperref}

\begin{document}


\makeatletter
\tikzset{
    dot diameter/.store in=\dot@diameter,
    dot diameter=2pt,
    dot spacing/.store in=\dot@spacing,
    dot spacing=8pt,
    dots/.style={
        line width=\dot@diameter,
        line cap=round,
        dash pattern=on 0pt off \dot@spacing
    }}
\tikzset{->-/.style={decoration={
  markings,
  mark=at position #1 with {\arrow{{Straight Barb}}}},postaction={decorate}}}

\tikzset{->>-/.style={decoration={
  markings,
  mark=at position #1 with {\arrow{{Straight Barb} {Straight Barb}}}},postaction={decorate}}}

\makeatother

\renewcommand{\shorttitle}{GRN}

\hypersetup{
pdftitle={Different Singular Limits in a Gene Regulatory Network with Multiple Small Parameters},
pdfauthor={L.~Baumgartner, S.~Jelbart},
pdfkeywords={Multi-parameter singular perturbations, Gene regulatory network, Geometric singular perturbation theory, Geometric blow-up},
}

\vspace*{1.2cm}

\begin{center}
\LARGE Different Singular Limits in a Gene Regulatory Network \\ with Multiple Small Parameters
\end{center}

\begin{center}
\renewcommand{\thefootnote}{\alph{footnote}}
    \footnote[0]{\small{Date: July 16, 2026.}}
    \large L.~Baumgartner\footnote{Institute of Analysis and Scientific Computing, TU Wien, Wiedner Hauptstrasse 8-10, 1040, Vienna, Austria. \\
    E-Mail: lukas.baumgartner@tuwien.ac.at} \& S.~Jelbart\footnote{School of Mathematical Sciences, University of Adelaide, North Terrace Campus, 5000, Adelaide, SA, Australia. \\
    E-Mail: sam.jelbart@adelaide.edu.au}  \\
\end{center}

\onehalfspacing 

\setcounter{footnote}{0}

\begin{abstract}
We 
consider a planar ODE system from an important class of models for gene regulatory dynamics. 
The system depends singularly on the \textit{steepness parameters} $0<\eps_1,\eps_2\ll1$ and converges to a piecewise-smooth system
as these parameters tend to zero. 
Unlike previous studies, 
we do not assume that $\eps_1 = \eps_2$. As a consequence, the dynamics when $(\eps_1, \eps_2) \to (0,0)$ depends upon how the limit is taken.
Using a preliminary blow-up in parameter space, we identify three 
distinct singular limits. We perform a two-parameter bifurcation analysis in each case, and apply multiple geometric blow-ups in variable and parameter space to determine the bifurcation structure and the associated
global dynamics.
Bogdanov-Takens bifurcations are revealed in two of three cases, and in one case in particular, a \textit{regularised visible-invisible two-fold singularity} is shown to organise the unfolding of singular bifurcations in the vicinity of canards. 
Our results show that the qualitative dynamics and overall sensitivity of the system to parameter variation depends on the relative size of the steepness parameters. More generally, the analytical framework developed herein provides a systematic approach to singular perturbation problems with multiple independent small parameters that should apply well beyond gene regulatory network models.
\end{abstract}

\keywords{Multi-parameter singular perturbations \and Gene regulatory networks \and Geometric singular perturbation theory \and Geometric blow-up}
\vspace{-0.3cm}
\classifications{34E10, 34E13, 34E15, 34E17, 37N25, 94C11}


\section{Introduction} \label{sec:Intro}
We present an in-depth singular perturbation analysis of a low-dimensional model for gene regulatory dynamics that was introduced by Plahte and Kjøglum in \cite{Plahte_GRN_2005}. The model is a toy problem which is intended to capture some of the main qualitative features of an important class of ODE models which date back to \cite{Glass1973,Thomas1973}, and which have since received a lot of attention, see e.g.~\cite{deJong2002,Edwards2015,Machina2013a,Machina2013b,Polynikis2009,Quee2021}. Plahte and Kjøglum's so-called \emph{fundamental problem} is a planar system of nonlinear ODEs given by
\begin{equation} \label{eq:fund_prob}
    \begin{aligned}
        \Dot{x}&=H(x,\eps_1,\theta_1)+H(y,\eps_2,\theta_2)-2 H(x,\eps_1,\theta_1) H(y,\eps_2,\theta_2)-\alpha_1 x, \\
        \Dot{y}&=1-H(x,\eps_1,\theta_1) H(y,\eps_2,\theta_2)- \alpha_2 y, 
    \end{aligned}
\end{equation}
where $(x,y) \in \R_+^2$, $\alpha_1,\, \alpha_2 \in \R_+$ are degredation parameters, and $H$ is a Hill function, i.e.
\begin{equation} \label{eq:Hill_function}
    H(x,\eps_1,\theta_1) = \frac{x^{1 / \eps_1}}{x^{1 / \eps_1}+\theta_1^{1 / \eps_1}}, \qquad  
    H(y,\eps_2,\theta_2) = \frac{y^{1 / \eps_2}}{y^{1 / \eps_2}+\theta_2^{1 / \eps_2}} .
\end{equation}
Here, $\theta_1, \, \theta_2 \in \R_+$ determine the switching locations.
The parameters $\eps_1, \eps_2$ are referred to as \textit{switching parameters} because they control the steepness of the first and second Hill function in \eqref{eq:Hill_function} respectively; see Figure \ref{fig:Hill_function}. For modelling purposes it is common to assume that $0 < \eps_1, \eps_2 \ll 1$, and subsequently to consider the singular limits $\eps_i \to 0$, $i=1,2$, in which the Hill functions converge pointwise to step functions:
\begin{equation} \label{eq:Hill_limit}
  \lim_{\eps_i \to 0}  H(z,\eps_i,\theta_i)= \begin{cases}
        1, & z > \theta_i, \\
        \frac{1}{2}, & z = \theta_i, \\
        0, & z < \theta_i.
    \end{cases}
\end{equation}
System \eqref{eq:fund_prob} with $0 < \eps_1, \eps_2 \ll 1$ has been analysed by a number of authors, e.g.~\cite{Simon_thesis,Machina2011,Edwards2015,Plahte_GRN_2005,Quee2021}, starting of course with \cite{Plahte_GRN_2005}. Our main focus in this work is to relax the `equal steepness assumption' that $\eps_1 = \eps_2$, which appears in all of these works.\footnote{Plahte and Kjøglum themselves considered the effects of varying steepnesses in system \eqref{eq:fund_prob} towards the end of \cite[Sec.~8.2]{Plahte_GRN_2005}, however, these considerations were kept brief in order to focus on other considerations.} In this work, we shall consider the dynamics of system \eqref{eq:fund_prob} over a neighbourhood of the origin in the (positive quadrant of) the $(\eps_1, \eps_2)$-plane. Our findings should be relevant to the study of larger networks to some extent, where the counterpart to this assumption takes the form $\eps_1 = \eps_2 = \cdots = \eps_n$ for all steepness parameters $\eps_i$. Our motivations for relaxing the equal-steepness assumption are two-fold:
\begin{figure}[t!]
    \centering
    \includegraphics[width=0.35\linewidth]{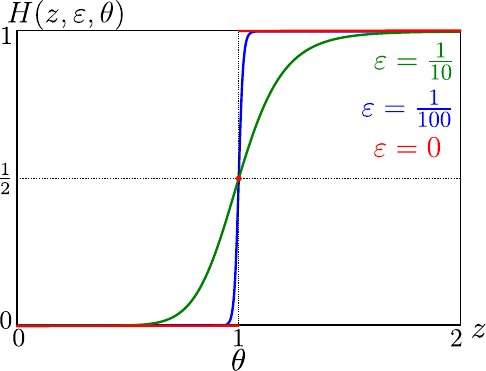}
    \caption{The Hill function $H(z,\eps,\theta)$ converges to a step function in the limit $\eps \to 0$, as described by \eqref{eq:Hill_limit}. Convergence is shown for decreasing values of $\eps = 10^{-1}$ (green), $\eps = 10^{-2}$ (blue) and $\eps \to 0$ (red).}
    \label{fig:Hill_function}
\end{figure}
\begin{itemize}
    \item Although it is very common in the literature, the equal steepness assumption is an analytical assumption that is generally not true in realistic models for GRN dynamics. Indeed, existing numerical investigations indicate non-trivial dependence of the qualitative dynamics on the steepness parameters; see e.g.~\cite{Lewis1991, Lewis1992, Killough2005, Wilds2009}. System \eqref{eq:fund_prob} presents a well-known case study which can be used to probe the extent to which the `symmetry' induced by this assumption biases our predictions about qualitative dynamics in GRN modelling.
    \item From a singular perturbations standpoint, system \eqref{eq:fund_prob} with $0 < \eps_1, \eps_2 \ll 1$ presents an interesting case study for the development of analytical tools for the study of systems with multiple independent small parameters. Our approach will be based on the so-called \textit{geometric blow-up method} \cite{dumortier1996canard,Krupa_2001_Extend,krupa_extending_transcritical}. The geometric blow-up method has already been used in GRN analyses, e.g.~for the fundamental problem with $\eps_1 = \eps_2$ and fixed parameters $\alpha_1, \alpha_2, \theta_1, \theta_2$ in \cite{Simon_thesis} (an MSc thesis supervised by P.~Szmolyan and precursor to this work), and for an activator-inhibitor model in \cite{Jelbart2026}.
\end{itemize}

Due to \eqref{eq:Hill_limit}, the fundamental problem \eqref{eq:fund_prob} is a piecewise-smooth (PWS) system in \textit{either limit} $\eps_1 \to 0$ or $\eps_2 \to 0$. In this sense, it is singularly perturbed with respect to both $\eps_1$ and $\eps_2$. 
Away from the sets
\begin{equation} \label{eq:switching_manifolds}
    \Sigma_1=\{(x,\theta_2) : x \geq 0 \}, \qquad 
    \Sigma_2=\{(\theta_1,y) : y \geq 0\},
\end{equation}
referred to as \textit{switching manifolds} and shown in yellow in Figure \ref{fig:PWS_phase_space}, the system is regularly perturbed. 
%
%
\begin{rem}
    Letting $(\eps_1,\eps_2) \to (0,0)$ in system \eqref{eq:fund_prob} leads to a PWS system with four isolated regions
    $$(A) :\,x > \theta_1,\, y > \theta_2, \qquad (B) :\, x< \theta_1,\, y > \theta_2, \qquad (C):\, x< \theta_1,\, y < \theta_2, \qquad (D):\, x > \theta_1,\, y < \theta_2,$$
    separated by the switching manifold $\Sigma_1 \cup \Sigma_2$ (see Figure \ref{fig:PWS_phase_space}). In all four regions (A)-(D), one obtains a decoupled linear system:
    \begin{equation} \label{eq:piecewise_smooth_systems}
    (A) :
    \begin{cases}
        \dot x = -\alpha_1 x, \\
        \Dot{y}= -\alpha_2 y,
    \end{cases}
    \quad 
    (B) : 
    \begin{cases}
        \dot x = 1 -\alpha_1 x, \\
        \Dot{y}= 1 -\alpha_2 y,
    \end{cases}
    \quad 
    (C) : 
    \begin{cases}
        \dot x = -\alpha_1 x, \\
        \Dot{y}= 1 - \alpha_2 y,
    \end{cases}
    \quad 
    (D) :
    \begin{cases}
        \dot x = 1 -\alpha_1 x, \\
        \Dot{y} = 1 -\alpha_2 y,
    \end{cases}
    \end{equation}
    each of which can be solved explicitly. To leading order, the dynamics of system \eqref{eq:fund_prob} on compact subsets of $\R^2_+ \setminus (\Sigma_1 \cup  \Sigma_2)$ is described by the decoupled linear systems above. 

    The fact that the limiting system \eqref{eq:piecewise_smooth_systems} is PWS has prompted a number of analyses based on Filippov theory. Our approach will be based in singular perturbation theory. In that sense, this work is complementary to e.g.~\cite{Ironi2011,Plahte_GRN_2005}, however, our approach will be considerably more geometric in flavour, and thereby more closely aligned with more recent analyses in e.g.~\cite{Jelbart2026,Simon_thesis}. We refer to \cite{Machina2013a,Machina2013b,Machina2011} for more on the relationship between singular perturbation and Filippov-based approaches in the context of GRN modelling.
\end{rem}

\begin{figure}[t!]
    \centering
    \includegraphics[width=0.4\linewidth]{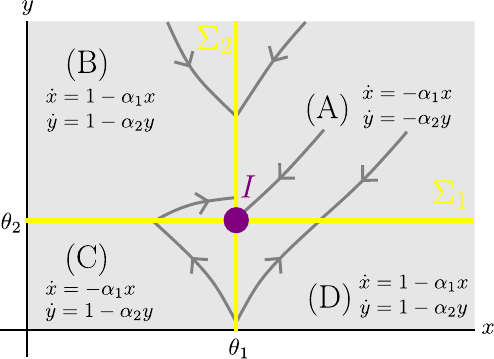}
    \caption{Sketch of the phase space associated with the piecewise-linear system \eqref{eq:piecewise_smooth_systems} obtained from system \eqref{eq:fund_prob} in the double singular limit $(\eps_1, \eps_2) \to (0,0)$. The switching manifolds $\Sigma_1, \Sigma_2$ are shown in yellow, and their intersection $I = \Sigma_1 \cap \Sigma_2$ is indicated in purple. Away from $\Sigma_1 \cup \Sigma_2$, the dynamics of system \eqref{eq:fund_prob} with $0 < \eps_1, \eps_2 \ll 1$ can be approximated using regular perturbation methods.}
    \label{fig:PWS_phase_space}
\end{figure}

A precise description of the leading order dynamics near $\Sigma_1 \cup \Sigma_2$ is considerably more involved, and it is this that will occupy us for the remainder of the manuscript. The dynamics near the intersection point $I : (\theta_1,\theta_2)$, shown in purple in Figure \ref{fig:PWS_phase_space}, is particularly degenerate in the limit $(\eps_1,\eps_2) \to (0,0)$. As we shall see, the qualitative dynamics of system \eqref{eq:fund_prob} close to $I$ depends sensitively on the parameters $\alpha_1, \alpha_2, \theta_1, \theta_2$ \textit{as well as} the relative size of $\eps_1$ and $\eps_2$.


Our main results describe the qualitative dynamics of system \eqref{eq:fund_prob} in an entire neighbourhood of $(0,0)$ in the positive quadrant of the $(\eps_1,\eps_2)$-plane, and in compact subsets of the two-parameter region defined by $0 < \kappa_1 < \kappa_2 < 1$, where
\begin{equation} \label{eq:product_parameters_kappa}
    \kappa_1:=\alpha_1\theta_1, \qquad 
    \kappa_2:=\alpha_2\theta_2 ,
\end{equation}
are considered as bifurcation parameters which replace $\alpha_1$ and $\alpha_2$ (we keep $\theta_1$ and $\theta_2$ fixed). The reasons for restricting to $0 < \kappa_1 < \kappa_2 < 1$ will become clearer as the analysis proceeds, see also Remark \ref{rem:parameter_restriction}. Three distinct two-parameter bifurcation sets are obtained, depending on the relative size of $\eps_1$ and $\eps_2$. This is a consequence of the fact that the small parameter space in problems involving more than one small parameter divides into different regions $B_i$, each of which corresponds to a different singular limit \cite{baumgartner_robertson_2025,Jelbart2026,Kuehn_double_limits_2022, Kristiansen_2024}. Our analysis reveals the existence of three such regions. 
The corresponding asymptotic regimes are
\[
B_1 : \eps_2 \ll \eps_1 \ll 1, \qquad 
B_2 : \eps_1 \sim \eps_2, \qquad 
B_3 : \eps_1 \ll \eps_2 \ll 1.
\]
In the MSc thesis \cite{Simon_thesis}, the geometric blow-up method was used to resolve the loss of smoothness which occurs in \eqref{eq:fund_prob} along the switching manifolds in the singular limit. The authors restricted to the special case with $\eps_1=\eps_2$ and fixed parameters $(\theta_1, \theta_2, \alpha_1, \alpha_2) = (1,1,0.6,0.9)$. 
This corresponds to one particular ray within the region $B_2$, and one particular choice of $(\kappa_1, \kappa_2, \theta_1, \theta_2)$. In this work, we build upon their ideas, and extend them in order to cover the much larger range of parameters specified above. Following their approach, we start by resolving the loss of smoothness by (i) blowing up the intersection point $I = \Sigma_1 \cap \Sigma_2$ to a sphere, and subsequently (ii) blowing up the four remaining branches of $\Sigma_1 \cup \Sigma_2 \setminus I$ to cylinders which are attached to the sphere. In the blown-up state space, solutions of \eqref{eq:fund_prob} extend uniquely to the blown-up switching manifolds and the global structure of trajectories becomes visible. Sliding motions along the switching manifolds are revealed as reduced flows along normally hyperbolic critical manifolds on the blow-up cylinders (these are identified via more classical scaling/boundary layer type arguments in \cite{Plahte_GRN_2005}). These critical manifolds perturb to Fenichel-type slow manifolds in the blown-up space, and blow down to locally invariant manifolds in system \eqref{eq:fund_prob}. Although the locally invariant manifolds obtained after blowing down are not slow manifolds in the sense of Fenichel \cite{Fenichel_1979}, they do inherit the usual invariance and normal attractivity/repulsivity properties from their Fenichel-type counterparts in the blown-up space. This is directly analogous to the identification of sliding dynamics via geometric blow-up in regularised PWS systems; see e.g.~\cite{Bonet2016,Buzzi2006,Kristiansen2019c,Llibre2009}.

While the invariant manifolds associated with sliding motions play an important role in our analysis, the bulk of this work will focus on the dynamics in a neighbourhood about $I = \Sigma_1 \cap \Sigma_2$, since it is here that the key bifurcations are unfolded. In the parameter region of interest, the system has either 1, 2 or 3 equilibria. One of these is a stable node which lies on an attracting critical manifold that is uncovered after one or two cylindrical blow-ups about the right-hand branch of $\Sigma_1$, depending on whether one works in $B_{2,3}$ or $B_1$ respectively. This equilibrium remains bounded away from $I$ in the parameter region of interest. The remaining equilibria, when they exist, lie in a neighbourhood about $I$, and their stability and type depend upon the bifurcation parameters $(\kappa_1,\kappa_2)$. We obtain three different two-parameter bifurcation sets in the above-specified region of the $(\kappa_1, \kappa_2)$-plane; one for each region $B_i$. 
The bifurcation set associated with small parameters $(\eps_1,\eps_2) \in B_1$ is the simplest; it consists only of a saddle-node curve correlated with the creation and destruction of the two equilibria close to $I$ under variation in $\kappa_1$ or $\kappa_2$. This saddle-node curve persists in $B_2$ and $B_3$, and has the same leading order parametrization in each case. We believe that this is a general feature of GRN models in the wider class of systems (we refer again to \cite{Glass1973,Thomas1973,Edwards2015,Machina2013b,Plahte_GRN_2005,Polynikis2009}), i.e.~we believe that saddle-node bifurcations are `robust to variation in steepness parameters'. Our results in $B_2$ and $B_3$ show that this is not the case for other bifurcation types. In $B_2$ we identify a regular Bogdanov-Takens bifurcation, along with the usual branches of (in this case subcritical) Hopf and homoclinic bifurcations. The most complicated dynamics occur in $B_3$. Here we show that the relevant system naturally features a \textit{regularised visible-invisible two-fold}, the normal form of which was recently treated in detail using geometric blow-up in \cite{Kristiansen_2023} (very similar unfoldings were considered in the earlier works \cite{Bonet2018,Kristiansen_2015}). Among other things, the results in \cite{Kristiansen_2023} allow us to prove the existence of canard cycles in a narrow region of the $(\kappa_1, \kappa_2)$-plane. Following this, we proceed to identify a singular Bogdanov-Takens bifurcation that occurs as an equilibrium traverses the neighbourhood associated with the visible-invisible two-fold singularity under parameter variation. The associated codimension-two bifurcation is `singular' in the sense that it requires several additional blow-ups and an associated non-trivial scaling in order to identify it. We provide a partial description of the associated Hopf and homoclinic curves, and outline the incomplete canard explosion that we expect to be associated with the identified geometric structures. 

Despite the complicated technicalities, which we interpret as a necessary consequence of the presence of multiple singular perturbation parameters, we would like to emphasize the (relatively) constructive approach adopted herein. The analytical approach, which we believe can be adapted to singular perturbation problems with multiple small parameters more generally, can be summarized as follows:
\begin{itemize}
    \item[1.] Identify the regions $B_i$ in the small parameter space which correspond to distinct singular limits (here $B_1$, $B_2$ and $B_3$);
    \item[2.] For each region $B_i$, determine the bifurcation set in the specified bifurcation parameters. Use local coordinates associated with a blow-up in parameters if necessary;
    \item[3.] For each region in each bifurcation set, determine the geometry and dynamics in the (possibly blown-up) phase space.
\end{itemize}
We refer to \cite{baumgartner_robertson_2025} for a recent illustration of this approach in the context of the Robertson model, a chemical reaction and well-known benchmark problem for numerical solvers of stiff ODEs, which skips over Step 2 due to the fact that the bifurcations were not considered. The approach is also complementary to the recent analysis in \cite{Jelbart2026}, in which geometric blow-up was used to investigate the validity of quasi-steady state reductions in larger GRN models which involve equations which govern mRNA translation; yet another problem stemming from the presence of multiple small parameters in GRN modelling. Overall, we are hopeful that this approach will be valuable for a wide range of singular perturbation problems featuring multiple small parameters.

The remainder of the manuscript is structured as follows. We start with an extensive presentation of our main results and geometrical constructions in Section \ref{sec:Main_results}, which aims to provide a clear geometric picture early on.
In Section \ref{sub:main_results_bifurcation_diagrams} we present and describe our main results, i.e.~the limiting two-parameter bifurcation sets associated with each region $B_i$, $i=1,2,3$ in the $(\eps_1, \eps_2)$-plane. In Section \ref{sub:main_results_blow-up} we provide an informal description of the blow-up construction that is used in later sections to prove these results, along with a description of the global geometry and dynamics in each case. In Sections \ref{sec:Region_B2}, \ref{sec:Region_B1} and \ref{sec:Region_B3} we present the detailed geometric blow-up analyses and prove our main results in the three cases corresponding to regions $B_2$, $B_1$ and $B_3$ respectively. We conclude with a summary and outlook in Section \ref{sec:Outlook}. 

\section{Main results} \label{sec:Main_results}


Our main results describe the dynamics and bifurcation structure of the fundamental problem \eqref{eq:fund_prob} with \eqref{eq:product_parameters_kappa}, i.e.~of the system
\begin{equation} \label{eq:fund_prob_kappa}
    \begin{aligned}
        \Dot{x}&=H(x,\eps_1,\theta_1)+H(y,\eps_2,\theta_2)-2 H(x,\eps_1,\theta_1) H(y,\eps_2,\theta_2)-\kappa_1 \theta_1^{-1} x, \\
        \Dot{y}&=1-H(x,\eps_1,\theta_1) H(y,\eps_2,\theta_2)- \kappa_2 \theta_2^{-1} y.
    \end{aligned}
\end{equation}
The parameters $\theta_1, \theta_2 > 0$ are fixed, $\kappa_1, \kappa_2 > 0$ are bifurcation parameters, and $(\eps_1, \eps_2) \in \mathcal U := B((0,0), \delta) \cap \{ \eps_1, \eps_2 \geq 0\}$ for some $\delta>0$. We further decompose the neighborhood $\mathcal U$ into three different regions, which we denote by $B_1$, $B_2$ and $B_3$, separated by the lines
\begin{equation}
    \label{eq:C1C2}
    C_1: \eps_2=\beta_1 \eps_1, \qquad C_2: \eps_2=\beta_2 \eps_1,
\end{equation}
see Figure \ref{fig:scaling_regions}. The regions $B_i$ can be written as
\begin{equation}
\label{eq:B_i}
    \begin{split}
    B_1 &= \{ (\eps_1, \eps_2) \in \mathcal U : \eps_2 < \beta_1 \eps_1 \} , \\
    B_2 &= \{ (\eps_1, \eps_2) \in \mathcal U : \beta_1 \eps_1 \leq \eps_2 \leq \beta_2 \eps_1 \} , \\
    B_3 &= \{ (\eps_1, \eps_2) \in \mathcal U : \eps_1 < \beta_2^{-1} \eps_2 \} . 
    \end{split}
\end{equation}
The constants $\beta_2 > \beta_1 > 0$ are independent of $(\eps_1,\eps_2)$, but $\beta_1$ ($\beta_2$) will need to be fixed sufficiently small (large) for the validity of our results.

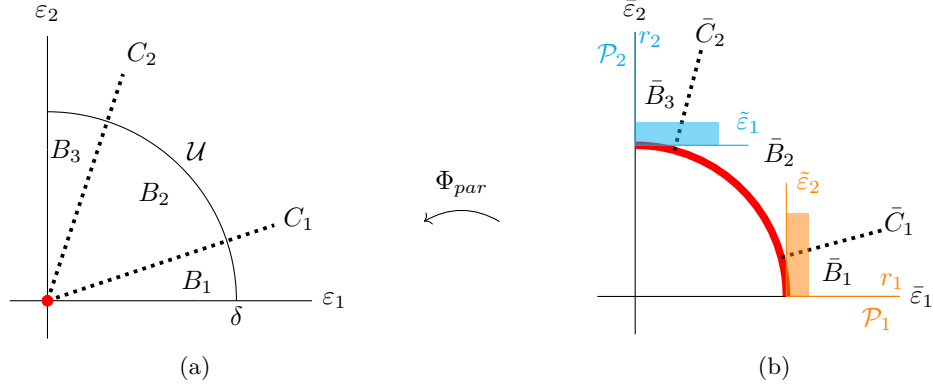
\begin{figure}[t!]
\centering
\begin{subfigure}{0.3\textwidth}
    \begin{tikzpicture}[scale=1]
			
        \draw[] (0,-0.5) -- (0,2);
        \draw[] (0,2) -- (0,3.5);
        \draw[] (-0.5,0) -- (2,0);
        \draw[] (2,0) -- (3.5,0);
        
        \draw[scale=1, domain=0:3, smooth, variable=\x, line width=0.5mm, dotted] plot ({\x}, {(1/3)*\x});
        \draw[scale=1, domain=0:1, smooth, variable=\x, line width=0.5mm, dotted] plot ({\x}, {(3)*\x});
        \draw (83:2) node{$\textcolor{black}{B_3}$};
        \draw (45:2) node{$\textcolor{black}{B_2}$};
        \draw (7:2) node{$\textcolor{black}{B_1}$};
        \draw (45:2.8) node{$\textcolor{black}{\mathcal U}$};
        \draw (18:3.5) node{$C_1$};
        \draw (69:3.5) node{$C_2$};
        \draw[] (2.5, 0) arc (0:90:2.5);
        \draw (3.8,0) node{$\varepsilon_1$};
        \draw (2.5,-0.2) node{$\delta$};
        \draw (0,3.8) node{$\varepsilon_2$};
        \filldraw[red] (0,0) circle (2pt);

			\end{tikzpicture}
            \caption{}
            \label{fig:scaling_regions}
\end{subfigure}
\begin{subfigure}{0.15\textwidth}
\begin{tikzpicture}
       \path[->]
    (1.3,2) edge[bend right] node [left] {} (0.3,2);
    \draw (0.8,2.5) node{$\Phi_{par}$};
    \draw (0,0) node{};
\end{tikzpicture}
\end{subfigure}
\begin{subfigure}{0.3\textwidth}
    \begin{tikzpicture}[scale=1]
        \draw[] (0,-0.5) -- (0,3.5);
        \draw[] (-0.5,0) -- (3.5,0);
        
        \draw[red,line width=1mm] (2, 0) arc (0:90:2);

        \draw[orange] (2,0) -- (2,1.5);
        \draw[orange] (2,0) -- (3.5,0);

        \draw[cyan] (0,2) -- (0,3.5);
        \draw[cyan] (0,2) -- (1.5,2);
        
            (0:2) -- (0:3) -- (0:3) arc(0:15:3)-- (15:2) --cycle;
            (15:2) -- (15:3) -- (15:3) arc(15:25:3)-- (25:2) --cycle;
            (25:2) -- (25:3) -- (25:3) arc(25:35:3)-- (35:2) --cycle;
            (35:2) -- (35:3) -- (35:3) arc(35:45:3)-- (45:2) --cycle;
            (45:2) -- (45:3) -- (45:3) arc(45:55:3)-- (55:2) --cycle;
            (55:2) -- (55:3) -- (55:3) arc(55:65:3)-- (65:2) --cycle;
            (65:2) -- (65:3) -- (65:3) arc(65:75:3)-- (75:2) --cycle;
            (75:2) -- (75:3) -- (75:3) arc(75:90:3)-- (90:2) --cycle;
            \draw[line width=0.5mm,dotted] (15:2) -- (15:3.4);
            \draw[line width=0.5mm,dotted] (75:2) -- (75:3.4);
        \draw (83:2.7) node{$\Bar{B}_3$};
        \draw (45:2.7) node{$\Bar{B}_2$};
        \draw (7:2.7) node{$\Bar{B}_1$};
        \draw (3.8,0) node{$\Bar{\varepsilon}_1$};
        \draw (0,3.8) node{$\Bar{\varepsilon}_2$};
        \draw[orange,line width=1mm] (3.2,-0.3) node{$\mathcal{P}_1$};
        \draw[orange,line width=1mm] (3.4,0.2) node{$r_1$};
        \draw[orange,line width=1mm] (2.3,1.5) node{$\Tilde{\eps}_2$};
        \fill [orange,opacity=0.5] (2,0) rectangle (2.3,1.1);
        \draw[cyan,line width=1mm] (-0.3,3.2) node{$\mathcal{P}_2$};
        \draw[cyan,line width=1mm] (0.2,3.4) node{$r_2$};
        \draw[cyan,line width=1mm] (1.5,2.3) node{$\Tilde{\eps}_1$};
        \fill [cyan,opacity=0.5] (0,2) rectangle (1.1,2.3);
        \draw[] (3.5,1) node{$\Bar{C}_1$};
        \draw[] (1,3.5) node{$\Bar{C}_2$};

			\end{tikzpicture}
            \caption{}
\end{subfigure}
			\caption{(a): The regions $B_1$, $B_2$ and $B_3$ defined in \eqref{eq:B_i}, corresponding to the scaling regimes $\eps_2 \ll \eps_1$, $\eps_1 \approx \eps_2$, $\eps_1\ll \eps_2$, respectively. (b): Parameter blow-up $\Phi_{par}$ of $(\eps_1,\eps_2)=(0,0)$, as defined in \eqref{eq:parameter_blow-up}, which blows the origin in the figure on the left up to the quarter-circle shown in red. The parameter charts $\mathcal{P}_1$, $\mathcal{P}_2$ and their local coordinate axes are shown in orange, blue respectively, and the pre-images $\bar B_i$ and $\bar C_i$ of the regions $B_i$ and curves $C_i$ respectively are also indicated. We also illustrate the rectangular geometry of the domains $B_1$ and $B_3$ in local coordinates $\mathcal P_1$ and $\mathcal P_2$ in shaded orange and blue respectively.} 
			\label{fig:parameter_blow-up}
\end{figure}

\begin{rem}
\label{rem:overlapping}
    As we have defined them above, the regions $B_i$, $i=1,2,3$, are non-intersecting for any fixed choice of $\beta_1$ and $\beta_2$. We have chosen to define them this way for clarity, however, it is important for the analysis presented in later sections that $B_2$ can be enlarged (by varying the constants $\beta_1$ and  $\beta_2$) so that $B_1 \cap B_2 \neq \emptyset$ and $B_2 \cap B_3 \neq \emptyset$.
\end{rem}

We present results on the following:
\begin{enumerate}
    \item[(i)] Two-parameter bifurcation sets on the triangular region of the $(\kappa_1, \kappa_2)$-plane defined by
    \[
    \Lambda := \{(\kappa_1,\kappa_2) \in \R^2:0<\kappa_1<\kappa_2<1\} .
    \]
    We obtain a distinct bifurcation set for each region $B_i$, $i=1,2,3$.
    \item[(ii)] The structure of the phase space after geometric blow-up and desingularisation in each region $B_i$, $i=1,2,3$.
\end{enumerate}
We begin with the former.

\begin{rem} \label{rem:parameter_restriction}
    The boundaries of the parameter set $\Lambda$ are correlated with important changes in the overall geometry and dynamics of system \eqref{eq:fund_prob_kappa}. In particular:
    \begin{enumerate}
        \item[(i)] For $\kappa_2>1$ there exists an equilibrium in $\R^2_+ \setminus (\Sigma_1 \cup \Sigma_2)$, which lies in a compact subset of region $(D)$; see Figure \ref{fig:PWS_phase_space}.
        \item[(ii)] For $\kappa_1>\kappa_2$ there is an equilibrium close to $\Sigma_2$ (instead of $\Sigma_1$ as in the situation described by Proposition \ref{prop:node} below).
    \end{enumerate}
    We expect the details associated with the variation of parameters over either one of these boundaries to be quite involved, and do not consider either case further herein.
\end{rem}

			
        


\subsection{Two-parameter bifurcation sets}
\label{sub:main_results_bifurcation_diagrams}

In order to formulate the results, it is helpful to blow the singular limit $(\eps_1,\eps_2)=(0,0)$ up to a circle via 
\begin{equation} \label{eq:parameter_blow-up}
 \begin{aligned}
         \Phi_{par}: [0,\infty) \times \mathbb{S}^1 \to \R^2, \quad 
         (r,\Bar{\eps}_1,\Bar{\eps}_2) \mapsto \begin{cases}
             \eps_1=r\Bar{\eps}_1, \\
             \eps_2=r\Bar{\eps}_2,
         \end{cases}
 \end{aligned}
\end{equation}
where we restrict ourselves to the meaningful parameter range $\Bar{\eps}_1,\Bar{\eps}_2\geq 0$. The preimage of the origin under $\Phi_{par}$ is the quarter circle ($r=0$, $\bar \eps_1 \geq 0$ and $\bar \eps_2 \geq 0$). 

\begin{rem}
    As with blow-up transformations in the variable/state space, the particular form of the transformation in \eqref{eq:parameter_blow-up} is determined based on scaling arguments. The homogeneous blow-up defined by \eqref{eq:parameter_blow-up} is the `correct' choice for our particular problem, but weighted blow-ups will be more suitable for other problems. Indeed, weighted blow-ups in parameters will be necessary for the analysis in Section \ref{sec:Region_B3} below. See also \cite{baumgartner_robertson_2025} for a recent application where a weighted blow-up in parameters has facilitated the analytic investigation of the Robertson model, a prototypical example of a multi-scale chemical reaction.
\end{rem}

Our analysis will be simplified significantly by working in directional charts $\mathcal P_1$ and $\mathcal P_2$, formally obtained by setting $\Bar{\eps}_1=1$ and $\bar \eps_2 = 1$ respectively. Denote the local representation of the blow-up transformation in each chart by
\[
\mathcal{P}_1: \eps_1 = r_1, \ \eps_2 = r_1 \Tilde{\eps}_2, \qquad 
\mathcal{P}_2: \eps_1 = r_2 \Tilde{\eps}_1, \ \eps_2 = r_2 ,
\]
where $r_i, \Tilde{\eps}_i, \geq 0$ for $i \in \{1,2\}$. In order to minimise confusion over subscripts, however, we will avoid the $r_i$ notation and write $\eps_2 = \eps_1 \tilde \eps_2$ in chart $\mathcal P_1$ and $\eps_1 = \eps_2 \tilde \eps_1$ in chart $\mathcal P_2$. 
Note that under \eqref{eq:parameter_blow-up}, 
the curves $C_i$ defined in \eqref{eq:C1C2} are mapped onto the straight lines defined via
$$\Bar{C}_1: \Tilde{\eps}_2=\beta_1 \, 
\qquad  
\Bar{C}_2: 
\Tilde \eps_1 = \beta_2^{-1} , $$ 
respectively. It follows that regions $B_1$ and $B_3$ are \textit{rectangular} domains in $\mathcal P_1$ and $\mathcal P_2$ respectively, given by $B_1 = \{ (\eps_1,\tilde \eps_2) : \eps_1 \in [0,\delta], \tilde \eps_2 \in [0, \beta_1] \}$ and $B_3 = \{ (\tilde \eps_1, \eps_2) : \tilde \eps_1 \in [0, \beta_2^{-1}], \eps_2 \in [0,\delta] \}$. The intermediate region $B_2$ can be made to overlap with both $B_1$ and $B_3$ (cf.~Remark \ref{rem:overlapping}), and can be formulated in terms of either $\mathcal P_1$ or $\mathcal P_2$ coordinates. 
A schematic representation of this is shown in Figure \ref{fig:parameter_blow-up}.

Before we describe the bifurcation structure in each region $B_i$, we state a preliminary result which holds for all $(\eps_1, \eps_2) \in \mathcal U$. We state it in terms of the radial variable $r \geq 0$ associated with the parameter blow-up \eqref{eq:parameter_blow-up}.

\begin{prop}
\label{prop:node}
    Fix $(\kappa_1, \kappa_2) \in \Lambda$. Then there exists a $\delta > 0$ such that for all $(\eps_1, \eps_2) \in \textup{Int}(\mathcal U)$, system \eqref{eq:fund_prob_kappa} has a stable node $q_{n} : (x_{n}, y_{n})$ which is $\mathcal O(r)$-close to $\Sigma_1$ and bounded to right of $I = \Sigma_1 \cap \Sigma_2$. In particular, $(x_{n}, y_{n}) \to (\theta_1 \kappa_2 / \kappa_1, \theta_2)$ as $r \to 0$. 
\end{prop}

Proposition \ref{prop:node} asserts the existence of a stable node $q_n$ close to the set $\Sigma_1$, which remains (i) a stable node, and (ii) bounded to the right of the degenerate point $I$ for all $(\kappa_1,\kappa_2) \in \Lambda$ and $(\eps_1, \eps_2) \in \textup{Int}(\mathcal U)$ (assuming of course that $\delta > 0$ and therefore $\mathcal U$ is sufficiently small). Note that the equilibrium $q_n$ does not appear in the PWS system \eqref{eq:piecewise_smooth_systems}. The proof will follow from the analyses in Sections \ref{sec:Region_B2}, \ref{sec:Region_B1} and \ref{sec:Region_B3}, which cover the geometry and dynamics separately in $B_2$, $B_1$ and $B_3$ respectively. It will follow in a straightforward way from `local' versions of the result in each $B_i$.


\subsubsection{Bifurcation set in $B_1$}

Unlike Proposition \ref{prop:node}, the remainder of the results presented below depend upon the region $B_i$. The first of these describes the two-parameter bifurcation set in $B_1$. We state it for system \eqref{eq:fund_prob_kappa} in parameter chart $\mathcal P_1$, i.e.~for the system
\begin{equation}
    \label{eq:fund_prob_P1}
    \begin{split}
    \Dot{x}&=H(x,\eps_1,\theta_1)+H(y,\eps_1 \tilde \eps_2,\theta_2)-2 H(x,\eps_1,\theta_1) H(y,\eps_1 \tilde \eps_2,\theta_2)-\kappa_1 \theta_1^{-1} x , \\
        \Dot{y}&=1-H(x,\eps_1,\theta_1) H(y,\eps_1 \tilde \eps_2,\theta_2)- \kappa_2 \theta_2^{-1} y ,
    \end{split}
\end{equation}
with $0 < \eps_1, \tilde \eps_2 \ll 1$.

\begin{minipage}{\textwidth}
\begin{thm}
\label{thm:B1}
    \textup{(Regular saddle-node bifurcation in $B_1$)}
    Fix a compact interval $\mathcal I \subset (0,1)$. There exists $\eps_{1,0} > 0$, $\Tilde{\eps}_{2,0}>0$ and a $C^1$-function $\kappa_{1,sn} : \mathcal I \times [0,\eps_{1,0}) \times [0,\Tilde{\eps}_{2,0}) \to \R$ such that 
    system \eqref{eq:fund_prob_P1} undergoes a saddle-node bifurcation when $\kappa_1=\kappa_{1,sn}(\kappa_2,\eps_1,\Tilde{\eps}_2)$. In particular,
    \begin{equation}
    \label{eq:sn_curve_B1}
        \kappa_{1,sn}(\kappa_2,0,0)=2\kappa_2-2+2\sqrt{1-\kappa_2} 
    \end{equation}
    for all $\kappa_2 \in \mathcal I$.
\end{thm}
\end{minipage}

\begin{figure}[t!]
    \centering
    \includegraphics[width=0.4\linewidth]{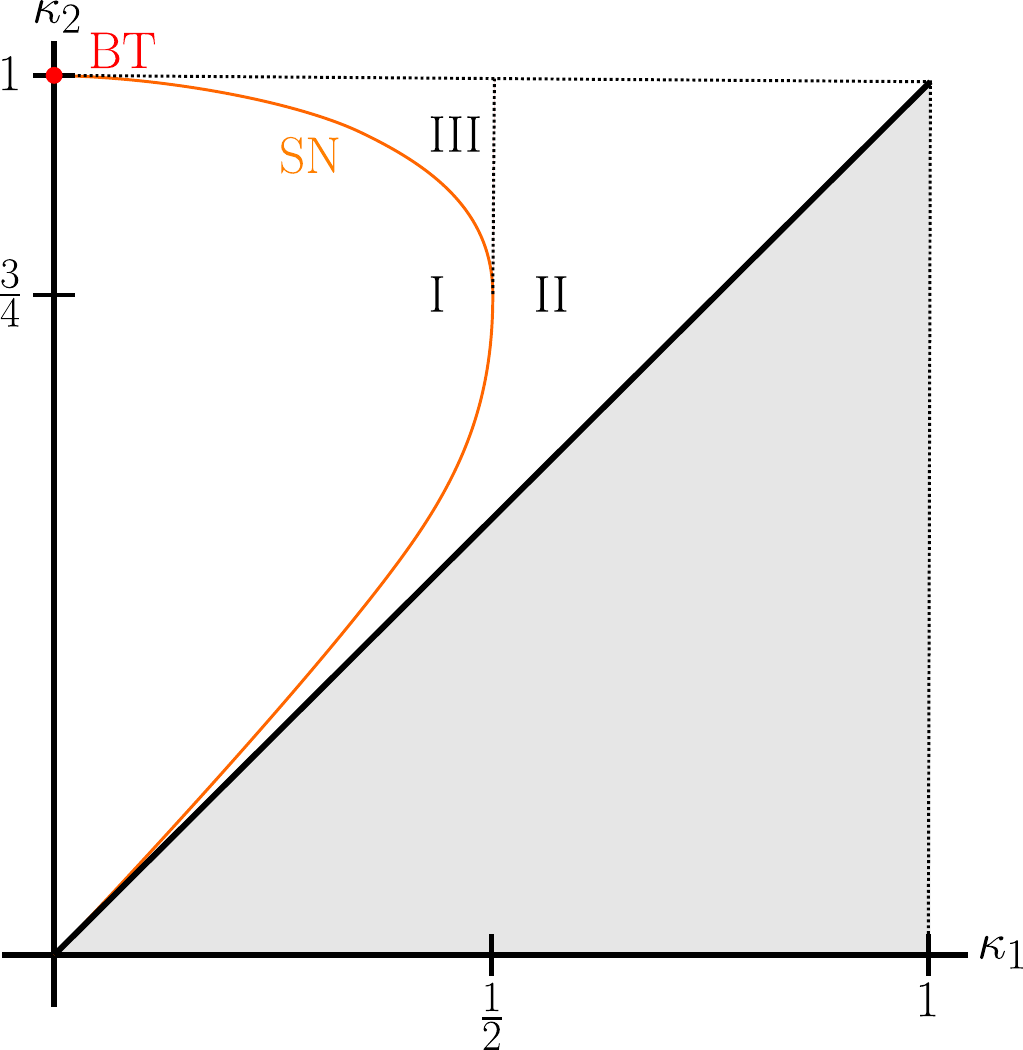}
    \caption{Limiting bifurcation set for system \eqref{eq:fund_prob_P1} in $B_1$, as described by Theorem \ref{thm:B1}. The (limiting) saddle-node curve, shown in orange and denoted by SN, is given by \eqref{eq:sn_curve_B1}. It separates two different regions, denoted I and II$\,\cup\,$III, for which the limiting system associated with system \eqref{eq:fund_prob_P1} has one vs three equilibria respectively. In particular, the system is at least bistable in regions II and III. The grey shaded region below the line $\kappa_1 = \kappa_2$ is not covered by our analysis, and the red point labeled BT at $(\kappa_1, \kappa_2) = (0,1)$ lies on the boundary of the parameter region $\Lambda$ (see Remark \ref{rem:BT_point} below for more on the significance of this point). The boundary between II and III (vertical dotted-black line) is expected to be important for the global dynamics; see Remark \ref{rem:heteroclinic}.}
    \label{fig:B1_bifurcation_diagram}
\end{figure}

The two-parameter bifurcation set in the $B_1$ singular limit $(\eps_1, \tilde \eps_2) = (0, 0)$ is shown in Figure \ref{fig:B1_bifurcation_diagram}. The saddle-node curve (SN, orange) divides the relevant region of the $(\kappa_1,\kappa_2)$-plane into two distinct regions, which we denote here by I and II$\,\cup\,$III, corresponding to  $\kappa_1<\kappa_{1,sn}(\kappa_2,0,0)$ and $\kappa_1>\kappa_{1,sn}(\kappa_2,0,0)$, respectively. In region I, the stable node $q_n$ is the only equilibrium. It appears to be a global attractor. In region III there are three equilibria; two stable equilibria and a saddle. Together, Propositions \ref{prop:node} and \ref{thm:B1} imply that system \eqref{eq:fund_prob_P1} features (at least) bistability for $(\kappa_1, \kappa_2)$-values in region II and III as $(\eps_1, \eps_2) \to (0,0)$ from within $B_1$. 
The saddle-node bifurcation described by Theorem \ref{thm:B1} 
occurs as two equilibria coalesce in a neighbourhood about the point $I$ which shrinks to zero as $(\eps_1, \tilde \eps_2) \to (0,0)$. The proof, which is presented in Section \ref{sec:B1_bifurcation_analysis}, follows from the identification of a \textit{regular} saddle-node bifurcation in local coordinates $(x_2, y_{22})$ defined via
\[
x = \theta_1 e^{\eps_1 x_2} , \qquad
y = \theta_2 e^{\eps_1 \tilde \eps_2 y_{22}} ,
\]
i.e.~when $x = \theta_1 + \mathcal O(\eps_1)$ and $y = \theta_2 + \mathcal O(\eps_2)$ as $\eps_1 \to 0$ and $\eps_2 \to 0$ respectively (we used $\eps_2 = \eps_1 \tilde \eps_2$ to infer the latter). These coordinates are obtained as part of a geometric blow-up construction involving multiple blow-ups; an overview of this procedure will be presented in Section \ref{sub:main_results_blow-up} below. The limiting bifurcation set in Figure \ref{fig:B1_bifurcation_diagram} is obtained when $\eps_1 = \tilde \eps_2 = 0$ in the rescaled system associated with these coordinates.

\begin{rem}
\label{rem:heteroclinic}
    The dotted black line $\{(1/2,\kappa_2),\, \kappa_2 \in (3/4,1)\}$ which separates regions II and III in Figure \ref{fig:B1_bifurcation_diagram} is correlated with a heteroclinic connection in the blown-up state space which is described in detail in Sections \ref{sub:main_results_blow-up} and \ref{sec:Region_B1} below. Similarly to canard solutions in slow-fast systems \cite{dumortier1996canard,Krupa_2001_Extend,Maesschalck_2021}, this connection does not persist as a bifurcation when $0 < \eps_1, \eps_2 \ll 1$ per se, however, it does act like a separatrix which organises the basins of attraction in an important way; we refer to Remark \ref{rem:main_global_dynamics} below for more details.
\end{rem}

\subsubsection{Bifurcation set in $B_2$}

The results in $B_2$ can be formulated in either $\mathcal P_1$ or $\mathcal P_2$. In the following we present our results in $\mathcal P_1$, i.e.~for system \eqref{eq:fund_prob_P1}, but now with $0 < \eps_1 \ll 1$ and fixed $\tilde \eps_2 > 0$.
For notational purposes we introduce the quantity
\begin{equation}
\label{eq:x_B}
    z_{BT}:=\frac{\tilde \eps_2 \theta_2}{\theta_1 + 2 \tilde \eps_2 \theta_2 }.
\end{equation}
The following result describes the bifurcation set in $B_2$.

\begin{thm}
\label{thm:B2}
\textup{(Regular Bogdanov-Takens bifurcation in $B_2$)}
 Fix $\beta_2 > \beta_1 > 0$ and $\tilde \eps_2 \in [\beta_1, \beta_2]$. 
 There exists an $\eps_{1,0} > 0$ and a $C^1$-function $(\kappa_{1,BT}, \kappa_{2,BT}) : [0, \eps_{1,0}) \to \R^2$ such that for all $\eps_1 \in (0,\eps_{1,0})$, system \eqref{eq:fund_prob_P1} undergoes a Bogdanov-Takens bifurcation at $(\kappa_1, \kappa_2) = (\kappa_{1,BT}, \kappa_{2,BT})(\eps_1)$. The leading order asymptotics are given by
 \begin{equation}
 \label{eq:BT_pars_B2}
     (\kappa_{1,BT},\kappa_{2,BT})(0)=( 2 z_{BT} (1-z_{BT}),  1-z_{BT}^2) .
 \end{equation}
 For any fixed compact interval $\mathcal I \subset (0,1)$, $\eps_{1,0} > 0$ can be chosen sufficiently small to ensure that the associated saddle-node bifurcation occurs along $\kappa_1 = \kappa_{1,sn}(\kappa_2,\eps_1)$, where $\kappa_{1,sn}: \mathcal I \times [0,\eps_{1,0}) \to \R$ is a $C^1$-function which satisfies
 \begin{equation}
 \label{eq:sn_curve_B2}
    \kappa_{1,sn}(\kappa_2,0)=2\kappa_2-2+2\sqrt{1-\kappa_2}
 \end{equation}
 for all $\kappa_2 \in \mathcal I$.
  %
\end{thm}

\begin{figure}[t!]
    \centering
    \includegraphics[width=0.4\linewidth]{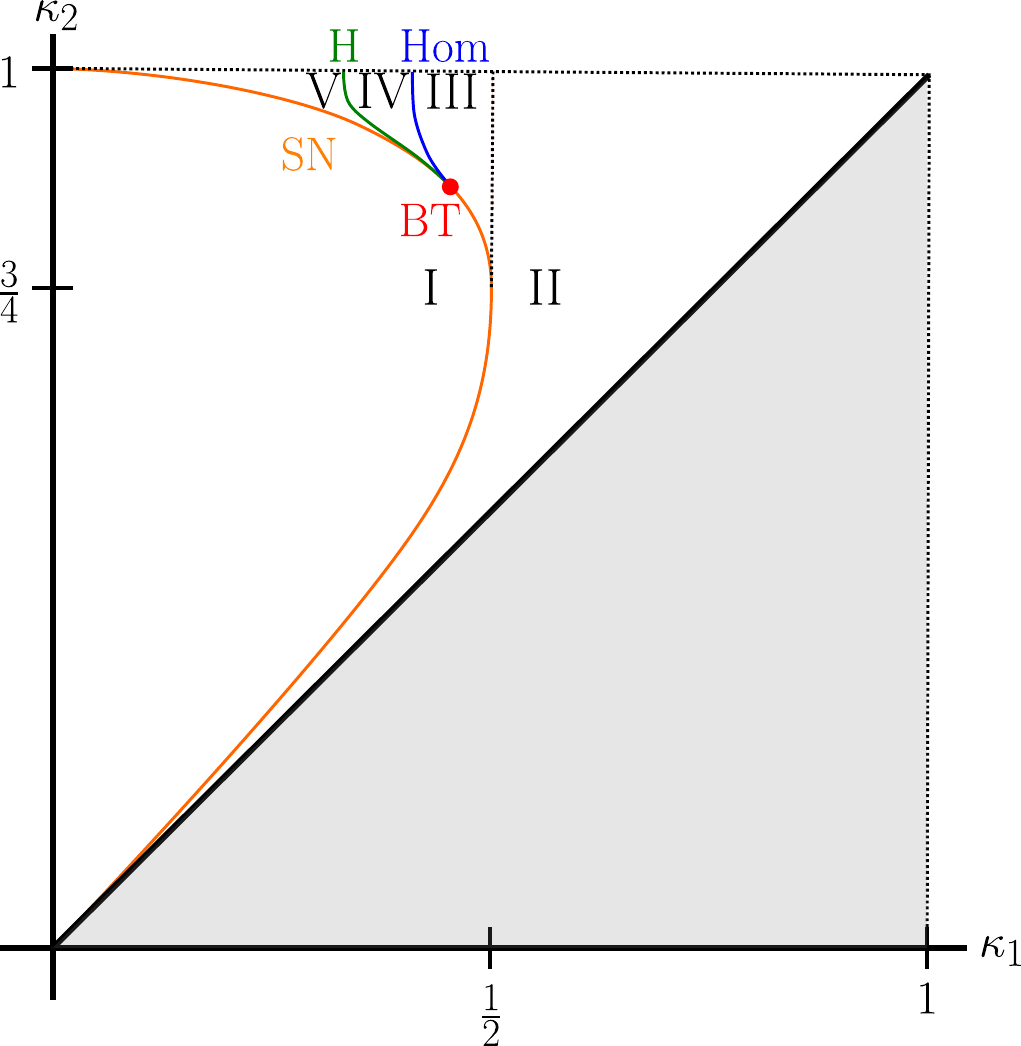}
    \caption{Limiting bifurcation set for system \eqref{eq:fund_prob_P1} in $B_2$, as described by Theorem \ref{thm:B2}. We use the same labelling and coloring conventions as in Figure \ref{fig:B1_bifurcation_diagram} to indicate the saddle-node curve and the lower triangular region not covered by our analysis. In $B_2$, we identify a Bogdanov-Takens point (red point), as well as an associated homoclinic and subcritical Hopf curve, shown here in blue and green respectively. Five open regions are identified: I, II, III, IV and V. The qualitative dynamics in regions I, II and III is similar to $B_1$ (cf.~Figure \ref{fig:B1_bifurcation_diagram}). The limiting system associated with system \eqref{eq:fund_prob_P1} has three equilibria in region V, but only one of these is stable (as opposed to the bistable situation in regions II and III). There are two stable equilibria in region IV, one of which is surrounded by an unstable cycle which is born in the subcritical Hopf bifurcation and eventually terminates in the homoclinic bifurcation as one traverses a path from region V to IV, to III. As in $B_1$, the boundary between II and III is expected to be important for the global dynamics.}
    \label{fig:B2_bifurcation_diagram}
\end{figure}

Theorem \ref{thm:B2} is proven in Section \ref{sec:B2_bifurcation_analysis}. The two-parameter bifurcation diagram when $\eps_1 = 0$ is shown in Figure \ref{fig:B2_bifurcation_diagram}. This time, the subset $\Lambda$ of the $(\kappa_1, \kappa_2)$-plane is divided into five different regions: I, II, III, IV and V. Theorem \ref{thm:B2} provides a global parameterisation of the saddle-node curve (SN, orange). It does not provide a corresponding parameterisation for the Hopf curve (H, green), however, an implicit formula for this curve will be given later in Remark \ref{rem:implicit_hopf}.
Numerical computations in MatCont \cite{MATCONT} indicate that the Hopf bifurcation is subcritical. Our proof of Theorem \ref{thm:B2} will show that this must be so near the Bogdanov-Takens point, however, we do not prove the subcriticality of the Hopf bifurcations along the entire curve in Figure \ref{fig:B2_bifurcation_diagram}. We also sketch the (numerically computed) homoclinic curve (Hom, blue) which is guaranteed to exist and emanate from the BT point, tangent to and to the right of the Hopf curve; see e.g.~\cite[Ch.~8.4]{Kuznetsov_1998}). 

Similarly to the results in $B_1$ that we presented in Theorem \ref{thm:B1}, the bifurcations described by Theorem \ref{thm:B2} -- including the `global' homoclinic bifurcation close to the BT point -- occur in a neighbourhood of the point $I$ in phase space which shrinks to zero as $\eps_1 \to 0$. This time, the relevant bifurcations are obtained as regular bifurcations in local coordinates $(x_2,y_2)$ defined by
\[
x = \theta_1 e^{\eps_1 x_2} , \qquad 
y = \theta_2 e^{\eps_1 y_2} .
\]
This means that $x = \theta_1 + \mathcal O(\eps_1)$ and $y = \theta_2 + \mathcal O(\eps_1)$ (which is equivalent to $x = \theta_1 + \mathcal O(\eps_2)$ and $y = \theta_2 + \mathcal O(\eps_2)$ in $B_2$ since $\eps_1 \sim \eps_2$ here). Because the equilibrium $q_{n}$ is bounded away from this region, the qualitative dynamics as $\eps_1 \to 0$ on compact regions of the $(x_2, y_2)$-plane is primarily organised by the normal form unfolding of the regular (subcritical) BT point; we refer again to \cite{Kuznetsov_1998} for details. Here we simply highlight that system \eqref{eq:fund_prob_P1} is (at least) bistable in regions II, III and IV, and that it exhibits unstable oscillations in IV. These are $\mathcal O(1)$ open regions in the $(\kappa_1,\kappa_2)$-plane, despite the localisation of bifurcation phenomena to a neighbourhood in phase space which shrinks to $I$ as $\eps_1 \to 0$. The proof of Theorem \ref{thm:B2} is given in Section \ref{sec:B2_bifurcation_analysis}, and these particular coordinates are identified with a coordinate chart that arises during a geometric blow-up construction which we outline in Section \ref{sub:main_results_blow-up}. The limiting bifurcation set shown in Figure \ref{fig:B2_bifurcation_diagram} is obtained when $\eps_1 = 0$ in the rescaled system associated with these coordinates.

\begin{rem}
    \label{rem:BT_point}
    The location of the Bogdanov-Takens bifurcation in $(\kappa_1,\kappa_2)$-space depends on the relative size of the small parameters via the dependence of $z_{BT}$ on $\tilde \eps_2 = \eps_2 / \eps_1$; recall \eqref{eq:x_B}. Since $z_{BT} \to 0$ as $\tilde \eps_2 \to 0$, $(\kappa_{1,BT}, \kappa_{2,BT}) \to (0, 1) \in \overline \Lambda \setminus \Lambda$. Thus, the Bogdanov-Takens point `leaves' the parameter set $\Lambda$ in the $B_1$ singular limit, as it converges to the $\kappa_2$-axis as shown in Figure \ref{fig:B1_bifurcation_diagram}.
\end{rem}

\begin{rem}
\label{rem:heteroclinic2}
    As in $B_1$, the analysis in $B_2$ reveals the existence of a heteroclinic connection in a blown-up space when $\kappa_1 = 1/2$ and $\kappa_2 \in (3/4, 1)$. This `connection' breaks when $\eps_1 > 0$, however we expect the existence of something akin to a separatrix for all $0 < \eps_1 \ll 1$, and this is expected to have important consequences for the basins of attraction associated with the two stable fixed points in regions II and III. Further details will be given below.
\end{rem}

\subsubsection{Canards and the singular bifurcation set in $B_3$}

Our final results for this section describe the two-parameter bifurcation set in $B_3$. These results will be stated for system \eqref{eq:fund_prob_kappa} in parameter chart $\mathcal P_2$, i.e.~for the system
\begin{equation}
    \label{eq:fund_prob_P2}
    \begin{split}
    \Dot{x}&=H(x, \tilde \eps_1 \eps_2,\theta_1)+H(y,\eps_2,\theta_2)-2 H(x,\tilde \eps_1 \eps_2,\theta_1) H(y,\eps_2,\theta_2)-\kappa_1 \theta_1^{-1} x , \\
        \Dot{y}&=1-H(x, \tilde \eps_1 \eps_2, \theta_1) H(y, \eps_2, \theta_2)- \kappa_2 \theta_2^{-1} y ,
    \end{split}
\end{equation}
with $0 < \tilde \eps_1, \eps_2 \ll 1$. Within $B_3$, our analysis reveals the existence of three important scaling regimes:
\begin{itemize}
    \item[(S1):] An outer regime defined on compact subsets of $\Lambda \setminus \{ (\tfrac{1}{2}, \kappa_2) : \kappa_2 \in [\tfrac{3}{4}, 1] \}$;
    \item[(S2):] An inner regime with asymptotics $(\kappa_1, \kappa_2) = (\tfrac{1}{2} + \mathcal O(\tilde \eps_1^2) , \tfrac{3}{4} + \mathcal O(\tilde \eps_1) )$;
    \item[(S3):] An inner regime with asymptotics $(\kappa_1, \kappa_2) = (\tfrac{1}{2} + \mathcal O(\tilde \eps_1), \kappa_2 )$, where $\kappa_2 \neq 3/4$.
\end{itemize}
Here the $\mathcal O$-notation is understood with respect to $\tilde \eps_1 \to 0$. Our first main result in $B_3$ describes a singular Bogdanov-Takens bifurcation for $(\kappa_1, \kappa_2)$-values in (S2). In order to state it, we work with $\kappa_2$-values in the interval
\[
\mathcal{I}_{M,\Tilde{\eps}_1}^2 := \{\kappa_2 \in (0,1):\, |\kappa_2-3/4|\leq M\Tilde{\eps}_1\}
\]
where $M > 0$ is arbitrarily large but fixed.

\begin{thm}
    \label{thm:B3_BT}
    \textup{(Singular Bogdanov-Takens bifurcation in $B_3$)}
    Fix $M > 0$. There exists $\Tilde{\eps}_{1,0} > 0$, $\eps_{2,0} > 0$ and a $C^1$-function $(\kappa_{1,BT}, \kappa_{2,BT}) : [0,\tilde \eps_{1,0}) \times [0,\eps_{2,0}) \to \R^2$ such that for all $(\tilde \eps_1, \eps_2) \in (0, \tilde \eps_{1,0}) \times (0, \eps_{2,0})$, system \eqref{eq:fund_prob_P2} undergoes a Bogdanov-Takens bifurcation at $(\kappa_1, \kappa_2) = (\kappa_{1,BT}, \kappa_{2,BT})(\tilde \eps_1, \eps_2)$. The leading order asymptotics are given by
    \[
    (\kappa_{1,BT},\kappa_{2,BT})(0, 0) = \left( \frac{1}{2} , \frac{3}{4} \right) .
    \]
    The associated saddle-node bifurcation occurs along $\kappa_1 = \kappa_{1,sn}(\kappa_2,\tilde\eps_1,\eps_2)$, where $\kappa_{1,sn}: \mathcal{I}_{M,\Tilde{\eps}_1}^2 \times [0, \tilde\eps_{1,0}) \times [0,\eps_{2,0}) \to \R$ is a $C^1$-function which satisfies
    \[
    \kappa_{1,sn}(\kappa_2,0,0)=2\kappa_2-2+2\sqrt{1-\kappa_2}
    \]
    for all $\kappa_2 \in \mathcal{I}_{M,\Tilde{\eps}_1}^2$.
    
    The associated Hopf bifurcation is subcritical and occurs along $\kappa_1 = \kappa_{1,h}(\kappa_2, \tilde\eps_1, \eps_2)$ when $\kappa_2 > \kappa_{2,BT}(\tilde \eps_1, \eps_2)$, where $\kappa_{1,h}: \mathcal{I}_{M,\Tilde{\eps}_1}^2 \times [0, \tilde\eps_{1,0}) \times [0, \eps_{2,0}) \to \R$ is a $C^1$-function which satisfies
    \[
    \kappa_{1,h}(\kappa_2,0,0)=\frac{1}{2} 
    \]
    for all $\kappa_2 \in \mathcal{I}_{M,\Tilde{\eps}_1}^2$.
\end{thm}

We defer our discussion of this result until after the statement of the remaining results in (S1) and (S3). In order to state these results, we need the following preliminary result.

\begin{lem}
    \label{lem:B3_invariant_manifolds}
    Fix $(\kappa_1, \kappa_2) \in \Lambda$, an arbitrarily small constant $\nu > 0$, 
    and consider system \eqref{eq:fund_prob_P2}. 
    There exists $\tilde \eps_{1,0}, \eps_{2,0} > 0$ such that for all $(\Tilde{\eps}_1, \eps_2) \in (0,\Tilde{\eps}_{1,0}) \times (0, \eps_{2,0})$ there exist attracting/repelling, locally invariant manifolds $\mathcal S^{a/r}$ which can be written as graphs of the form
    \[
    \begin{split}
    \mathcal S^a = \{ (h^a(y, \tilde \eps_1, \eps_2) , y) : y \in [\theta_2 + \nu, \nu^{-1}] \} , \qquad 
    \mathcal S^r = \{ (h^r(y, \tilde \eps_1, \eps_2) , y) : y \in [0, \theta_2 - \nu] \} ,
    \end{split}
    \]
    where $h^{a/r}$ are $C^r$-smooth functions which satisfy $h^a(y, \tilde \eps_1, \eps_2) = \theta_1 + \mathcal O(\Tilde{\eps}_1\eps_2)$ and $h^r(y, \tilde \eps_1, \eps_2) = \theta_1 + \mathcal O(\Tilde{\eps}_1\eps_2)$ as $\eps_1 = \Tilde{\eps}_1\eps_2 \to 0$.
\end{lem}

Lemma \ref{lem:B3_invariant_manifolds} will be proven in Section \ref{sec:Region_B3}. The manifolds $\mathcal S^{a/r}$ are obtained as Fenichel slow manifolds $\mathcal{S}_{\Tilde{\eps}_1}^{a/r}$ in a suitable blown-up space. They are $\mathcal O(\eps_1)$-close to the switching manifold $\Sigma_2$ as $\eps_1 \to 0$ (recall that $\eps_1=\Tilde{\eps}_1\eps_2$). 
Strictly speaking, they are not slow manifolds in the blown-down space (i.e.~in system \eqref{eq:fund_prob_P2}), however, they do inherit many of the usual invariance and attractivity/repulsivity properties of slow manifolds; we refer again to \cite{Bonet2016,Buzzi2006,Kristiansen2019c,Llibre2009} for details (noting however that the situation is slightly different in our case because of the additional small parameter). 

\begin{rem}
    The blow-up analysis in Sections \ref{sec:Region_B2}-\ref{sec:Region_B3} reveals the existence of an additional attracting and locally invariant manifold which is $\mathcal O(\eps_2)$-close to $\Sigma_1$ on $\{ x > \theta_1 \}$ as $\eps_2 \to 0$. This invariant manifold will also play a role in the analysis, however, we have not included it in Lemma \ref{lem:B3_invariant_manifolds} because it is not necessary for the statement of the main results below.
\end{rem}

Finally, we state our main results on the dynamics in $B_3$ for $(\kappa_1, \kappa_2)$-values in (S1) and (S3). To state it, we consider $\kappa_2$-values in the interval
\[
\mathcal{I}_{M}^1 := \{\kappa_2 \in (0,1):\, |\kappa_2-3/4|\geq 1/M\} ,
\]
where for simplicity we let $M > 0$ denote that the same arbitrarily large but fixed constant used in the definition of $\mathcal I^2_{M,\tilde \eps_1}$.


\begin{thm}
\label{thm:B3}
    \textup{(Singular bifurcations and canards in $B_3$)}
    Fix $M > 0$. There exist $\Tilde{\eps}_{1,0} > 0, \eps_{2,0} > 0$ such that the following assertions are true:
    \begin{enumerate}
        \item[(i)] 
        There exists a $C^1$-function $\kappa_{1,sn}: \mathcal I_M^1 \times [0, \tilde\eps_{1,0}) \times [0,\eps_{2,0}) \to \R$ such that 
        system \eqref{eq:fund_prob_P2} undergoes a saddle-node bifurcation when $\kappa_1 = \kappa_{1,sn}(\kappa_2,\tilde\eps_1,\eps_2)$. Moreover,
        \begin{equation}
        \label{eq:sn_curve_B3}
            \kappa_{1,sn}(\kappa_2,0,0)=2\kappa_2-2+2\sqrt{1-\kappa_2}
        \end{equation}
        for all $\kappa_2 \in \mathcal I^1_M$.
        \item[(ii)] There exists a $C^1$-function $\kappa_{1,h}: \mathcal I_M^1 \times [0, \tilde\eps_{1,0}) \times [0,\eps_{2,0}) \to \R$ such that 
        system \eqref{eq:fund_prob_P2} undergoes a (possibly degenerate) singular Hopf bifurcation when $\kappa_1 = \kappa_{1,h}(\kappa_2,\tilde \eps_1,\eps_2)$ and $\kappa_2 > 3/4$. Moreover,
        \[
        \kappa_{1,h}(\kappa_2,0,0) = \frac{1}{2}
        \]
        for all $\kappa_2 \in \mathcal I^1_M \cap \{\kappa_2 > 3/4\}$.
        \item[(iii)] Fix $\kappa_2 \in \mathcal I^1_M \cap \{\kappa_2 > 3/4\}$. Then there exists a $C^1$-function $\kappa_{1,c} : [0, \tilde \eps_{1,0}) \times [0, \eps_{2,0}) \to \R$ such that for all $(\tilde \eps_1, \eps_2) \in (0, \tilde \eps_{1,0}) \times (0, \eps_{2,0})$, the forward and backward extension of the attracting and repelling invariant manifolds $\mathcal S^a$ and $\mathcal S^r$ under the flow of system \eqref{eq:fund_prob_P2} respectively undergo a canard-type intersection when $\kappa_1 = \kappa_{1,c}(\tilde \eps_1, \eps_2)$. Moreover, 
        \[
        \kappa_{1,c}(0,0)=\frac{1}{2} .
        \]
    \end{enumerate}
\end{thm}

\begin{figure}
\centering
\begin{subfigure}{.4\linewidth}
  \includegraphics[width=\linewidth]{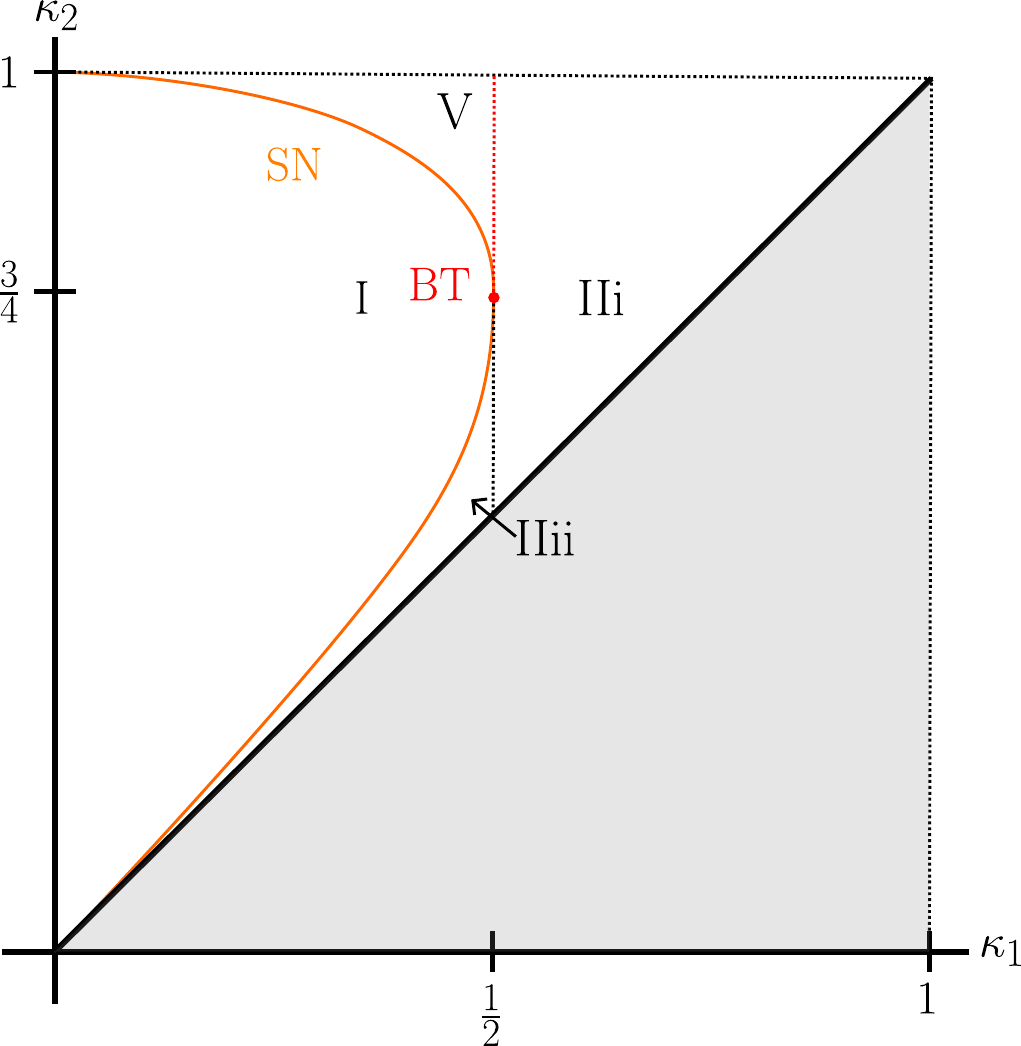}
  \caption{}
  \label{fig:B3_bifurcation_diagram_singular}
\end{subfigure}
\hspace{1cm} 
\begin{subfigure}{.4\linewidth}
  \includegraphics[width=\linewidth]{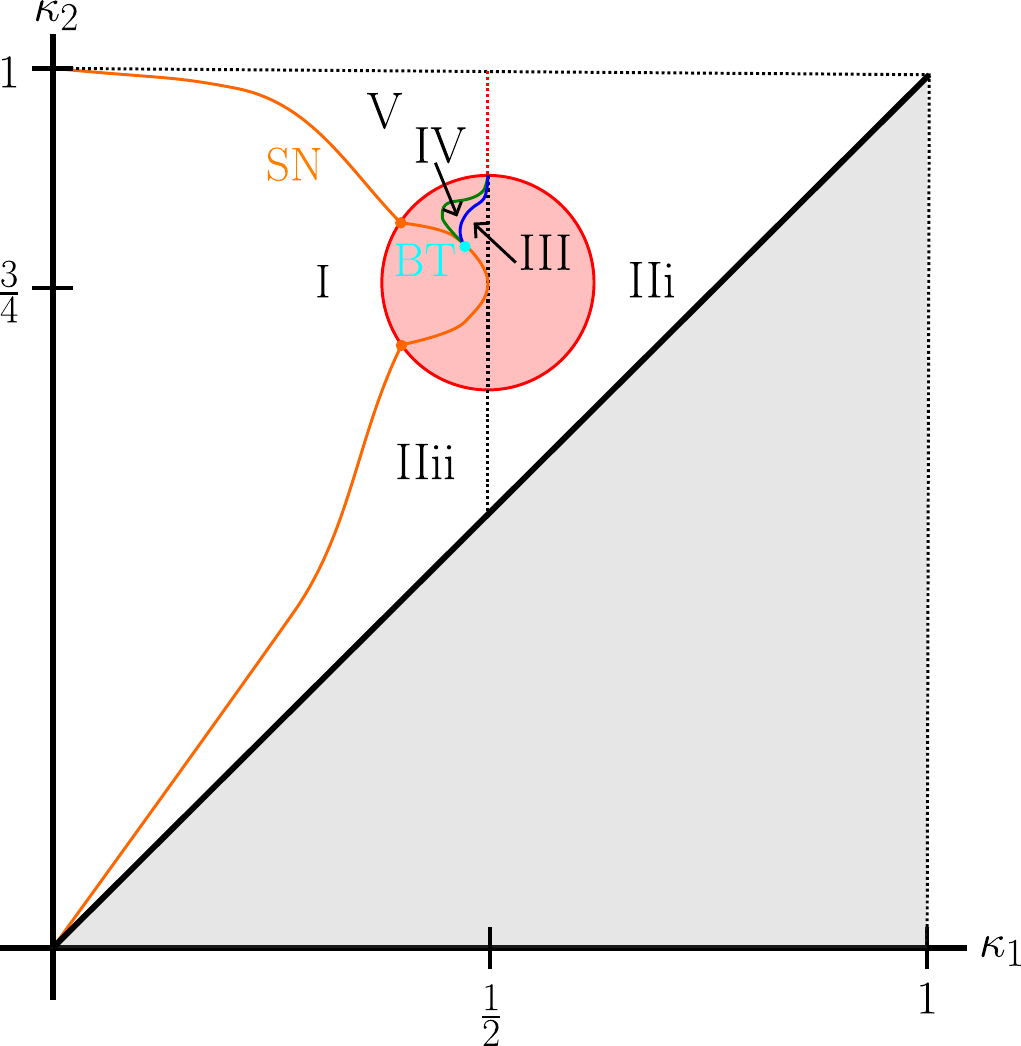}
  \caption{}
  \label{fig:B3_bifurcation_diagram_sphere_only}
\end{subfigure}
\caption{Limiting bifurcation set for system \eqref{eq:fund_prob_P2} in $B_3$, as described by Theorems \ref{thm:B3_BT} and \ref{thm:B3}. The labelling and coloring conventions are as in Figures \ref{fig:B1_bifurcation_diagram} and \ref{fig:B2_bifurcation_diagram}. The bifurcation set in (a) resembles the limiting bifurcation set in $B_2$ shown in Figure \ref{fig:B2_bifurcation_diagram} except that (i) the (now degenerate) Bogdanov-Takens point BT has moved to the turning point; (ii) the Hopf and Hom curves overlap along the degenerate line in red which emanates from the point BT (thus, region III and IV disappear) and organizes the global dynamics of the two attracting equilibria; (iii) region II has been split into subregions II(i) and II(ii) for the reasons provided in the text (see Section \ref{sub:main_results_blow-up}). Figure (b) shows a partially resolved bifurcation set. Here, a regular Bogdanov-Takens with associated subcritical Hopf and homoclinic curves are revealed after spherical blow-up of the degenerate BT point in $(\kappa_1, \kappa_2, \tilde \eps_1)$-space (the $\tilde \eps_1$-axis is omitted to keep the figure minimal). The red line which separates regions IIi and V above the sphere is still degenerate; we refer to Figure \ref{fig:main_results} below for a continuation of this figure and zoom of the bifurcation set on the blow-up sphere.}
\label{fig:parameter_blow-up_2D_sequence}
\end{figure}

\begin{rem}
    The qualifier "possibly degenerate" appears in Assertion (ii) because we were unable to prove that the associated first Lyapunov coefficient is non-zero for all $0 < \tilde \eps_1, \eps_2 \ll 1$. Numerical and analytical investigation suggest that the bifurcation is nondegenerate and subcritical (as it is near the singular Bogdanov-Takens point). We refer to Remark \ref{rem:Hopf_proof} below for further details and a proposed approach to a proof.
\end{rem}

Theorems \ref{thm:B3_BT} and \ref{thm:B3} are proven in Section \ref{sec:Region_B3}. Figures \ref{fig:parameter_blow-up_2D_sequence} and \ref{fig:main_results} show a sequence of two-parameter bifurcation sets, which are progressively desingularised via geometric blow-up transformations in order to reveal the (singular) bifurcation structure described in Theorems \ref{thm:B3_BT} and \ref{thm:B3} in the singular limit when $(\tilde \eps_1, \eps_2) \to (0,0)$. Figure \ref{fig:B3_bifurcation_diagram_singular} shows the bifurcation diagram in the subset $\Lambda$ of the $(\kappa_1,\kappa_2)$-plane. The same saddle-node curve is shown in orange. The red point labeled ``BT'' has shifted to the turning point, and is degenerate in the original coordinates when $(\tilde \eps_1, \eps_2) = (0,0)$. Assertion (i) in Theorem \ref{thm:B3} describes the persistence of this saddle-node bifurcation away from the degenerate point BT; it can be proven in the outer scaling regime (S1). In order to uncover the singular Bogdanov-Takens bifurcation described by Theorem \ref{thm:B3_BT}, which describes the qualitative dynamics in (S2), we apply an additional (weighted) blow-up in the $(\kappa_1, \kappa_2, \tilde \eps_1)$-parameter space; this is sketched in Figure \ref{fig:B3_bifurcation_diagram_sphere_only}. 
Our analysis reveals a regular (subcritical) Bogdanov-Takens bifurcation in local coordinates which cover the top of this blow-up sphere. We also identify the extension of the saddle-node curve (which connects to the saddle-node curve outside of the blow-up sphere), and the associated Hopf and homoclinic curves (shown in green and blue respectively). We conjecture that the Hopf and homoclinic curves connect to the `north pole' of the blow-up sphere as shown in the figure, however, we do not prove this.

\begin{rem}
    Our analysis in Section \ref{sec:Region_B3} leads to the `almost complete' bifurcation set shown in Figure \ref{fig:main_results}, however, we did not connect the SN curves across the equator of the blow-up sphere, or match the Hopf and homoclinic curves between the sphere and the cylinder. 
    Further details on how the latter matching problem can be addressed using geometric blow-up techniques are given in Remark \ref{rem:restriction}.
\end{rem}

Finally, Figure \ref{fig:B3_bifurcation_diagram} shows the extension of the Hopf and (conjectured) homoclinic curves along the vertical axis which tracks the surface of a subsequent cylindrical blow-up along $\kappa_1 = 1/2$, $\kappa_2 > 3/4$ which is used in order to study the dynamics in (S3). Assertions (ii) and (iii) of Theorem \ref{thm:B3} are proven in local coordinates which cover the top of this blow-up cylinder. The (possibly degenerate) singular Hopf bifurcation described in assertion (ii) appears as a (possibly degenerate) regular Hopf in these coordinates, and the canard-like intersection of the invariant manifolds $\mathcal S^{a/r}$ described in assertion (iii) is proven using a straightforward adaptation of the Melnikov-based arguments used in e.g.~\cite{Krupa_2001_Extend,krupa_extending_transcritical,Wechselberger2002}. Indeed, the proof of Theorem \ref{thm:B3} reveals the presence of a \textit{regularised visible-invisible two-fold} singularity in this region of parameter space. This allows us to appeal to existing analyses in \cite{Bonet2018,Kristiansen_2015,Kristiansen_2023} (the reference \cite{Kristiansen_2023} in particular), in order to describe some of the finer details associated with the unfolding of this singularity. In particular it allows us to prove the existence of canard cycles in the relevant blown-up space. The interested reader is referred to Proposition \ref{prop:B3_canard_uniqueness} in Section \ref{sec:Region_B3} for a detailed statement of this result; we have opted not to present it above in order to ease the presentation of the other results.

\begin{rem}
\label{rem:canards_ordering}
    The observations above suggest the existence of a function $\kappa_{1,hom}(\kappa_2, \tilde \eps_1, \eps_2)$ such that for each fixed $\kappa_2 \in (3/4,1)$, system \eqref{eq:fund_prob_P2} undergoes a homoclinic bifurcation when $\kappa_1 = \kappa_{1,hom}(\kappa_2, \tilde \eps_1, \eps_2)$, and $\kappa_{1,hom}(\kappa_2, 0, 0) = 1/2$. Although we do not prove this, consistency with the bifurcation diagram in (S2) and the constraints imposed by the geometry of the phase space suggest that when $0 < \tilde \eps_1, \eps_2 \ll 1$ we have
    \[
    \kappa_{1,h}(\kappa_2, \tilde \eps_1, \eps_2) < \kappa_{1,hom}(\kappa_2, \tilde \eps_1, \eps_2) < \kappa_{1,c}(\kappa_2, \tilde \eps_1, \eps_2) .
    \]
    In particular, we expect regions III and IV to persist as narrow wedge-shaped regions which converge to the line $\kappa_1 = 1/2, \kappa_2 \in (3/4,1)$ as $(\tilde \eps_1, \eps_2) \to (0,0)$; we refer to Figure \ref{fig:B3_blown-up_bifurcation_set} in Section \ref{sec:Region_B3} for a sketch of this. If this is correct, then our results suggest that the topology of the phase portraits in regions I, II, III IV and V is similar in regions $B_2$ and $B_3$. The presence of singular bifurcations and canards in $B_3$, however, shows a substantially increased sensitivity of the system to parameter variation relative to $B_2$. It is also significant to note that region IV, where oscillations exist, shrinks to zero as $(\tilde \eps_1, \eps_2) \to (0,0)$.
\end{rem}

\begin{figure}[t!]
\centering
%
\begin{subfigure}{.4\linewidth}
  \includegraphics[width=\linewidth]{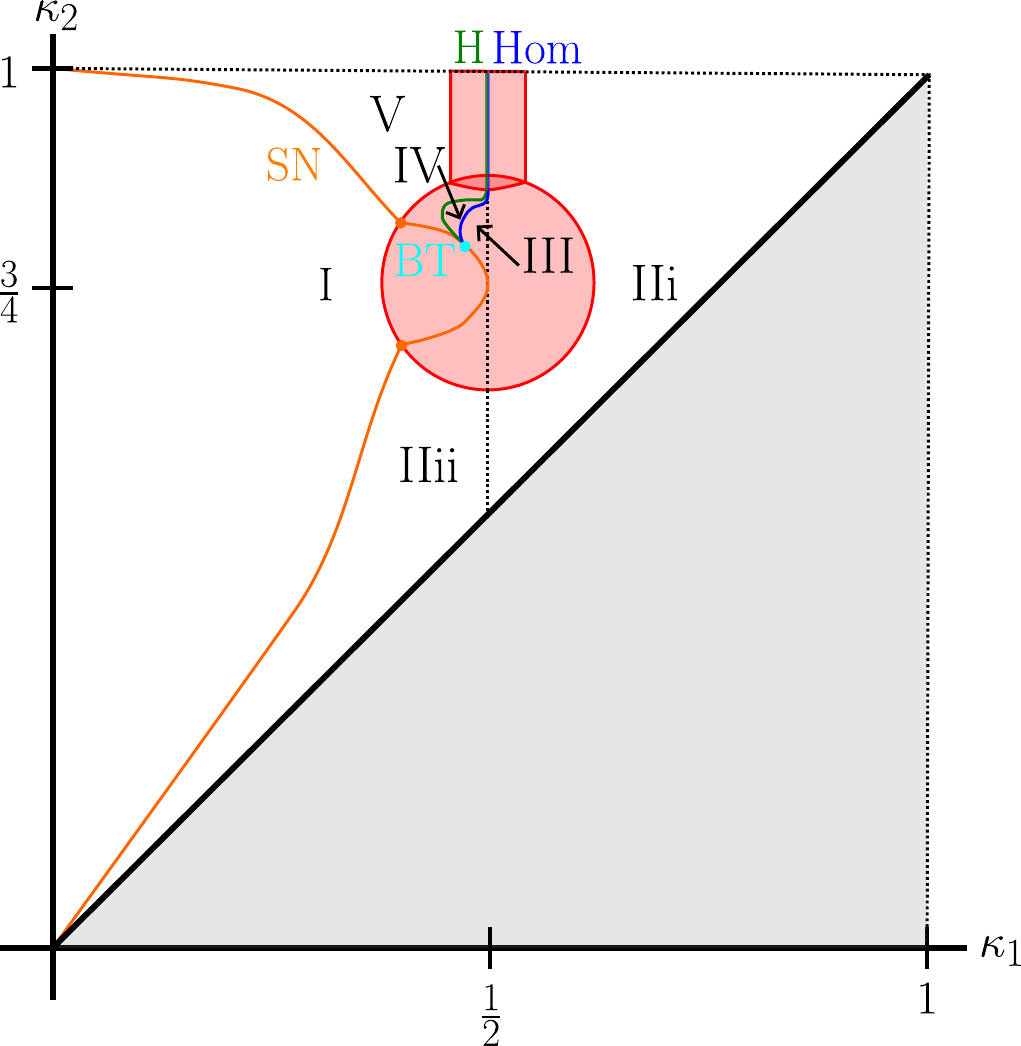}
  \caption{}
  \label{fig:B3_bifurcation_diagram}
\end{subfigure}
\hspace{1cm}
\begin{subfigure}{.41\linewidth}
  \includegraphics[width=\linewidth]{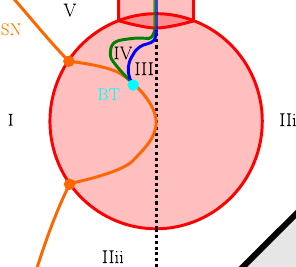}
  \hspace{0.5cm}
  \caption{}
  \label{fig:B3_bifurcation_diagram_zoom}
\end{subfigure}
\caption{Continuation of the blow-up sequence shown in Figure \ref{fig:parameter_blow-up_2D_sequence}. In (a): the degenerate line which emanates from the north pole of the blow-up sphere in Figure \ref{fig:B3_bifurcation_diagram_sphere_only} is blown-up to a cylinder (also shown in shaded red), allowing for the Hopf and homoclinic curves to be extended. These two curves (shown in green and blue respectively) overlap along the cylinder due to the canard-related phenomena identified in Theorem \ref{thm:B3}, however, we expect that they are non-intersecting away from the cylinder's surface, i.e., for $0<\Tilde{\eps}_1 \ll 1$, and that region IV persists as a narrow interval which shrinks to the line shown on the cylinder in the (dual) singular limit; see Remark \ref{rem:canards_ordering} and Figure \ref{fig:B3_blown-up_bifurcation_set} for more on this. (b) A zoom of the bifurcation set on the blow-up sphere and cylinder appearing in (a).}
\label{fig:main_results}
\end{figure}

\begin{rem}
\label{rem:heteroclinic3}
    The presence of canard-like solutions when $\kappa_1 = \kappa_{1,c}(\tilde \eps_1, \eps_2)$, together with the fact that the system is bistable in this regime, appears to have important consequences for the expected global dynamics as $t \to \infty$. In particular, the canards appear to act like a separatrix which determines which attractor solutions are expected to converge to $t \to \infty$. This is discussed in more detail in Section \ref{sub:main_results_blow-up} below. Interestingly, even though it only appears as a canard-type solution in $B_3$, similar solutions appear quite robustly for $\kappa_1 \approx 1/2$ and $\kappa_2 \in (3/4,1)$ across all three regions $B_1$, $B_2$ and $B_3$; cf Remarks \ref{rem:heteroclinic} and \ref{rem:heteroclinic2} above. 
\end{rem}

\begin{rem} \label{rem:main_theorem}
    The fact that the leading order expressions for the saddle-node curve given by \eqref{eq:sn_curve_B1}, \eqref{eq:sn_curve_B2} and \eqref{eq:sn_curve_B3} in $B_1$, $B_2$ and $B_3$ respectively are identical shows that the saddle-node bifurcation in system \eqref{eq:fund_prob_kappa} is extremely robust to variation in the switching parameters $\eps_i$, $i \in \{1,2\}$. Comparing Theorems \ref{thm:B1}, \ref{thm:B2} and \ref{thm:B3} shows that the most significant effects of varying the relative size of $\eps_1$ and $\eps_2$ are (i) the existence and location of Hopf and Bogdanov-Takens bifurcations, and (ii) the multi-scale structure and associated sensitivity of the system to variation in the other system parameters (here $\kappa_1$ and $\kappa_2$). The question of whether these features are true of the higher dimensional ODE-based models for GRN dynamics considered by \cite{Plahte_GRN_2005,Edwards2015,Ironi2011,Machina2013a} and many others is open, and remains for future work.
\end{rem}

\subsection{Singular geometry and dynamics in phase space}
\label{sub:main_results_blow-up}

The (limiting) two-parameter bifurcation sets shown in Figures \ref{fig:B1_bifurcation_diagram}, \ref{fig:B2_bifurcation_diagram} and \ref{fig:main_results}, which are characterised by Theorems \ref{thm:B1}, \ref{thm:B2}, \ref{thm:B3_BT} and \ref{thm:B3}, organize the dynamics in phase space for the respective scaling regimes $B_1$, $B_2$ and $B_3$. 
%
%
The differences between the bifurcation sets are primarily due to the fact that the geometry and dynamics as $(\eps_1, \eps_2) \to (0,0)$ differs when $(0,0)$ is approached from within $B_1$, $B_2$ or $B_3$. In this section we informally summarise the geometry and dynamics in phase space, focusing on the singular limit associated with each region $B_i$; the details will be presented in Sections \ref{sec:Region_B2}, \ref{sec:Region_B1} and \ref{sec:Region_B3} below. In each case, multiple geometric blow-up transformations are required in order to resolve degeneracies associated with the loss of smoothness along discontinuity sets and, in some cases, the loss of hyperbolicity at equilibria identified after preliminary blow-ups. The global geometry and dynamics are described without proof, however, the blow-up constructions (i) show that these degeneracies can be resolved, thereby paving the way for more detailed analyses if necessary, and (ii) lead naturally to the non-trivial coordinate scalings which are needed in order to unfold the bifurcation structure and prove the results of the previous section.

We start by considering the 3-dimensional extended system obtained after appending the trivial equation $\eps_1' = 0$ to system \eqref{eq:fund_prob_P1} in $B_2$. Since $\tilde \eps_2 \in [\beta_1, \beta_2]$ is fixed in $B_2$, the singular limit corresponds to the set $\{\eps_1 = 0\}$, within which the system suffers a loss of smoothness along $\Sigma_1 \cup \Sigma_2 \times \{0\}$. A total of 5 geometric blow-ups are applied in order to resolve this degeneracy: a single spherical blow-up of the intersection point $I = \Sigma_1 \cap \Sigma_2 \times \{0\}$, followed by 4 cylindrical blow-ups which resolve the loss of smoothness along the left/right branches of $\Sigma_1 \times \{0\}$ which emanate from the blow-up sphere, as well as the upper/lower branches of $\Sigma_2 \times \{0\}$. The sequence of blow-up transformations is the same as in \cite{Simon_thesis,Jelbart2026}, and is sketched (in birds-eye view) in Figure \ref{fig:B2_blow-up}. The figure shows normally hyperbolic critical manifolds that are identified in three of the four cylindrical blow-ups; $\mathcal S^+_1$ and $\mathcal S_2^+$ are attracting, whereas $\mathcal S_1^-$ is repelling. These critical manifolds are directly correlated with stable and unstable sliding in the blown-down PWS singular limit (we refer again to \cite{Buzzi2006,Kristiansen2019c,Llibre2009} for an applicable result from the regularized PWS point of view, and to \cite{Plahte_GRN_2005} for a GRN-specific justification of this claim). The reduced dynamics along $\mathcal S^\pm_2$ is oriented towards the sphere, and an analysis of the reduced problem on $\mathcal S_1^+$ reveals the attracting equilibrium $q_{n}$. Proposition \ref{prop:node} ensures that $q_{n}$ remains bounded away from the blow-up sphere on $\mathcal S_1^+$ for all $(\kappa_1, \kappa_2) \in \Lambda$. Finally, we note that each critical manifold can be locally extended onto the blow-up sphere as a center manifold with the orientation indicated in the figure, and that there is a hyperbolic equilibrium $q_{in}$ on the equator whose stable manifold forms a separatrix between solutions which are expected to converge to $q_{n}$ as $t \to \infty$, vs solutions which are expected to converge to an alternative attractor on the blow-up sphere (if there is one), or to pass over the blow-up sphere en route to $q_n$.

\begin{figure}[t!]
\begin{subfigure}{.359\linewidth}
  \includegraphics[width=\linewidth]{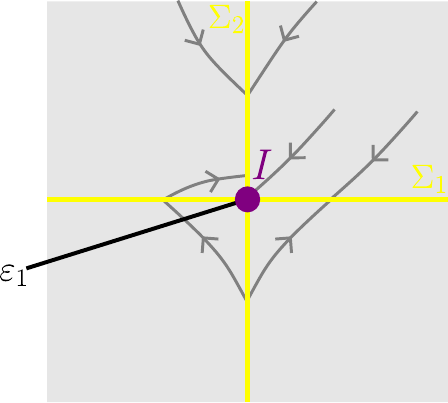}
  \caption{}
  \label{fig:B2_before_blow-up}
\end{subfigure}\hspace{0.13cm}
\begin{subfigure}{.32\linewidth}
  \includegraphics[width=\linewidth]{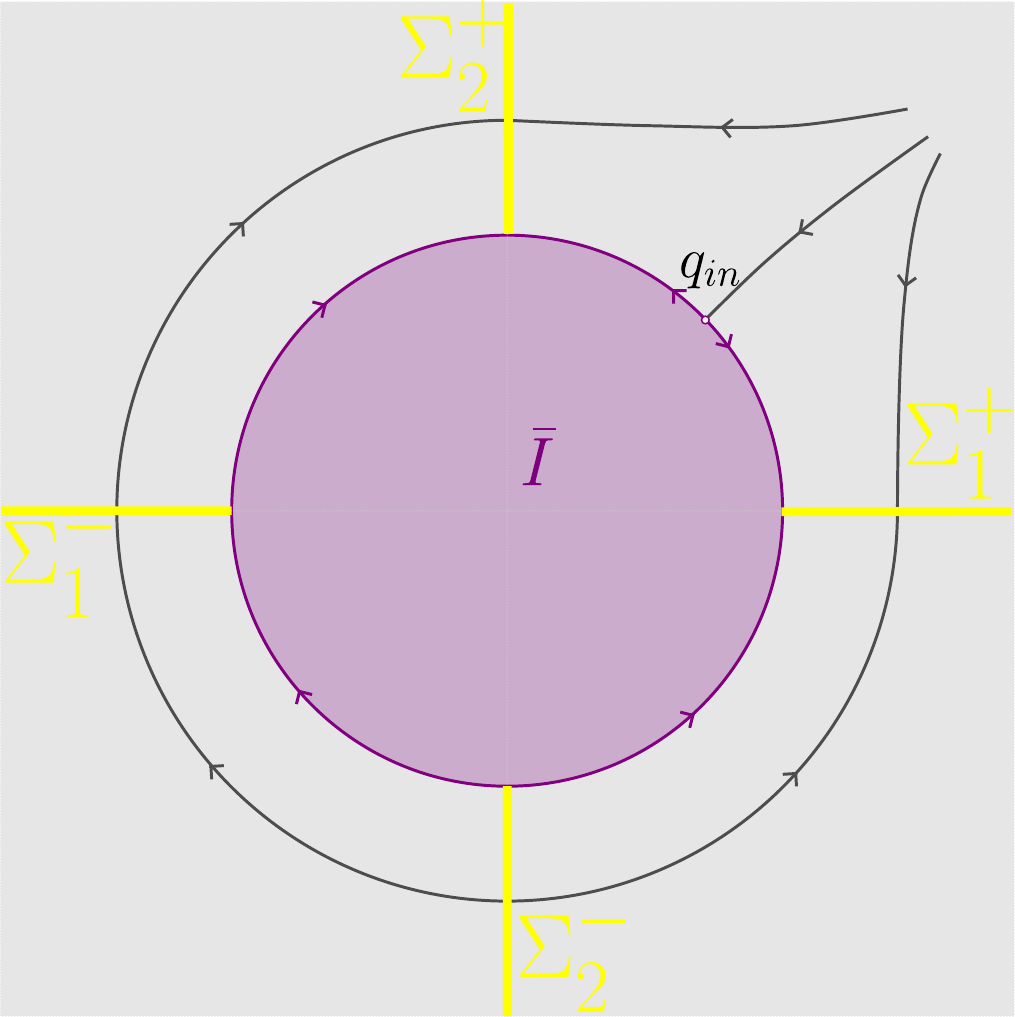}
  \caption{}
  \label{fig:B2_blow-up_1}
\end{subfigure}\hfill 
~ 
\begin{subfigure}{.32\linewidth}
  \includegraphics[width=\linewidth]{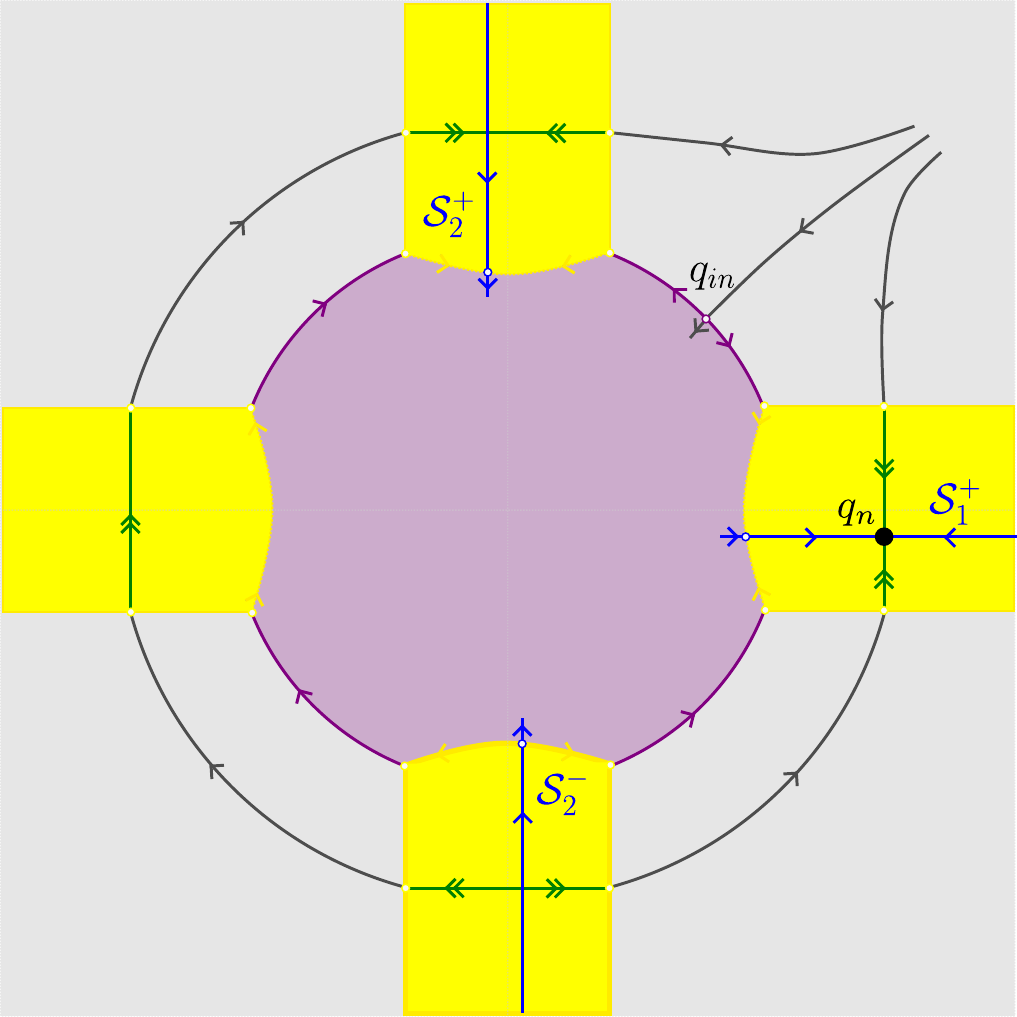}
  \caption{}
  \label{fig:B2_blow-up_2}
\end{subfigure} 
\caption{The blow-up sequence that resolves the loss of smoothness in system \eqref{eq:fund_prob_P1} when $\eps_1 = 0$. (a) Blown-down dynamics in the extended $(x,y,\eps_1)$-space. The PWS dynamics, including the switching manifold $\Sigma_1 \cup \Sigma_2$ and the intersection point $I = \Sigma_1 \cap \Sigma_2$ are contained within $\{\eps_1 = 0\}$. The labelling and coloring conventions are the same as in Figure \ref{fig:PWS_phase_space}. (b) Partial resolution after spherical blow-up of the intersection point $I$. There is still a loss of smoothness along the (now separated) branches denoted here by $\Sigma_1^\pm$ and $\Sigma_2^\pm$. There is a hyperbolic saddle point $q_{in}$ on the equator. (c) Fully resolved geometry and dynamics, following cylindrical blow-ups along $\Sigma_1^\pm$ and $\Sigma_2^\pm$. Attracting critical manifolds $\mathcal S_1^+$ and $\mathcal S_2^+$ are identified on the top and right-hand cylinders, and a repelling critical manifold $\mathcal S_2^-$ is revealed on the lower blow-up cylinder. All are sketched in blue, with fast/layer dynamics in green. We also indicate the orientation of the reduced flow on $\mathcal S_1^+$, $\mathcal S_2^-$ and $\mathcal S_2^+$, the local dynamics onto or off of the sphere, as well as the location of the stable node $q_n \in \mathcal S_1^+$. The dynamics on the blow-up sphere, which depends upon $(\kappa_1, \kappa_2)$, is described by Theorem \ref{thm:B2}.}
\label{fig:B2_blow-up}
\end{figure}

The geometric and dynamic features described above apply for any $(\kappa_1, \kappa_2) \in \Lambda$, however, the qualitative dynamics on the blow-up sphere will vary significantly with $(\kappa_1, \kappa_2)$. The dynamics on the sphere is organised by the unfolding of the Bogdanov-Takens bifurcation that is described in Theorem \ref{thm:B2}, and represented in the 2-parameter bifurcation set in Figure \ref{fig:B2_bifurcation_diagram}. 

\begin{rem} \label{rem:main_global_dynamics}
    Solutions with initial conditions below the stable manifold of $q_{in}$ and to the right of $\Sigma_2$ 
    are expected to approach a slow manifold which perturbs from $\mathcal S_1^+$, and thereafter converge to $q_n$.
    Solutions with initial conditions between 
    $\Sigma^+_2$ and the stable manifold of $q_{in}$, as well as all solutions with initial conditions to the left of $\Sigma_2$, 
    are expected to approach a slow manifold which perturbs from $\mathcal S_2^+$, before tracking the slow flow onto the blow-up sphere. From here, they can either (i) converge to an attractor 
    on the blow-up sphere, or (ii) exit along the intersection with $\mathcal S_1^+$ and converge to $q_n$. 
    Our analysis in Sections \ref{sec:Region_B2} - \ref{sec:Region_B3} suggests the existence of a function $\kappa_{1,het}(\kappa_2, \eps_1, \eps_2)$ with $\kappa_{1,het}(\kappa_2, 0, 0)=1/2$ such that system \eqref{eq:fund_prob_P2}$|_{\kappa_1 = \kappa_{1,het}(\kappa_2,\eps_1,\eps_2)}$ has a hyperbolic saddle $q_s$ on the blow-up sphere, whose stable manifold $w^s(q_s)$ contains the slow manifold which perturbs from $\mathcal S_2^+$ if $\kappa_2>3/4$. 
    This leads to a kind of separatrix, which seems to be independent of the ratio of the singular perturbation parameters $0<\eps_1, \eps_2 \ll 1$ and exists in all regions $B_i$, $i=1,\,2,\,3$ (we refer again to Remarks \ref{rem:heteroclinic}, \ref{rem:heteroclinic2} and \ref{rem:heteroclinic3} above). In region $B_3$ in particular, this is correlated with the existence of canard-like solutions for $\kappa_{1,het} = \kappa_{1,c}$.
    %
\end{rem}

We turn now to the singular geometry and dynamics in $B_1$, which is obtained by considering system \eqref{eq:fund_prob_P1} in the dual singular limit as $(\eps_1, \tilde \eps_2) \to (0, 0)$. Since the blow-up construction in $B_2$ outlined above already resolves degeneracies associated with the loss of smoothness as $\eps_1 \to 0$, we continue to work in blown-up space shown in Figure \ref{fig:B2_blow-up}. While the sequence of blow-up transformations in $B_2$ suffices to resolve the loss of smoothness in system \eqref{eq:fund_prob_P1} for each fixed $\tilde \eps_2 > 0$, there is an additional loss of smoothness along $y = \theta_2$ when $\tilde \eps_2 \to 0$. This induces a PWS system on the blow-up sphere and left/right cylinders. The dynamics is sketched in Figure \ref{fig:B1_piecewise_smooth}; note the presence of an associated switching manifold $\Sigma_1^1$ along the horizontal axis. The PWS analyses on the left and right blow-up cylinders reveals simple crossing and stable sliding dynamics respectively. The transition from stable sliding to crossing is (on the PWS level) facilitated by a regular invisible fold point (denoted by $F$ in Figure \ref{fig:B1_piecewise_smooth}) which is identified in a PWS analysis of the limiting system which governs the dynamics on the blow-up sphere. Three more cylindrical blow-ups are required to resolve the loss of smoothness along $\Sigma_1^1$; one on each cylinder, and another on the sphere. This results in the geometry and dynamics shown in Figure \ref{fig:B1_blow-up_structure}. As before, solutions simply cross over the blow-up cylinder on the left, and an attracting, normally hyperbolic critical manifold (which we again denote by $\mathcal S_1^+$) provides the mechanism for stable sliding on the right. Analysing the reduced problem on $\mathcal S_1^+$ reveals the stable equilibrium $q_{n}$, similarly to the $B_2$ analysis above. This is consistent with Proposition \ref{prop:node}, which asserts the presence of $q_{n}$ as a stable node in this region of phase space for all $(\eps_1, \eps_2) \in \mathcal U$ and $(\kappa_1, \kappa_2) \in \Lambda$.

\begin{figure}[t!]
    \centering
    \begin{subfigure}[t]{0.45\textwidth}
        \centering
        \includegraphics[width=0.99\linewidth]{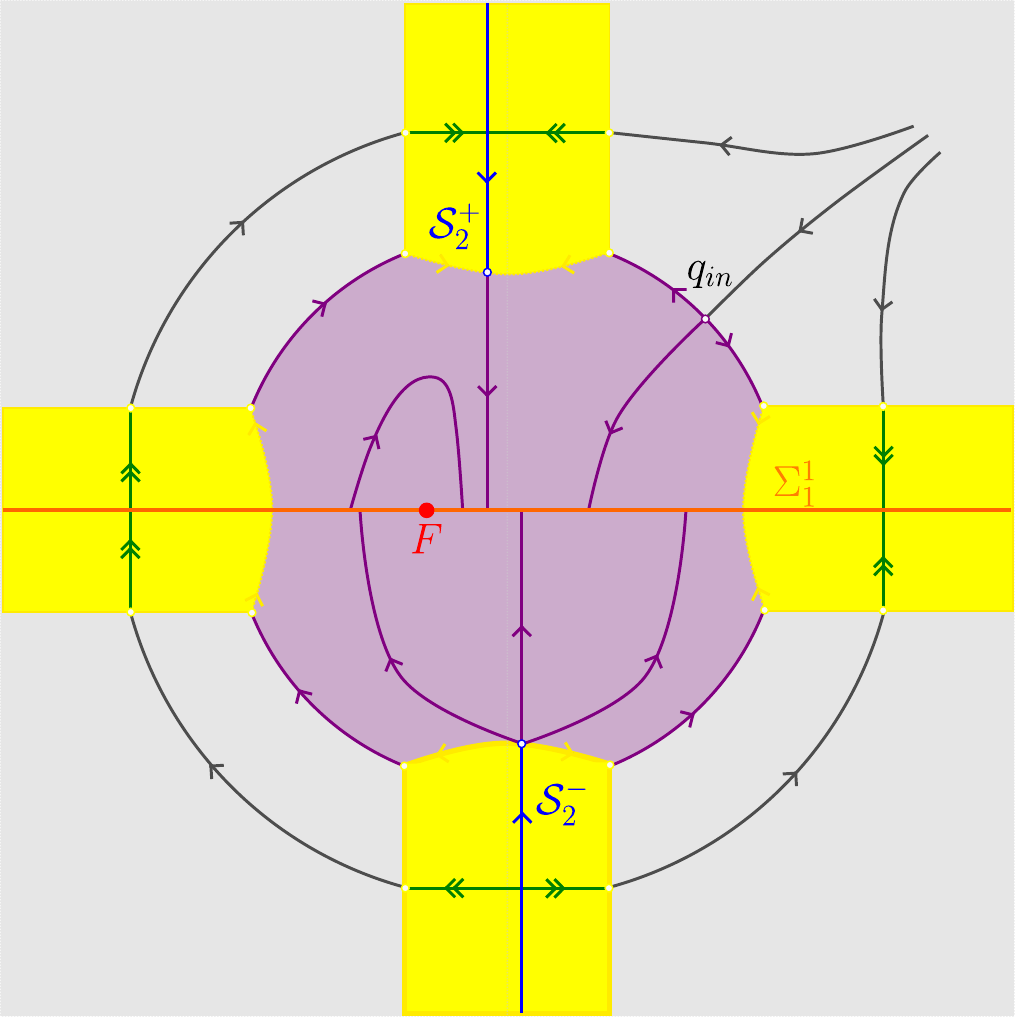}
    \caption{}
    \label{fig:B1_piecewise_smooth}
    \end{subfigure}%
    ~ 
    \begin{subfigure}[t]{0.45\textwidth}
        \centering
        \includegraphics[width=0.99\linewidth]{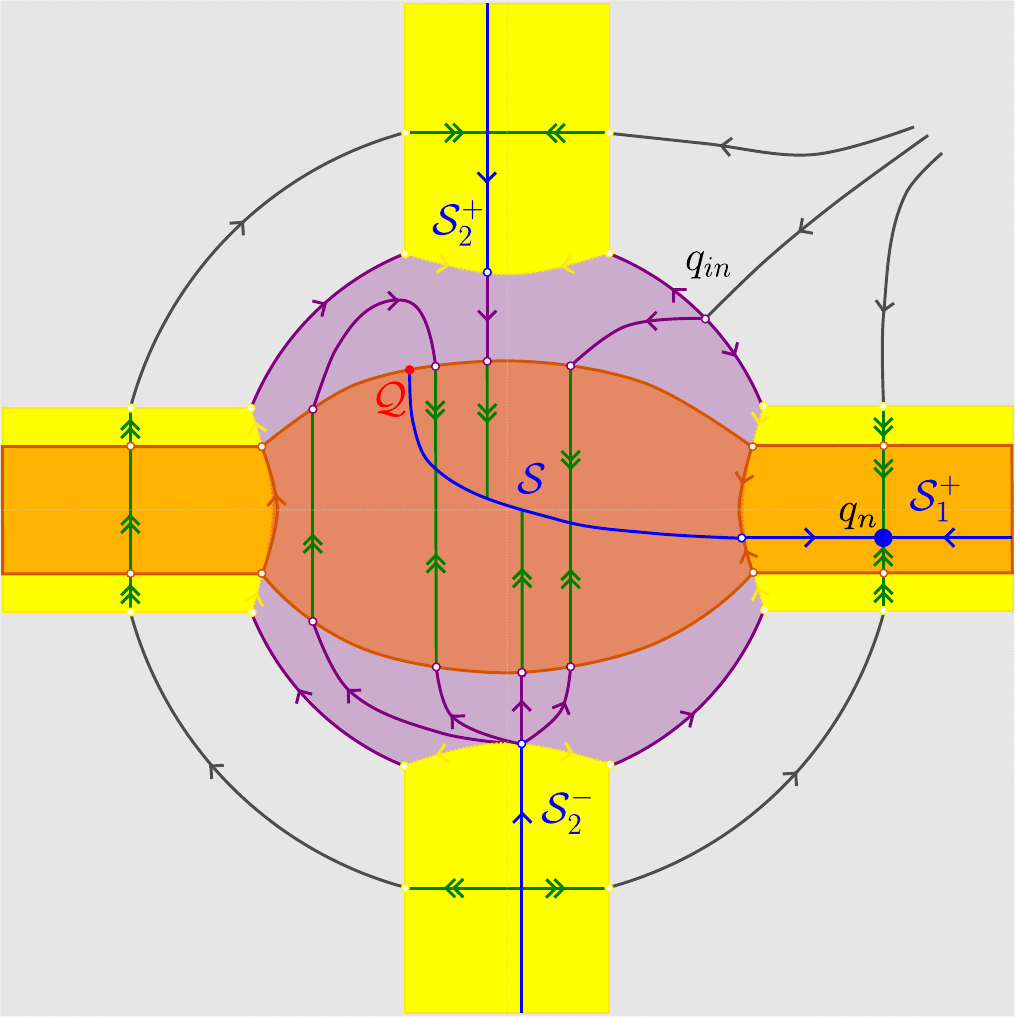}
    \caption{}
    \label{fig:B1_blow-up_structure}
    \end{subfigure}
    \caption{In (a): Singular PWS dynamics 
    in the $B_1$ singular limit $(\eps_1, \Tilde{\eps}_2) \to (0, 0)$. The switching manifold $\Sigma_1^1$ (orange) dissects the phase space horizontally. On the right cylinder (yellow) there is attracting sliding and on the left cylinder there is crossing dynamics. The transition from crossing to sliding is mediated by the presence of an invisible fold point $F$ (red) on the blow-up cylinder. In (b): Dynamics in region $B_1$ after cylindrical blow-up of $\Sigma_1^1$. On the `central' blow-up cylinder in orange, we obtain an attracting critical manifold $\mathcal{S}$, which connects the critical manifold $\mathcal{S}_1^+$ with a degenerate point denoted here by $\mathcal Q$ (red). The reduced flow on $\mathcal{S}$ is not shown; it depends upon the parameters $(\kappa_1, \kappa_2) \in \Lambda$.}
    \label{fig:B1_blow-up}
\end{figure}

The qualitative features of the geometry and dynamics described above are the same for all $(\kappa_1, \kappa_2) \in \Lambda$, but there are key qualitative differences in the dynamics on the `central' blow-up cylinder, i.e.~on the blow-up cylinder which lies on top of the sphere. Here, we identify an attracting critical manifold (denoted by $\mathcal S$ in the figure) which connects a degenerate point $\mathcal Q$ which stems from the presence of the invisible fold point $F$ with the endpoint of $\mathcal S_1^+$ for all $(\kappa_1, \kappa_2) \in \Lambda$. The reduced dynamics on $\mathcal S$ can have either $0$, $1$ or $2$ equilibria depending on the region in the $(\kappa_1,\kappa_2)$-plane due to the saddle-node bifurcation described in Theorem \ref{thm:B1}.

\begin{rem}
    \label{rem:Q_degeneracy}
    Degenerate points like the point $\mathcal Q$ shown in Figure \ref{fig:B1_blow-up_structure} arise naturally in blow-up desingularisations of smooth perturbations of PWS systems with tangency points; see e.g.~\cite{Bonet2016,Kristiansen2019c,Jelbart2020,Jelbart2021}. However, the loss of hyperbolicity at $\mathcal Q$ is more complicated here because of an exponential/beyond all orders tangency between the fast foliation and $\mathcal S$ at $\mathcal Q$. This stems from the fact that the Hill function $H(z)$ is flat as $z \to \pm \infty$. It should be possible to resolve this degeneracy using geometric blow-up techniques based on the method in \cite{Kristiansen_2017} (see also \cite{Jelbart_2021,Kristiansen2019b} for applications of this method), but this is beyond the scope of the present article. For our purposes, it suffices to note that equilibria of the reduced problem on $\mathcal S$, when they exist, are bounded away from $\mathcal Q$ for any fixed $(\kappa_1, \kappa_2) \in \Lambda$.
\end{rem}


We conclude this section with an overview of the geometry and dynamics in $B_3$. Here we consider system \eqref{eq:fund_prob_kappa} in parameter chart $\mathcal P_2$, i.e.~we consider system \eqref{eq:fund_prob_P2}, in the dual singular limit $(\tilde \eps_1, \eps_2) \to (0, 0)$. The loss of smoothness along $\Sigma_1 \cup \Sigma_2$ when $\eps_2 = 0$ can be resolved using the same sequence of transformations described in the overview of the problem in $B_2$ if one starts instead with the extended system obtained by appending $\eps_2' = 0$ to system \eqref{eq:fund_prob_P2}; the same geometry and dynamics is obtained in $\mathcal P_1$ as long as $\tilde \eps_1 > 0$ is fixed. Additional loss of smoothness along $x = \theta_1$ occurs when $\tilde \eps_1 \to 0$, leading to the PWS dynamics on the upper/lower cylinders and the blow-up sphere. The PWS dynamics is sketched for three different values of $\kappa_1$ in Figure \ref{fig:B3_PWS}, where the new switching manifold is denoted by $\Sigma_2^1$. The figure illustrates an important feature of the PWS dynamics on the sphere, namely, the presence of a visible-invisible two-fold singularity when $\kappa_1 = 1/2$, which facilitates a transition from (apparent) unstable sliding on the lower blow-up cylinder to (apparent) stable sliding on the top blow-up cylinder. The PWS dynamics on the blow-up sphere itself depends on the relative orientation of the visible and invisible fold points, which themselves depend upon the sign of $\kappa_1 - 1/2$.

\begin{figure}[t]
    \centering
    \begin{subfigure}[t]{0.3\textwidth}
        \centering
        \includegraphics[width=0.99\linewidth]{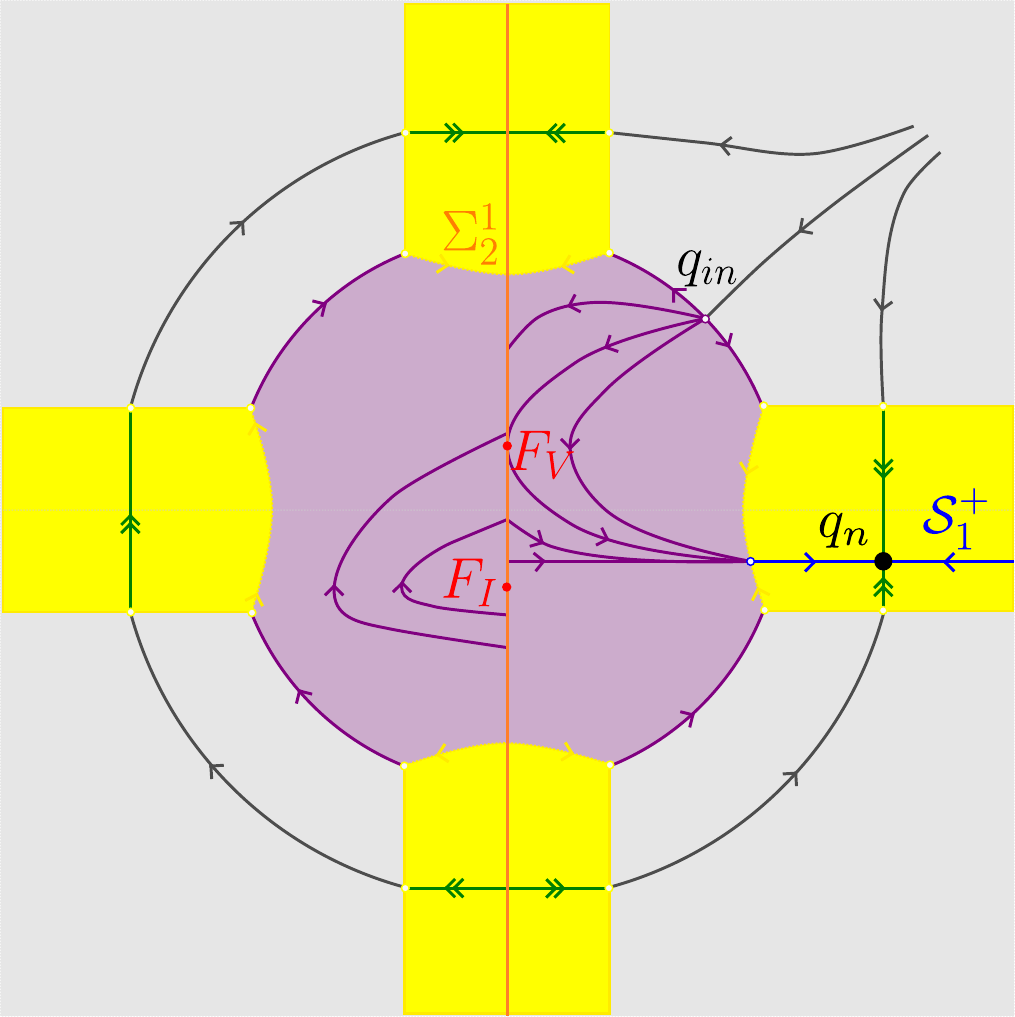}
    \caption{$\kappa_1<1/2$}
    \label{fig:B3_PWS_kappa1<0,5}
    \end{subfigure}%
\hfill
    \begin{subfigure}[t]{0.3\textwidth}
        \centering
        \includegraphics[width=0.99\linewidth]{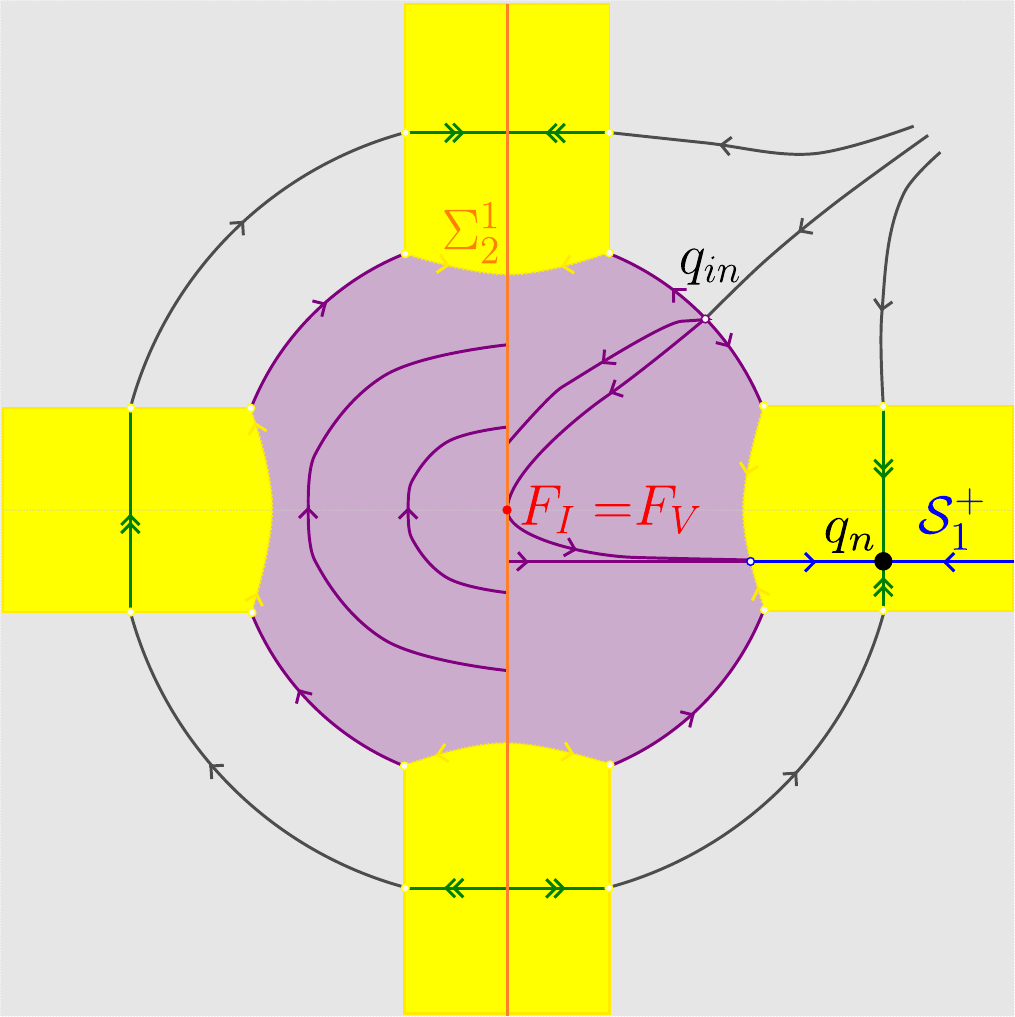}
    \caption{$\kappa_1=1/2$}
    \label{fig:B3_PWS_kappa1=0,5}
    \end{subfigure}
\hfill
    \begin{subfigure}[t]{0.3\textwidth}
        \centering
        \includegraphics[width=0.99\linewidth]{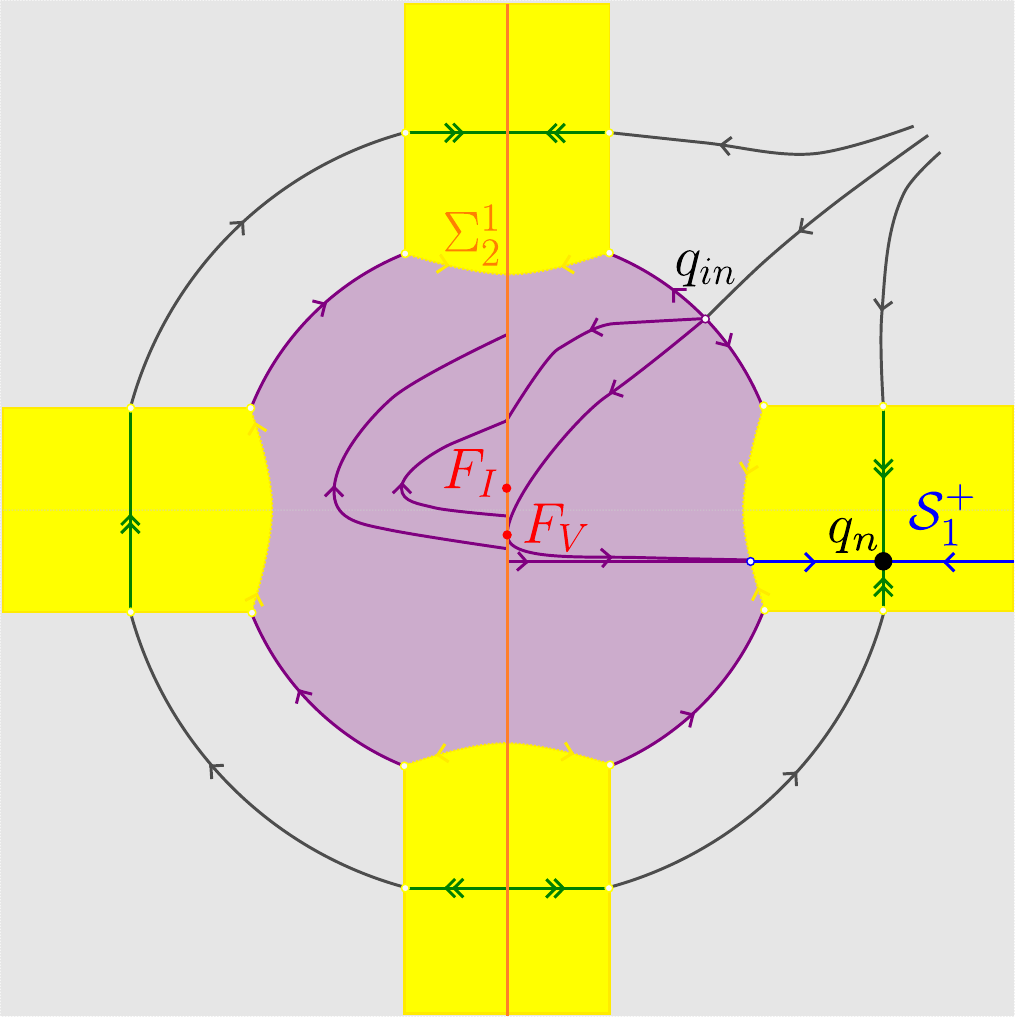}
    \caption{$\kappa_1>1/2$}
    \label{fig:B3_PWS_kappa1>0,5}
    \end{subfigure}
    \caption{PWS geometry and dynamics in the $B_3$ singular limit $(\tilde \eps_1, \eps_2) \to (0,0)$. The switching manifold $\Sigma_2^1$ (orange) dissects the phase space vertically. On the upper cylinder there is attracting sliding and on the lower cylinder there is repelling sliding. The change from stable to unstable sliding is mediated by the presence of a visible-invisible two-fold singularity which is correlated with the coincidence of a visible and an invisible fold point (shown in red and denoted by $F_V$ and $F_I$ respectively) when $\kappa_1 = 1/2$. The figure illustrates the unfolding of the two-fold singularity with (a) $\kappa_1 < 1/2$, (b) $\kappa_1 = 1/2$ and (c) $\kappa_1 > 1/2$.}
    \label{fig:B3_PWS}
\end{figure}

Three more cylindrical blow-ups are needed to resolve the loss of smoothness along $\Sigma_2^1$ when $\tilde \eps_1 = 0$; one for each branch on the upper/lower cylinders, and one to resolve the part of the switching manifold which lies on the sphere. This leads to the blow-up geometry shown in Figure \ref{fig:B3_kappa1_geometry}. Attracting and repelling normally hyperbolic critical manifolds, which we denote again by $\mathcal S^+_2$ and $\mathcal S^-_2$ respectively, are identified on top of the cylinders. This aligns with expectations from the PWS analysis, which suggested stable and unstable sliding in these regions respectively. A direct analysis reveals that the reduced flow along $\mathcal S^+_2$ and $\mathcal S^-_2$ is oriented towards the inner blow-up cylinder which lies on top of the sphere.

The features described above are present for all $(\kappa_1, \kappa_2) \in \Lambda$, however the geometry and slow-fast dynamics in the central blow-up (i.e.~on the cylinder on top of the sphere), differs significantly in different regions of the $(\kappa_1, \kappa_2)$-plane. 

\begin{figure}[t]
    \centering
    \begin{subfigure}[t]{0.3\textwidth}
        \centering
            \includegraphics[width=1\linewidth]{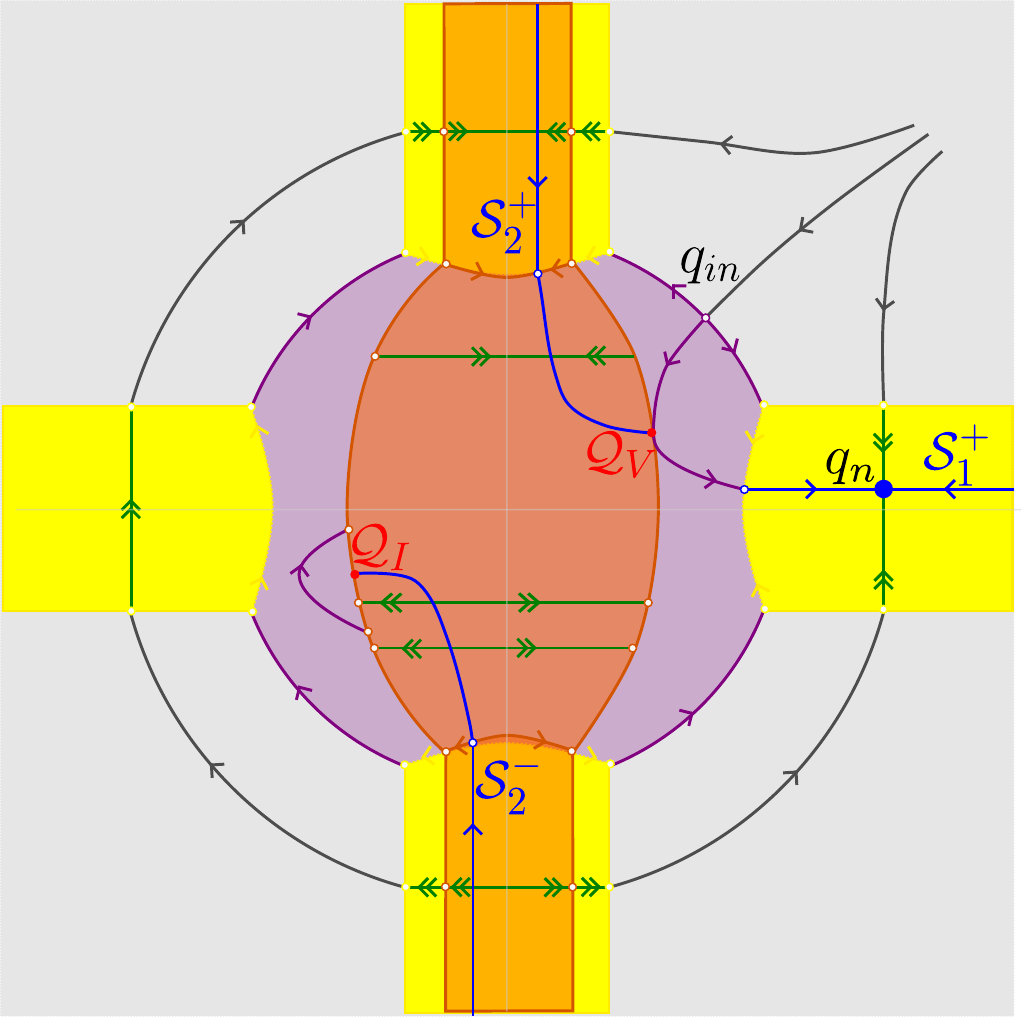}
    \caption{$\kappa_1<1/2$}
    \label{fig:B3_kappa1_less_0,5}
    \end{subfigure}%
    \hfill
    \begin{subfigure}[t]{0.3\textwidth}
        \centering
            \includegraphics[width=1\linewidth]{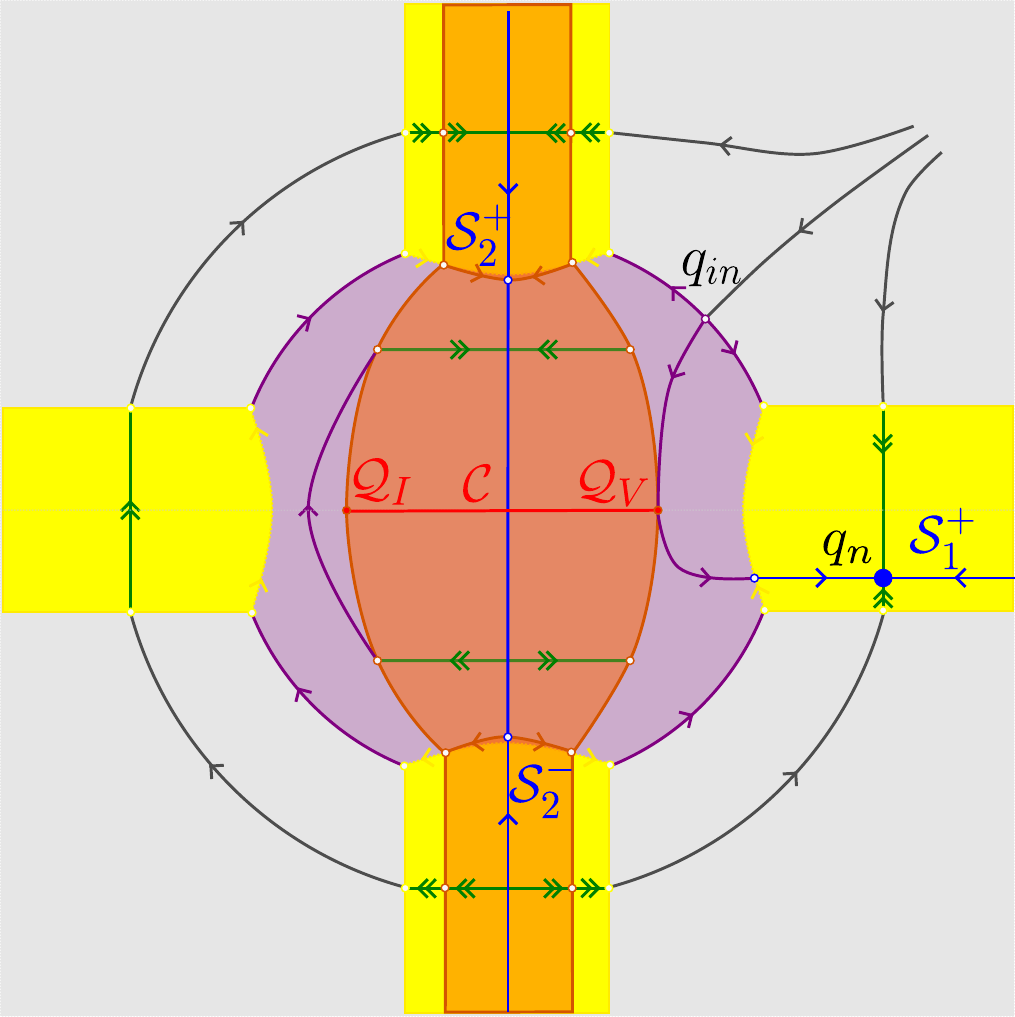}
    \caption{$\kappa_1=1/2$.}
    \label{fig:B3_kappa1=1/2}
    \end{subfigure}
    \hfill
    \begin{subfigure}[t]{0.3\textwidth}
        \centering
        \includegraphics[width=1\linewidth]{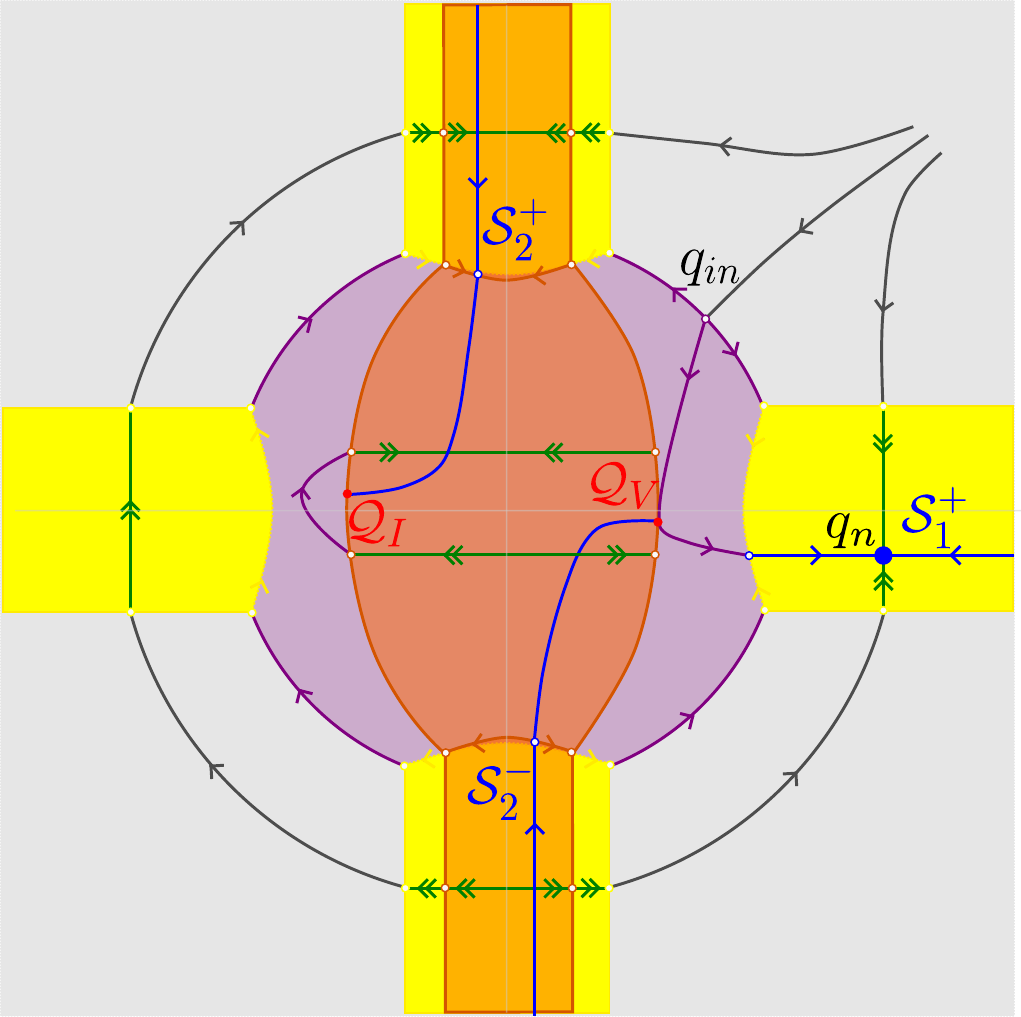}
    \caption{$\kappa_1>1/2$}
    \label{fig:B3_kappa1_greater_0,5}
    \end{subfigure}
    \caption{Singular geometry and dynamics after cylindrical blow-up of the switching manifold $\Sigma_2^1$ on the blow-up sphere and the upper/lower cylinders, for (a) $\kappa_1 < 1/2$, (b) $\kappa_1 = 1/2$ and (c) $\kappa_1 > 1/2$ as in Figure \ref{fig:B3_PWS}. The relative positioning of the critical manifolds $\mathcal S_2^\pm$, which `extend' to the critical manifolds shown in blue on the central blow-up cylinder, changes as $\kappa_1$ varies over the critical value $\kappa_1 = 1/2$ associated with the two-fold singularity in Figure \ref{fig:B3_PWS}. In (b) we show the intersection of these critical manifolds when $\kappa_1 = 1/2$, as well as the presence of an additional critical manifold $\mathcal C$, shown here in red, which is degenerate (i.e.~not normally hyperbolic). The same geometry was recently identified and analysed in \cite{Kristiansen_2023}.}
    \label{fig:B3_kappa1_geometry}
\end{figure}

Figure \ref{fig:B3_kappa1_geometry} shows three different possibilities depending on whether $\kappa_1 < 1/2$, $\kappa_1 = 1/2$ or $\kappa_1 > 1/2$, and should be compared with Figure \ref{fig:B3_PWS} (these are the cases already distinguished by the PWS dynamics above). When $\kappa_1 < 1/2$, a repelling critical manifold connects the endpoint of $\mathcal S_2^-$ with a degenerate point $\mathcal Q_I$ associated with the invisible tangency from the left, and an attracting critical manifold connects the endpoint of $\mathcal S_2^+$ with a degenerate point $\mathcal Q_V$ associated with the visible tangency from the right. The converse orientation applies when $\kappa_1 > 1/2$; here the attracting critical manifold connects to $\mathcal Q_I$, and the repelling critical manifold connects to $\mathcal Q_V$. This transition is mediated by a very degenerate situation when $\kappa_1 = 1/2$, which is shown in Figure \ref{fig:B3_kappa1=1/2}. There are two critical manifolds in this case. The attracting and repelling critical manifolds $\mathcal S_2^\pm$ extend up to the origin, where they intersect another critical manifold which lies along the horizontal axis. This critical manifold is shown in red and denoted by $\mathcal C$ in Figure \ref{fig:B3_kappa1=1/2}, and is degenerate (the non-trivial eigenvalue associated with its linearisation is identically zero). Consideration of the reduced flow on $\mathcal S_2^\pm$ reveals another important feature when $\kappa_1 = 1/2$, namely, the presence of a stable node $q_{f/n} \in \mathcal S_2^+$ when $\kappa_2 \in (1/2, 3/4)$, vs a saddle $q_s \in \mathcal S_2^-$ when $\kappa_2 \in (3/4,1)$. 
\begin{figure}[]
    \centering
    \begin{subfigure}[t]{0.3\textwidth}
        \centering
            \includegraphics[width=1\linewidth]{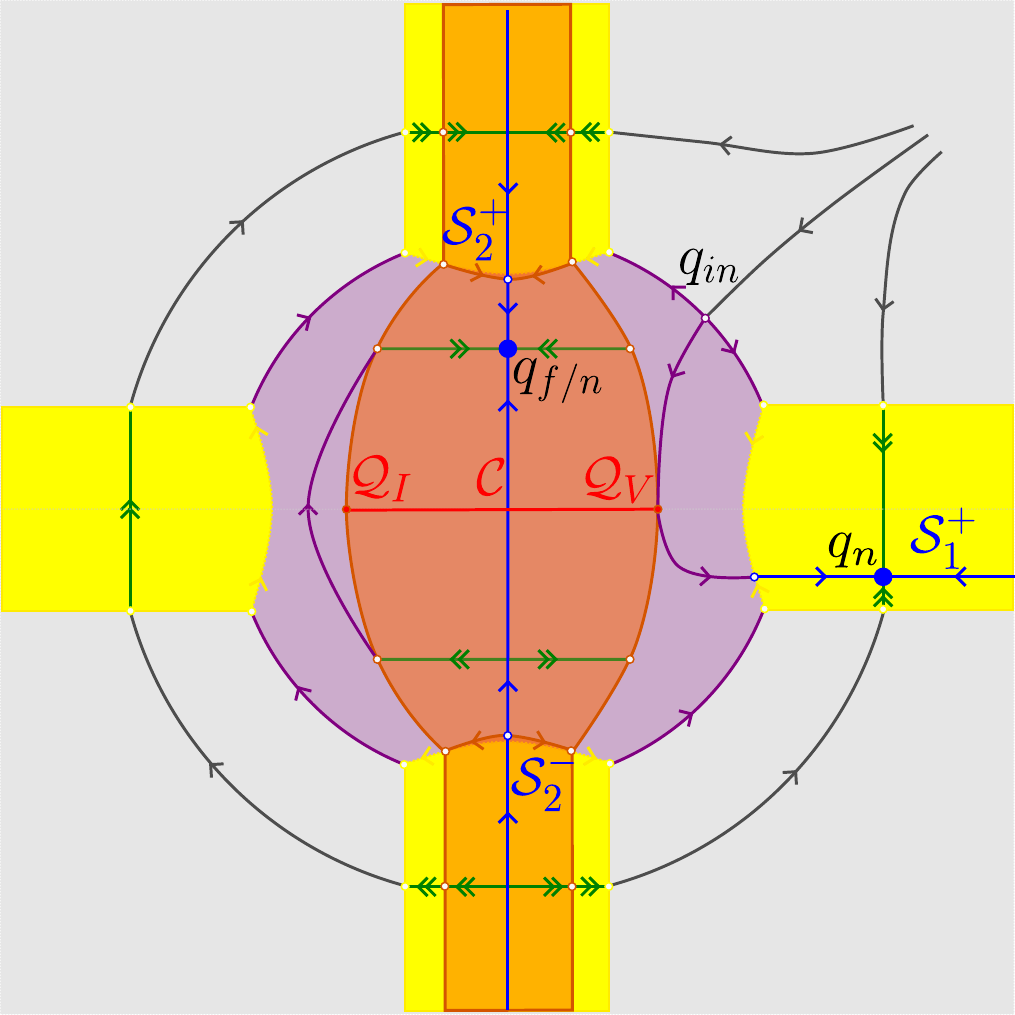}
    \caption{$\kappa_1=1/2$, $\kappa_2<3/4$}
    \label{fig:B3_kappa1_0.5_kappa2_less_0.5}
    \end{subfigure}%
    \hfill
    \begin{subfigure}[t]{0.3\textwidth}
        \centering
            \includegraphics[width=1\linewidth]{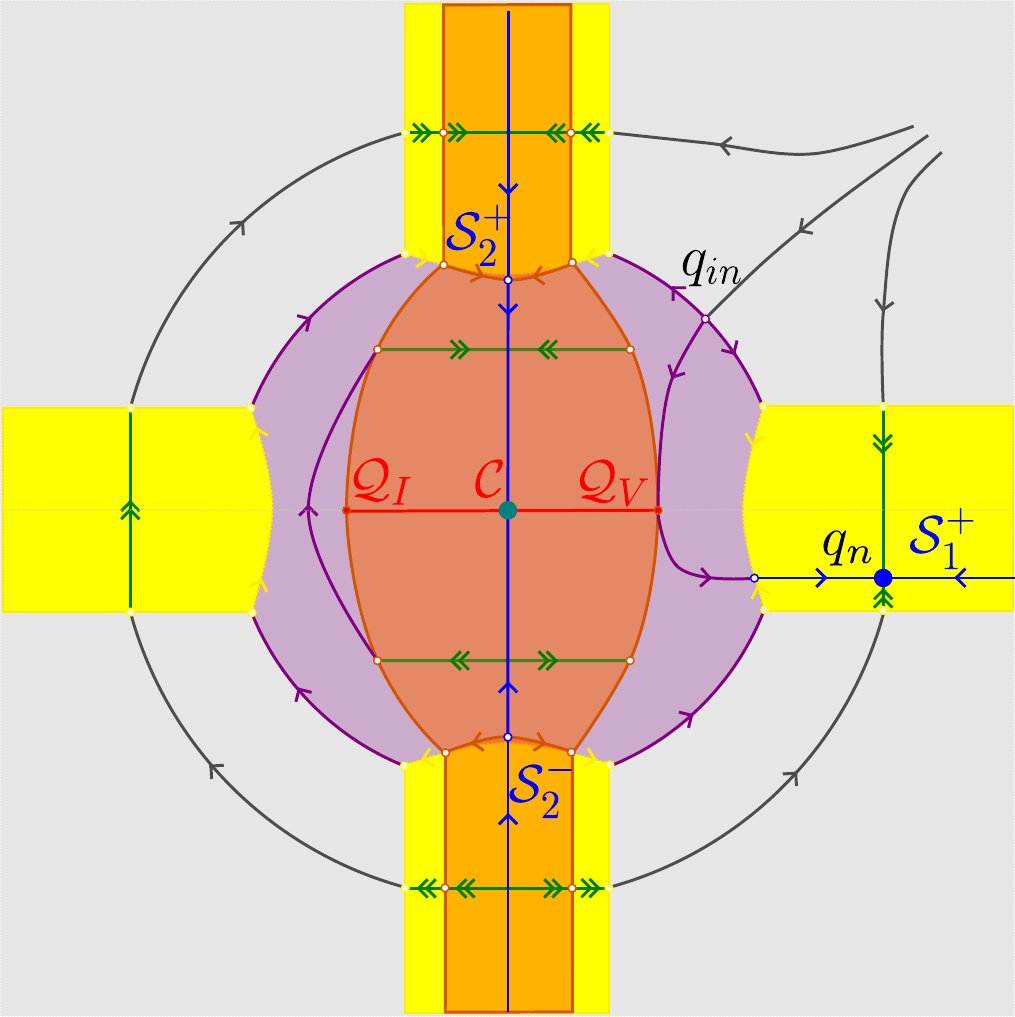}
    \caption{$\kappa_1=1/2$, $\kappa_2=3/4$}
    \label{fig:B3_kappa1_0.5_kappa2=0.5}
    \end{subfigure}
    \hfill
    \begin{subfigure}[t]{0.3\textwidth}
        \centering
        \includegraphics[width=1\linewidth]{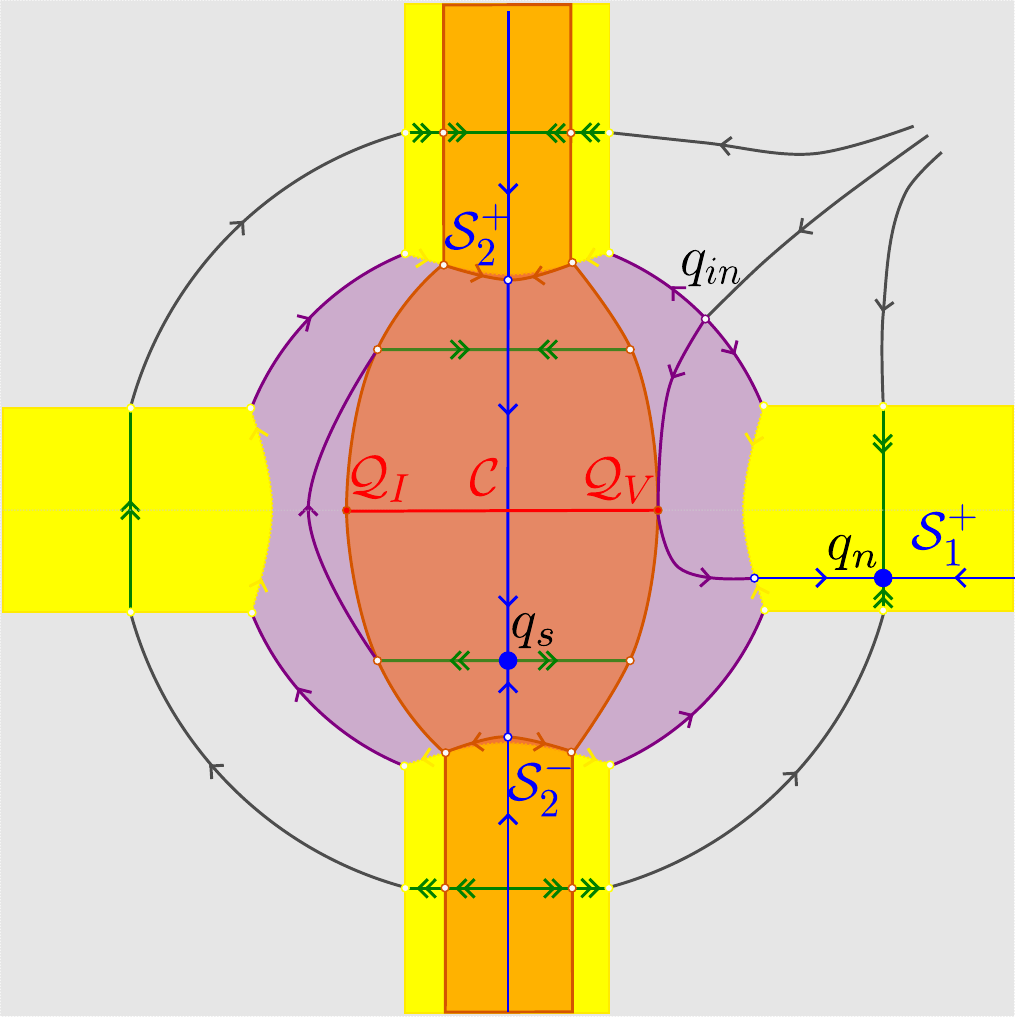}
    \caption{$\kappa_1=1/2$, $\kappa_2>3/4$}
    \label{fig:B3_kappa1_0.5_kappa2_greater_0.5}
    \end{subfigure}
    \caption{Singular geometry and dynamics when $\kappa_1 = 1/2$ (cf.~Figure \ref{fig:B3_kappa1=1/2} above) when (a) $\kappa_2 < 3/4$, (b) $\kappa_2 = 3/4$ and (c) $\kappa_2 > 3/4$. When $\kappa_2 < 3/4$ there is a stable equilibrium $q_{f/n}$ on $\mathcal S_2^+$ (above $\mathcal C$). When $\kappa_2 > 3/4$, there is a saddle equilibrium $q_s$ on $\mathcal S_2^-$ (below $\mathcal C$). In each case, the third equilibrium in the system is `lost' in the sense that it lies within $\mathcal C$. When $\kappa_2 = 3/4$, i.e.~in (b), both $q_{f/n}$ and $q_s$ (if they exist) are expected to lie at the origin, i.e.~at the intersection of $\mathcal S_2^-$, $\mathcal C$ and $\mathcal S_2^+$.}
    \label{fig:B3_kappa1_0.5_kappa2_varied}
\end{figure}
This feature, which we illustrate in Figure \ref{fig:B3_kappa1_0.5_kappa2_varied}, suggests an additional degeneracy when $\kappa_2 = 3/4$. This motivates an additional blow-up of the point $(\kappa_1, \kappa_2, \tilde \eps_1) = (1/2, 3/4, 0)$ in parameter space, followed by a spherical blow-up in variable space. These two blow-ups lead to the blow-up sphere in the bifurcation set in Figure \ref{fig:B3_bifurcation_diagram} and the blown-up geometry in phase space sketched in Figure \ref{fig:BT_phase_blow-up_geometry} respectively. The results in Theorem \ref{thm:B3_BT} follow from the identification of a regular Bogdanov-Takens bifurcation which occurs on this final blow-up sphere in variable space, and on the blow-up sphere in the $(\kappa_1, \kappa_2, \tilde \eps_1)$-parameter space which is shown in Figure \ref{fig:B3_bifurcation_diagram_zoom}.

\begin{figure}[h!]
    \centering
    \includegraphics[width=0.5\linewidth]{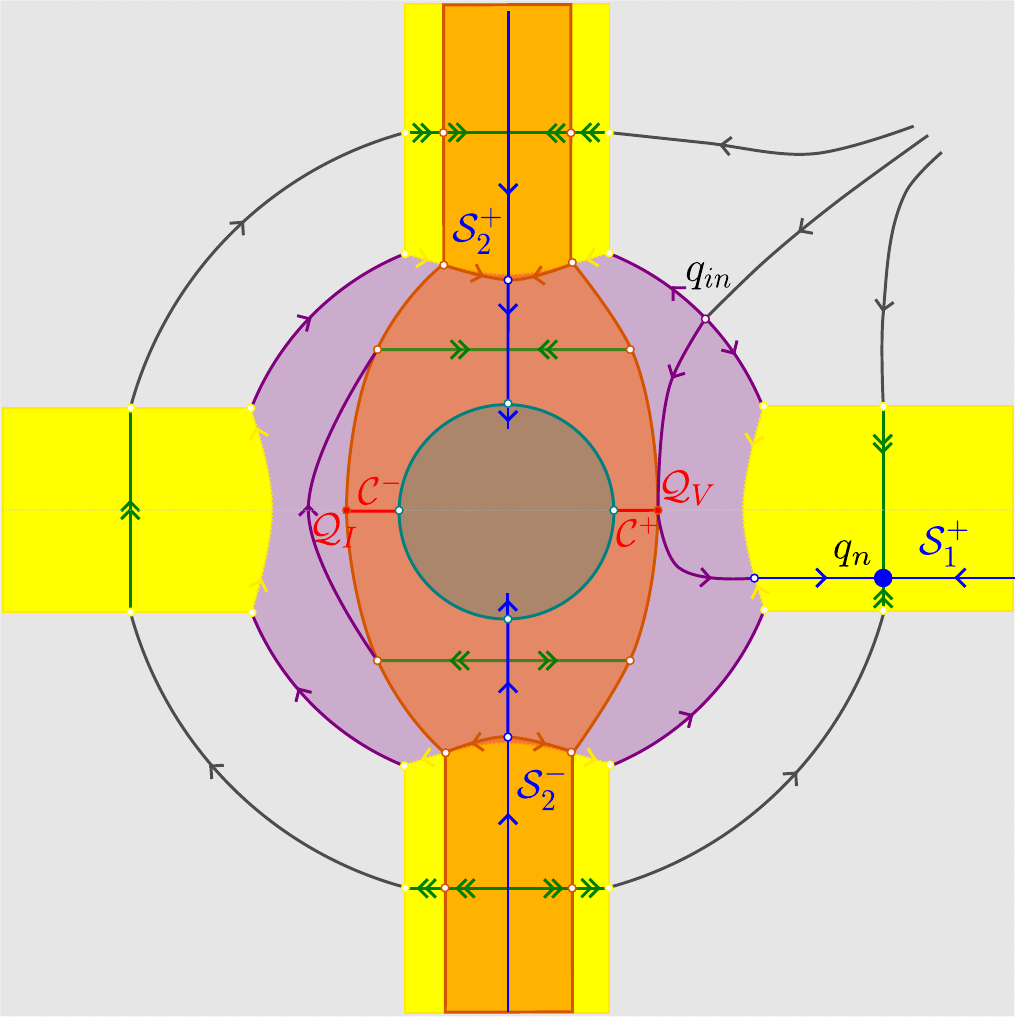}
    \caption{Blow-up geometry in phase space for $(\kappa_1,\kappa_2)=(1/2,3/4)$ after a spherical blow-up of the degenerate point (teal) which lies at the intersection of $\mathcal C$, $\mathcal S^-_2$ and $\mathcal S_2^+$ in Figure \ref{fig:B3_kappa1_0.5_kappa2=0.5}. The dynamics on the blow-up sphere, shown here in blue, depends on the local parameter coordinates obtained from the parameter blow-up close to $(\kappa_1,\kappa_2)=(1/2,3/4)$ and resembles the phase portraits of the classical Bogdanov-Takens bifurcation.}
    \label{fig:BT_phase_blow-up_geometry}
\end{figure}

\begin{figure}[h!]
        \centering
        \includegraphics[width=0.5\linewidth]{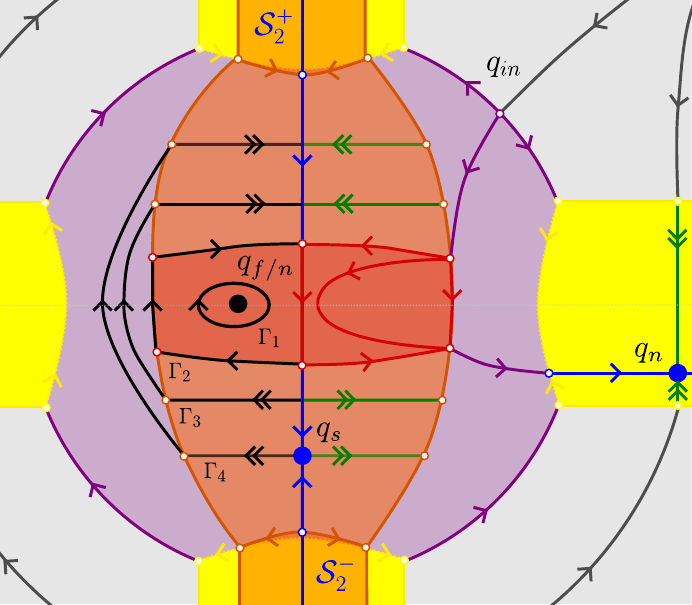}
    \caption{Fully resolved geometry and dynamics in the $B_3$ singular limit $(\tilde \eps_1, \eps_2) \to (0, 0)$ when $\kappa_1 = 1/2$ and $\kappa_2 \in (3/4, 1)$; cf Figure \ref{fig:B3_kappa1_0.5_kappa2_greater_0.5}. Cylindrical blow-up of $\mathcal C$ leads to the innermost `horizontal' blow-up cylinder in red. The equilibrium $q_{f/n}$ is a nonlinear center when $\kappa_1 = 1/2$ and $\kappa_2 \in (3/4, 1)$, and there is a singular canard-like intersection along the vertical axis. The left-hand side of the innermost blow-up cylinder is foliated by periodic orbits; a representative is shown and denoted here by $\Gamma_1$. These connect to the heteroclinic orbit denoted by $\Gamma_2$, which is the smallest of a continuous family of singular canard cycles; a representative singular canard cycle is indicated by $\Gamma_3$. The largest of these cycles is a singular canard homoclinic cycle which connects to the saddle $q_s$; this is denoted by $\Gamma_4$. The overall geometry strongly suggests the presence of an incomplete canard explosion for $0 < \tilde \eps_1, \eps_2 \ll 1$.}
    \label{fig:B3_singular_canard_1}
\end{figure}

\begin{rem}
    The red lines $\mathcal C^\pm$ which connect the points $\mathcal Q_I$ and $\mathcal Q_V$ to the blow-up sphere in Figure \ref{fig:BT_phase_blow-up_geometry} are still degenerate, because the spherical blow-up described above is insufficient to resolve the degeneracy along all of $\mathcal C$. We have chosen not to include the additional blow-ups that are necessary to resolve these degeneracies because they do not play an important role in the proof or description of our main results.
\end{rem}

In order to prove Theorem \ref{thm:B3}, we need to resolve one last degeneracy in both parameter and variable space which is associated with the scaling regime (S3). This degeneracy occurs for $\kappa_1=1/2$ and $\kappa_2 \in (3/4, 1)$ when $\tilde \eps_1 = 0$, i.e.~along the degenerate line which emanates from the top of the blow-up sphere in Figure \ref{fig:B3_bifurcation_diagram_sphere_only}. 
In phase space, it correlates with the presence of the degenerate critical manifold $\mathcal C$ when $\kappa_2 \in (3/4, 1)$, i.e.~in Figure \ref{fig:B3_kappa1_0.5_kappa2_greater_0.5}. This is resolved in a pair of cylindrical blow-ups in both parameter and variable space, leading to the blown-up bifurcation diagram in Figure \ref{fig:B3_bifurcation_diagram} and the singular geometry and dynamics shown in Figure \ref{fig:B3_singular_canard_1} respectively. Figure \ref{fig:B3_singular_canard_1} shows the dynamics at the critical value $\kappa_1 = 1/2$. Here, the planar system governing the dynamics on the final blow-up cylinder (shown in red in the figure) has a first integral, a nonlinear center, periodic orbits and an orbit which connects the endpoints of the attracting and repelling critical manifolds $\mathcal S^\pm_2$. This is very reminiscent of the geometry and dynamics associated with canard points in planar systems \cite{Maesschalck_2021,dumortier1996canard,Krupa_2001_Canard}. Indeed, the proofs for Assertions (ii)-(iii) in Theorem \ref{thm:B3} given in later sections will rely directly on existing results that characterise the close relationship between regularised visible-invisible tangencies and canards in slow-fast systems \cite{Kristiansen_2023}. Like the existing results on canards and associated phenomena near regularised two-fold singularities considered in \cite{Bonet2018,Kristiansen_2015,Kristiansen_2023}, these are obtained using adaptations of the methods developed in e.g.~\cite{Maesschalck_2021,dumortier1996canard,Krupa_2001_Canard}. Although we do not prove it in this work, the singular limit analysis which leads to Figure \ref{fig:B3_singular_canard_1} provides a limiting mechanism for the growth of the small amplitude oscillations that are born in the subcritical singular Hopf bifurcation at $q_{f/n}$ under parameter variation. 
We conjecture that these cycles undergo an incomplete canard explosion in which the Hopf cycles (such as $\Gamma_1$ in the figure) transition to canard-type cycles which perturb from $\Gamma_2$ in the figure, 
after which there is an explosive growth of canard cycles mediated by a family of underlying singular canard cycles (a representative $\Gamma_3$ is shown in the figure) which terminates in a canard-homoclinic orbit which perturbs from a singular canard-homoclinic orbit (denoted $\Gamma_4$ in the figure). 
This is similar to the incomplete canard explosion treated in \cite{Wechselberger2015}. 
We have decided to omit the rigorous treatment of the explosion, for the sake of brevity in what is already an admittedly long manuscript.



\section{Region \texorpdfstring{$B_2$}{B2}} \label{sec:Region_B2}
In this section we present the geometric blow-up construction in region $B_2$, followed by our proof of Theorem \ref{thm:B2}. We start with the intermediate region $B_2$ for the same reasons that we decided to start our informal discussion in Section \ref{sub:main_results_blow-up} in $B_2$: a large amount of the blow-up structure 
established here can be built upon later in the $B_1$ and $B_3$ analyses.


\subsection{Blow-up, singular geometry and dynamics}
\label{sub:B2_geometry}

As in the statement of Theorem \ref{thm:B2}, we work in chart $\mathcal P_1$ and consider system \eqref{eq:fund_prob_P1} with $0 < \eps_1 \ll 1$ and $\tilde \eps_2 \in [\beta_1, \beta_2]$. In order to resolve the loss of smoothness along $\Sigma_1 \cup \Sigma_2$ when $\eps_1 \to 0$, we consider the extended system
\begin{equation}
\label{eq:P1_extended}
\begin{split}
    \Dot{x}&=H(x,\eps_1,\theta_1)+H(y,\eps_1 \tilde \eps_2,\theta_2)-2 H(x,\eps_1,\theta_1) H(y,\eps_1 \tilde \eps_2,\theta_2)-\kappa_1 \theta_1^{-1} x, \\
    \Dot{y}&=1-H(x,\eps_1,\theta_1) H(y,\eps_1 \tilde \eps_2,\theta_2)-\kappa_2\theta_2^{-1} y, \\
    \dot \eps_1 &= 0 ,
\end{split}
\end{equation}
and apply a sequence of geometric blow-up transformations which resolve the degeneracy along $\Sigma_1 \cup \Sigma_2 \times \{0\}$. 
We begin by blowing-up the intersection point $I = \{(\theta_1, \theta_2, 0)\} = \Sigma_1 \cap \Sigma_2 \times \{0\}$ via
\begin{equation} \label{eq:sphere_exponential_blow-up}
    r \geq 0, \  (\bar x, \bar y, \bar \eps_1) \in \mathbb S^2 \mapsto
    \begin{cases}
        x = \theta_1e^{r \Bar{x}} , \\
        y = \theta_2e^{r \Bar{y}} , \\
        \eps_1 = r \Bar{\eps}_1 .
    \end{cases}
\end{equation}
Note that we have permitted a slight abuse of notation by reusing the variable $r$, which was briefly used and replaced by an alternative notation when defining the blow-up transformation in \eqref{eq:parameter_blow-up} above. In order to cover the entire region of interest, $5$ coordinate charts are needed:
\[
\mathcal K_1^\pm : \bar x = \pm 1, \qquad 
\mathcal K_2 : \bar \eps_1 = 1, \qquad 
\mathcal K_3^\pm : \bar y = \pm 1.
\]
In what follows we present the most important details in charts $\mathcal K_1^+$ and $\mathcal K_2$, and omit the majority of details associated with the other charts. Our motivations for these omissions are two-fold: (i) the analysis in charts $\mathcal K_1^-$ and $\mathcal K_3^\pm$ is similar to the analysis in $\mathcal K_1^+$, and (ii) only the $\mathcal K_2$ analysis is necessary for proving Theorem \ref{thm:B2}.

\begin{rem}
    The `usual approach' in blow-up analyses would be to set $x = \theta_1 + r \bar x$ and $y = \theta_2 + r \bar y$ instead of using the exponential transformations in \eqref{eq:sphere_exponential_blow-up}. We opt for the latter, which were introduced in \cite{Simon_thesis} and subsequently applied in \cite{Jelbart2026}, because they lead to simpler formulae for the Hill functions:
    \[
    z = \theta_i e^{r \bar z} \ \implies \ 
    H(z, \eps_1, \theta_i) = \frac{e^{\bar z}}{1 + e^{\bar z}} =: \hat H(\bar z) .
    \]
    Up to a constant scaling, the two approaches coincide to leading order since the exponential transformations satisfy $x = \theta_1 (1 + r \bar x) + \mathcal O(r^2)$ and $y = \theta_2(1 + r \bar y) + \mathcal O(r^2)$ as $r \to 0$.
\end{rem}

We start by considering the dynamics in the scaling chart $\mathcal K_2$. Transforming system \eqref{eq:fund_prob_P1} into local coordinates
$$\mathcal{K}_2 : \  x=\theta_1e^{r_2 x_2}, \qquad y = \theta_2e^{r_2y_2}, \qquad \eps_1 = r_2, $$
and writing $r_2 = \eps_1$ leads to
\begin{equation}\label{eq:B2_spherical_scaling_chart}
    \begin{aligned} 
    x_2'&=\frac{1}{\theta_1e^{\eps_1 x_2}}\left( \Hat{H}(x_2)+ \Hat{H} \left(\frac{y_2}{\Tilde{\eps}_2} \right) -2 \Hat{H}(x_2) \Hat{H} \left( \frac{y_2}{\Tilde{\eps}_2} \right) -\kappa_1 e^{\eps_1 x_2}\right) , \\
    y_2'&=\frac{1}{\theta_2e^{\eps_1 y_2}}\left( 1- \Hat{H}(x_2) \Hat{H} \left(\frac{y_2}{\Tilde{\eps}_2} \right) -\kappa_2 e^{\eps_1 y_2}\right) , 
\end{aligned}
\end{equation}
where $0 < \eps_1 \ll 1$ and $(\cdot)’$ denotes differentiation with respect to the new time scale $t_2 = t / \eps_1$.

Since $\tilde \eps_2 \in [\beta_1, \beta_2]$ in $B_2$, system \eqref{eq:B2_spherical_scaling_chart} is regularly perturbed on compact regions of the $(x_2, y_2)$-plane in the sense that the limiting system when $\eps_1 \to 0$, namely
\begin{equation}\label{eq:B2_dynamics_on_sphere}
    \begin{aligned} 
    x_2'&=\frac{1}{\theta_1}\left( \Hat{H}(x_2)+ \Hat{H} \left(\frac{y_2}{\Tilde{\eps}_2} \right) -2 \Hat{H}(x_2) \Hat{H} \left(\frac{y_2}{\Tilde{\eps}_2} \right) -\kappa_1 \right) , \\
    y_2'&=\frac{1}{\theta_2}\left( 1 - \Hat{H}(x_2) \Hat{H} \left( \frac{y_2}{\Tilde{\eps}_2} \right) -\kappa_2 \right) ,
\end{aligned}
\end{equation}
only has isolated equilibria (if any). A comprehensive analysis of the qualitative dynamics of system \eqref{eq:B2_spherical_scaling_chart} based on the dynamics of system \eqref{eq:B2_dynamics_on_sphere} and regular perturbation theory will be given in Section \ref{sec:B2_bifurcation_analysis} below, in order to prove Theorem \ref{thm:B2}. For now, we continue with the blow-up construction which resolves the remaining degeneracies in the remaining coordinate charts.

We turn to the entry and exit charts $\mathcal K_1^\pm$ and $\mathcal K_3^\pm$ in order to take care of the (still degenerate) branches of the switching manifolds $\Sigma_1$ and $\Sigma_2$ defined by
\begin{equation} \label{eq:Sigma_1_+_-}
    \Sigma_1^+=\{(x,y,0) \in \R^3:y=\theta_2, x> \theta_1 \}, \qquad \Sigma_1^-= \{(x,y,0) \in \R^3:y=\theta_2, x< \theta_1 \},
\end{equation}
and
\begin{equation} \label{eq:Sigma_2_+_-}
       \Sigma_2^+=\{(x,y,0) \in \R^3:x=\theta_1, y > \theta_2 \}, \qquad   \Sigma_2^-=\{(x,y,0) \in \R^3:x=\theta_1, y < \theta_2 \} ,
\end{equation}
respectively. We start in chart $\mathcal K_1^+$, where $\Sigma_1^+$ is visible, 
and work in local coordinates $(r_1, y_1, \eps_{11})$ defined by
$$\mathcal{K}_1^+: \ x = \theta_1 e^{r_1}, \quad y = \theta_2e^{r_1 y_1}, \quad \eps_1 = r_1 \eps_{11} .$$
Rewriting system \eqref{eq:P1_extended} in these coordinates and applying a desingularisation which amounts to a formal multiplication of the vector field by $r_1$, we obtain the system
\begin{equation}\label{eq:B2_spherical_sigma2_+}
    \begin{aligned} 
    r_1'&=r_1G\left(\frac{1}{\eps_{11}},\frac{y_1}{\eps_{11}\Tilde{\eps}_2}, r_1\right) , \\
    y_1'&=\frac{1}{\theta_2e^{r_1 y_1}}\left( 1-  \Hat{H}\left(\frac{1}{\eps_{11}}\right) \Hat{H}\left(\frac{y_1}{\eps_{11}\Tilde{\eps}_2}\right) -\kappa_2 e^{r_1y_1}\right)-y_1G\left(\frac{1}{\eps_{11}},\frac{y_1}{\eps_{11}\Tilde{\eps}_2}, r_1\right) , \\
    \eps_{11}'&=-\eps_{11}G\left(\frac{1}{\eps_{11}},\frac{y_1}{\eps_{11}\Tilde{\eps}_2}, r_1\right) , 
\end{aligned}
\end{equation}
where 
$$G\left(\frac{1}{\eps_{11}},\frac{y_1}{\eps_{11}\Tilde{\eps}_2}, r_1\right):=\frac{1}{\theta_1e^{r_1}}\left( \Hat{H}\left(\frac{1}{\eps_{11}}\right)+ \Hat{H}\left(\frac{y_1}{\eps_{11}\Tilde{\eps}_2}\right) -2 \Hat{H}\left(\frac{1}{\eps_{11}}\right) \Hat{H}\left(\frac{y_1}{\eps_{11}\Tilde{\eps}_2}\right) -\kappa_1 e^{r_1}\right) ,$$
and we permit an abuse of notation by allowing the prime to denote differentiation with respect to the new independent variable.

\begin{rem} \label{rem:desingularization_time_scale}
    In what follows we shall consistently abuse notation in this way. In particular, we shall denote differentiation with respect to `desingularised time-scales' with primes, and differentiation with respect to slow time-scales associated with reduced flow along critical manifolds with dots.
\end{rem}

Direct calculations using this system reveal the existence of a saddle equilibrium at $q_{in} : (0, \kappa_2 \theta_1 / \kappa_1 \theta_2, 0)$, which has a stable manifold $w^s(q_{in})$ contained within the invariant line $\{(r_1,\kappa_2\theta_1/\kappa_1\theta_2,0)\}$ as shown in Figure \ref{fig:B2_blow-up}. The most significant thing to note, however, is that the vector field induced by \eqref{eq:B2_spherical_sigma2_+} is still non-smooth along the line $\{(r_1,0,0) : r_1 \geq 0\}$ which corresponds to the switching manifold $\Sigma_1^+$ (extended up to the blow-up sphere $\{r_1 = 0\}$). We can resolve the loss of smoothness with a cylindrical blow-up of the form
\begin{equation} \label{eq:B2_Sigma_2^+_blow-up_transformation}
    \rho \geq 0, \ (\Bar{y}_1,\Bar{\eps}_1) \in \mathbb{S}^1  \mapsto 
    \begin{cases}
        y_1 = \rho \Bar{y}_1, \\
        \eps_{11} = \rho \Bar{\eps}_{11}.
    \end{cases}
\end{equation}
The relevant region of phase space can be covered by local coordinate charts corresponding to $\bar \eps_{11} = 1$ and $\bar y_1 = \pm 1$. Here we restrict attention to the geometry and dynamics in $\mathfrak K_2: \bar \eps_{11} = 1$, where local coordinates may be denoted by
\begin{equation}
\label{eq:mathfrak_K2_coordinates}
\mathfrak K_2: \ y_1 = \rho_2 y_{12} , \qquad 
\eps_{11} = \rho_2 .
\end{equation}
Rewriting system \eqref{eq:B2_spherical_sigma2_+} in these coordinates and applying a desingularisation which amounts to formal multiplication by $\rho_2$ leads to the following system in $\bar \eps_{11} = 1$:
%
\begin{equation}\label{eq:B2_spherical_sigma2_+_cylinder_scaling}
    \begin{aligned} 
    r_1'&=r_1\rho_2 G\left(\frac{1}{\rho_2},\frac{y_{12}}{\Tilde{\eps}_2}, r_1\right), \\
    y_{12}'&=\frac{1}{\theta_2e^{r_1\rho_2 y_{12}}}\left( 1-  \Hat{H}\left(\frac{1}{\rho_2}\right) \Hat{H}\left(\frac{y_{12}}{\Tilde{\eps}_2}\right) -\kappa_2 e^{r_1\rho_2 y_{12}}\right) , \\
    \rho_2'&=-\rho_2^2G\left(\frac{1}{\rho_2},\frac{y_{12}}{\Tilde{\eps}_2}, r_1\right) .
\end{aligned}
\end{equation}
The key singular limit features associated with system \eqref{eq:B2_spherical_sigma2_+_cylinder_scaling} are summarised in the following result.

\begin{lem} \label{lem:B2_S_2^+}
    Consider system \eqref{eq:B2_spherical_sigma2_+_cylinder_scaling} with $(\kappa_1, \kappa_2) \in \Lambda$ and $\tilde \eps_2 \in [\beta_1, \beta_2]$. The following assertions are true: 
    \begin{enumerate}
        \item[(i)] The line
        \[
        \mathcal{S}_1^+ := \left\{ \left(r_1, \Tilde{\eps}_2 \ln \frac{1-\kappa_2}{\kappa_2}, 0\right) : r_1 \geq 0 \right\} 
        \]
        defines a critical manifold which is normally hyperbolic and attracting when considered as a critical manifold for the planar system which governs the flow in $\{ \rho_2 = 0 \}$.
        \item[(ii)] The reduced flow on $\mathcal S_1^+$ has a unique attracting equilibrium
        $$q_n : 
        \left(\ln\left(\frac{\kappa_2}{\kappa_1}\right),\Tilde{\eps}_2 \ln \frac{1-\kappa_2}{\kappa_2},0\right).$$
    \end{enumerate}
\end{lem}

\begin{proof}
The existence of the critical manifold $\mathcal S_1^+$ follows from direct calculations. Note that the expression which defines $\mathcal S_1^+$ is well-defined because $(\kappa_1, \kappa_2) \in \Lambda \implies \kappa_2 \in (0,1)$. 
The fact that $\mathcal S_1^+$ is normally hyperbolic and attracting within $\{ \rho_2 = 0 \}$ (i.e.~on the blow-up cylinder) follows from the fact that the non-trivial eigenvalue associated with the linearisation along $\mathcal S_1^+$ is 
$$\lambda=-\frac{1}{\Tilde{\eps}_2}\Hat{H}\left(\frac{y_{12}}{\Tilde{\eps}_2}\right)\left( 1-\Hat{H}\left(\frac{y_{12}}{\Tilde{\eps}_2}\right)\right)\bigg|_{\mathcal{S}_1^+}=\frac{(\kappa_2-1)\kappa_2}{\Tilde{\eps}_2}<0.$$

In order to prove Assertion (ii), we need to determine the reduced problem on $\mathcal S_1^+$. This can be done
by using the conserved quantity $\eps_1 = r_1 \rho_2$ 
to eliminate $\rho_2$ and rewrite \eqref{eq:B2_spherical_sigma2_+_cylinder_scaling} in the following form:
\begin{equation}\label{eq:B2_spherical_sigma2_+_cylinder_away_from_sphere}
    \begin{aligned} 
    r_1'&=\eps_1 G\left(\frac{r_1}{\eps_1},\frac{y_{12}}{\Tilde{\eps}_2}, r_1\right) , \\
    y_{12}'&=\frac{1}{\theta_2e^{\eps_1 y_{12}}}\left( 1-  \Hat{H}\left(\frac{r_1}{\eps_1}\right) \Hat{H}\left(\frac{y_{12}}{\Tilde{\eps}_2}\right) -\kappa_2 e^{\eps_1 y_{12}}\right) , 
\end{aligned}
\end{equation}
which is a planar slow-fast system with perturbation parameter $0 < \eps_1 \ll 1$. Note that this is only possible when $\rho_2 = \eps_1 / r_1$ makes sense, i.e.~as long as we restrict to regions with $r_1 > 0$ that are bounded away from the blow-up sphere at $\{r_1 = 0\}$. 
Since system \eqref{eq:B2_spherical_sigma2_+_cylinder_scaling} is slow-fast in standard form, the reduced problem on $\mathcal S_1^+$ can be obtained in the `usual manner' and shown to be given by
$$\Dot r_1 = \frac{1}{\theta_2 e^{r_1}}(\kappa_2-\kappa_1 e^{r_1}) , $$
when $r_1 > 0$ (the dot denotes differentiation with respect to slow time). This equation has a unique, stable equilibrium which corresponds precisely to the point $q_n$ referred to in Assertion (iii). 
\end{proof}

Lemma \ref{lem:B2_S_2^+} justifies that part of Figure \ref{fig:B2_blow-up_2} which shows the singular geometry and dynamics associated with the blow-up of the sphere and right-most cylinder.

\begin{rem} \label{rem:center_manifold_connection}
    Figure \ref{fig:B2_blow-up_2} also indicates local flow on a local 1-dimensional center manifold which `extends' $\mathcal S_1^+$ onto the blow-up sphere. The existence of a 2-dimensional local center-stable manifold follows from Assertion (ii) of Lemma \ref{lem:B2_S_2^+} and the center manifold theorem, and restriction of this center manifold to $\{r_1 = 0\}$ yields the 1-dimensional center manifold shown in the figure (the orientation of the flow is confirmed via standard computations). The proof is (i) similar to the entry chart analysis in e.g.~\cite{Krupa_2001_Extend,krupa_extending_transcritical}, and (ii) not sufficiently relevant to the analysis which follows in this work, so we have opted to omit it.
\end{rem}

The degeneracies associated with the loss of smoothness along $\Sigma_1^-$ and $\Sigma_2^\pm$ can be resolved using homogeneous cylindrical blow-up transformations analogous to \eqref{eq:B2_Sigma_2^+_blow-up_transformation} in charts $\mathcal K_1^-$ and $\mathcal K_2^\pm$ respectively. These analyses, which we omit for the sake of brevity, reveal the existence of an attracting (repelling) normally hyperbolic critical manifold, which are sketched and denoted by $\mathcal S_2^+$ ($\mathcal S_2^-$) in Figure \ref{fig:B2_blow-up}. Direct calculations which mirror those applied in the proof of Lemma \ref{lem:B2_S_2^+} Assertion (iii) show that the reduced flow along each critical manifold is oriented towards the blow-up sphere for any fixed $(\kappa_1, \kappa_2) \in \Lambda$ and $\tilde \eps_2 \in [\beta_1, \beta_2]$, and that this flow can be locally extended using center manifold theory via the same type of argument that is summarised in Remark \ref{rem:center_manifold_connection} above. All together, this analysis leads to the blown-up geometry shown in Figure \ref{fig:B2_blow-up_2}, which is fully resolved in the sense that there are no additional points or sets which need to be blown-up.

\subsection{Dynamics on the blow-up sphere and the proof of Theorem \ref{thm:B2}}
\label{sec:B2_bifurcation_analysis}

In order to prove Theorem \ref{thm:B2} it suffices to work with system \eqref{eq:B2_spherical_scaling_chart} with $0 < \eps_1 \ll 1$ on compact subsets of $\R^2$. For our purposes, it is easier to perform the smooth change of variables 
    \begin{align} \label{B2_variable_change_BT}
        X = \Hat{H}(x_2), \qquad Y = \Hat{H} \left( \frac{y_2}{\Tilde{\eps}_2} \right),
    \end{align}
    and consider the system
    \begin{equation}
        \label{eq:B2_SN_equation_proof_perturbed}
        \begin{pmatrix}
        X' \\
        Y'
    \end{pmatrix}
    = F(X,Y,\kappa_1,\kappa_2) + \eps_1 F_{rem}(X,Y,\kappa_1,\kappa_2,\eps_1)=:\hat F(X,Y,\kappa_1,\kappa_2,\eps_1)
    \end{equation}
    where
    \[
    F(X,Y,\kappa_1,\kappa_2) := 
    \begin{pmatrix}
        F_1(X,Y,\kappa_1) \\
        F_2(X,Y,\kappa_2)
    \end{pmatrix}
    =
    \begin{pmatrix}
        \frac{1}{\theta_1}X(1-X)(X+Y-2XY-\kappa_1) \\
        \frac{1}{\theta_2 \Tilde{\eps}_2}Y(1-Y)(1-XY-\kappa_2)
    \end{pmatrix}. 
    \]
    The remainder $F_{rem}(X,Y,\kappa_1,\kappa_2,\eps_1)$ is smooth and $\mathcal O(1)$ as $\eps_1 \to 0$ due to the fact that $e^{\eps_1 x_2} = 1 + \mathcal O(\eps_1)$ and $e^{\eps_1 y_2} = 1 + \mathcal O(\eps_1)$ as $\eps_1 \to 0$ on compact subsets of $(X,Y) \in (0,1)^2$ (which correspond to compact subsets of $\R^2$ in the $(x_2, y_2)$-plane). The advantage of working with system \eqref{eq:B2_SN__equation_proof} is that the limiting problem when $\eps_1 = 0$, namely
    \begin{equation}
    \label{eq:B2_SN__equation_proof}
       \begin{pmatrix}
            X' \\
            Y' 
      \end{pmatrix}
      = F(X,Y,\kappa_1,\kappa_2) ,
    \end{equation}
    is \textit{polynomial}. In the previous section, we claimed that system \eqref{eq:B2_spherical_scaling_chart} is a regular perturbation of the limiting system \eqref{eq:B2_dynamics_on_sphere} obtained for $\eps_1 \to 0$ when considered on compact domains in $\R^2$, in the sense that the latter system has only isolated equilibria. This can be seen by considering the equilibria of the limiting system \eqref{eq:B2_SN__equation_proof}. When they exist, these are located at $(X,Y) = (X^\pm, Y^\pm)$ where
    \begin{equation}
    \label{eq:eq_coords}
        X^\pm = 1-\kappa_2+\frac{\kappa_1}{2}\pm \sqrt{(1-\kappa_2+\frac{\kappa_1}{2})^2+\kappa_2 -1}, \qquad
        Y^\pm = \frac{1-\kappa_2}{X^\pm} .
    \end{equation}
    In the following we denote these equilibria by $q_s : (X^+, Y^+)$ and $q_{f/n} : (X^-, Y^-)$. The (somewhat suggestive) notation will be justified below. 

    \begin{rem}
        When $\kappa_1 = 1/2$ the limiting system \eqref{eq:B2_SN__equation_proof} satisfies
        \[
        X'|_{X = 1/2} = 0 , \qquad 
        Y'|_{X = 1/2} = \frac{1}{\theta_2 \Tilde{\eps}_2}Y(1-Y) \left( 1-\frac{1}{2}Y-\kappa_2 \right) ,
        \]
        which shows that the line $\{(1/2,Y) : Y \in (0,1)\}$ which corresponds to $x_2 = 0$ and $y_2 \in \R$ in system \eqref{eq:B2_dynamics_on_sphere} is invariant. There is an equilibrium on this line when $Y = 2 (1 - \kappa_2)$ as long as $\kappa_2 \in (1/2, 1)$ (we require that $Y \in (0,1)$), which is saddle-type as long as $\kappa_2 \in (3/4, 1)$. This justifies the observations in Remark \ref{rem:heteroclinic2}, see also Figure \ref{fig:B2_heteroclinic}.
    \end{rem}

\begin{figure}[]
    \centering
    \begin{subfigure}[t]{0.3\textwidth}
        \centering
            \includegraphics[width=1\linewidth]{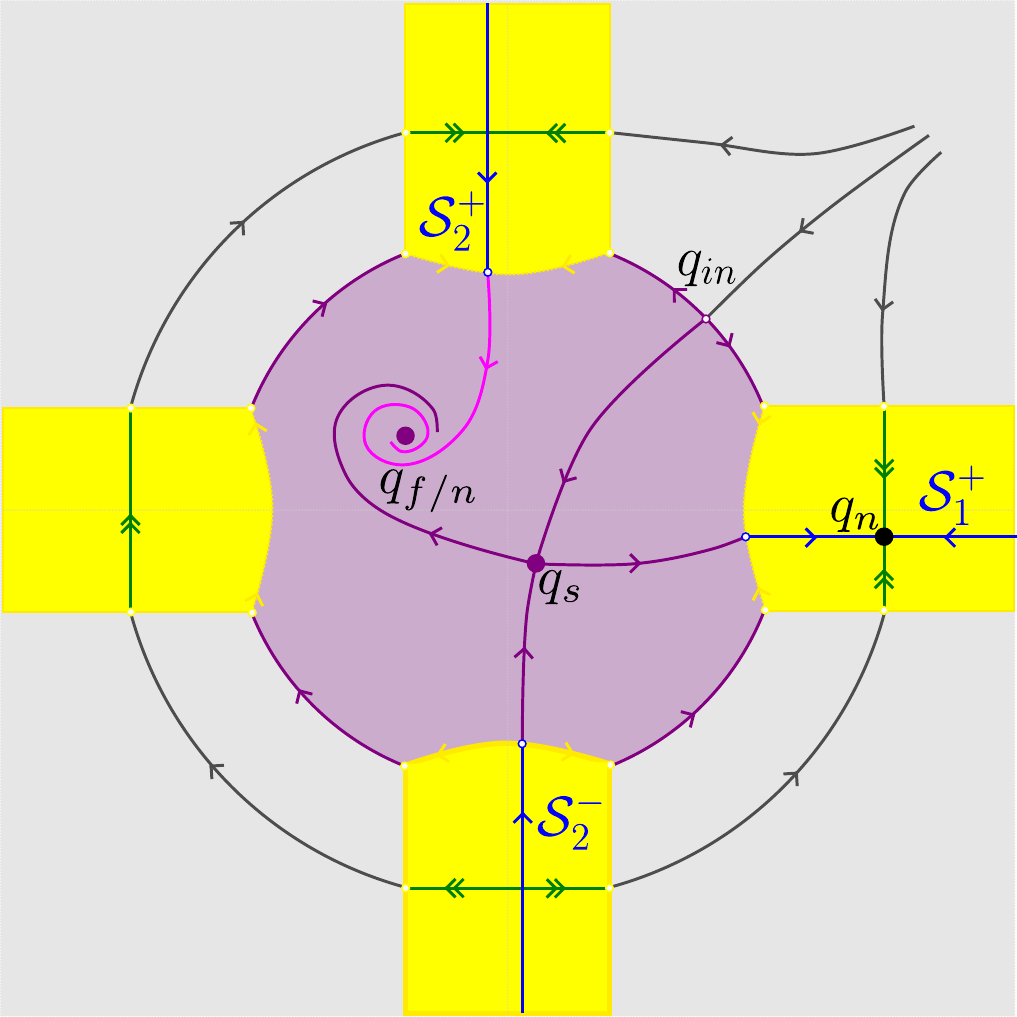}
    \caption{Region II.}
    \label{fig:B2_II}
    \end{subfigure}%
    \hfill
    \begin{subfigure}[t]{0.3\textwidth}
        \centering
            \includegraphics[width=1\linewidth]{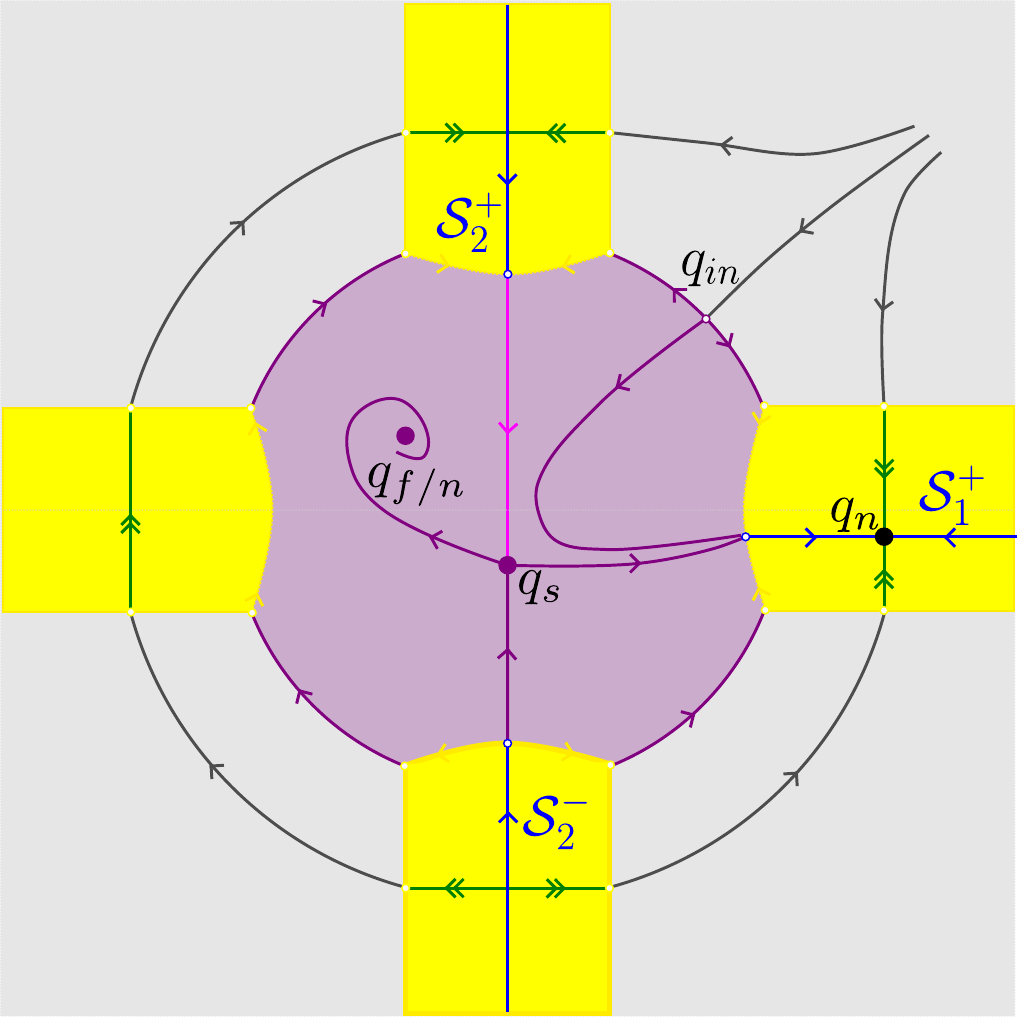}
    \caption{Border between region II and III in Figure \ref{fig:B2_bifurcation_diagram}, i.e., $\kappa_1=1/2$ and $\kappa_2 \in (3/4,1)$.}
    \label{fig:B2_II_III}
    \end{subfigure}
    \hfill
    \begin{subfigure}[t]{0.3\textwidth}
        \centering
        \includegraphics[width=1\linewidth]{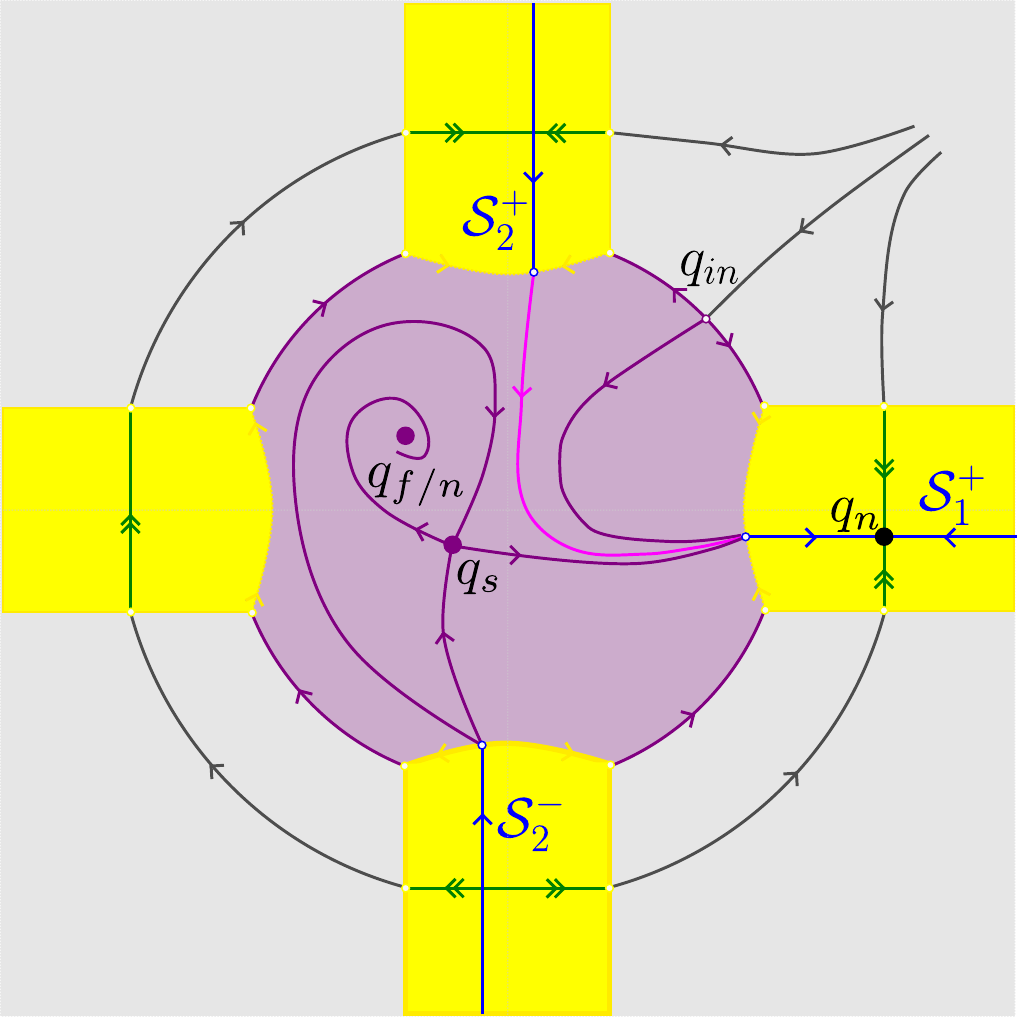}
    \caption{Region III.}
    \label{fig:B2_III}
    \end{subfigure}
    \caption{Selected phase portraits of \eqref{eq:B2_SN__equation_proof}, emphasizing the importance for the global dynamics of the transition between region II and III in the $B_2$ bifurcation diagram (vertical dotted black line  in Figure \ref{fig:B2_bifurcation_diagram}) as indicated in Remark \ref{rem:main_global_dynamics}. 
    In (a): Phase portrait in region II. The stable manifold of $q_s$, denoted $w^s(q_s)$, acts as a separatrix and dissects the state space. Solutions with initial conditions to the right of the separatrix are attracted to $\mathcal{S}_1^+$ and follow the slow flow towards $q_n$ (in the singular limit at least). Solutions with initial conditions to the left of $w^s(q_s)$ are attracted to $\mathcal{S}_2^+$ and follow its (singular) extension on the sphere (magenta) towards $q_{f/n}$. 
    In (b): The critical case on the border line between region II and III where a heteroclinic connection (magenta) between the endpoint of $\mathcal{S}_2^+$ and $q_s$ exists (in fact $\mathcal{S}_2^+$ and its extension is part of the stable manifold $w^s(q_s)$).
    In (c): Phase portrait in region III. The (singular) extension (magenta) of $\mathcal{S}_2^+$ on the sphere (and therefore all solutions attracted by $\mathcal{S}_2^+$) converges to $q_n$. The basin of attraction of $q_{f/n}$ shrinks to the region enclosed by the stable manifold of $q_s$ which emanates from $\mathcal{S}_2^-$.}
    \label{fig:B2_heteroclinic}
\end{figure}

   
    The following result describes a (subcritical) Bogdanov-Takens bifurcation in the limiting system \eqref{eq:B2_SN__equation_proof}.

    \begin{lem}
    \label{lem:B2_BT}
        Fix $\beta_2 > \beta_1 > 0$ and $\tilde \eps_2 \in [\beta_1, \beta_2]$. Then the limiting system \eqref{eq:B2_SN__equation_proof} has a regular Bogdanov-Takens bifurcation at $(X,Y) = (z_{BT}, z_{BT})$ when $(\kappa_1, \kappa_2) = (\kappa_{1,BT}, \kappa_{2,BT})$, where $z_{BT}$ is given by \eqref{eq:x_B} and
        \[
        \kappa_{1,BT} = 2 z_{BT} (1 - z_{BT}) , \qquad 
        \kappa_{2,BT} = 1 + z_{BT}^2.
        \]
        The associated Hopf bifurcation is subcritical.
    \end{lem}

    \begin{proof}
    We start by applying a coordinate and parameter translation of the form
    $$
    \tilde X := X-z_{BT}, \qquad 
    \tilde Y := Y-z_{BT}, \qquad 
    \tilde{\kappa}_1 := \kappa_1 - \kappa_{1,BT}, \qquad 
    \Tilde{\kappa}_2 := \kappa_2 - \kappa_{2,BT},
    $$
    so that the (proposed) Bogdanov-Takens bifurcation takes place at the origin, both in parameter and variable space. This leads to
\begin{equation}\label{eq:B2_BT_system_origin}
\begin{aligned} 
    \tilde X'&=\frac{1}{\theta_1}(\tilde X+z_{BT})(1-\tilde X-z_{BT})\left( \tilde X+ \tilde Y -2 \tilde X \tilde Y-2(\tilde X+\tilde Y)z_{BT} -\Tilde{\kappa}_1 \right) , \\
    \tilde Y'&=\frac{1}{\Tilde{\eps}_2 \theta_2}(\tilde Y+z_{BT})(1-\tilde Y-z_{BT})\left( -  \tilde X \tilde Y - (\tilde X+\tilde Y)z_{BT} -\Tilde{\kappa}_2 \right) .
\end{aligned}
\end{equation}
We now show that system \eqref{eq:B2_BT_system_origin} has a regular Bogdanov-Takens bifurcation at $(\tilde X, \tilde Y, \tilde \kappa_1, \tilde \kappa_2) = (0, 0, 0, 0)$ by verifying the relevant eigenvalue, genericity and non-degeneracy conditions from \cite[Theorem 8.4]{Kuznetsov_1998}. The Jacobian at $(0, 0, 0, 0)$ is given by
\begin{equation} \label{eq:B2_BT_Jacobian}
   J_0 = z_{BT}(1-z_{BT}) \begin{pmatrix}
        \frac{1-2z_{BT}}{\theta_1} & \frac{1-2z_{BT}}{\theta_1}\\
        -\frac{z_{BT}}{\Tilde{\eps}_2 \theta_2} & -\frac{z_{BT}}{\Tilde{\eps}_2 \theta_2}
    \end{pmatrix} .
\end{equation}
Note that $z_{BT}$ is contained within a compact subinterval of $(0,1)$, since $\tilde \eps_2 \in [\beta_1, \beta_2]$. A direct calculation using \eqref{eq:x_B} shows that $\trace J_0 = \det J_0 = 0$, implying a double-zero eigenvalue. It remains to check the following conditions from \cite[Theorem 8.4]{Kuznetsov_1998}, which may be formulated using the notation $\tilde Z := (\tilde X, \tilde Y)$ and $\tilde \kappa := (\tilde \kappa_1, \tilde \kappa_2)$:
\begin{itemize}
    \item[(BT.0)] $J_0 \neq \mathbb O_{2,2}$, where $\mathbb O_{2,2}$ denotes the $2 \times 2$ zero matrix;
    \item[(BT.1)] $a_{20}(0) + b_{11}(0) \neq 0$, where $a_{20}(\tilde \kappa)$ and $b_{11}(\tilde \kappa)$ denote coefficients associated with the linear part of the Bogdanov-Takens normal form (these are defined precisely below);
    \item[(BT.2)] $b_{20}(0) \neq 0$ where, as above, $b_{20}(\tilde \kappa)$ is a coefficient associated with the linear part of the Bogdanov-Takens normal form (this is also defined precisely below);
    \item[(BT.3)] the map $(\tilde Z, \tilde \kappa) \mapsto (F, \trace J, \det J)$ is regular at $(\tilde Z, \tilde \kappa) = (0, 0)$. 
\end{itemize}
(BT.0) follows immediately from \eqref{eq:B2_BT_Jacobian} and the fact that $z_{BT} \in (0,1)$. In order to check (BT.1)-(BT.2), we need to determine the coefficients $a_{20}(0)$, $b_{11}(0)$ and $b_{20}(0)$ using formulae derived as part of the normal form derivation presented in \cite[Sec.~8.4.1]{Kuznetsov_1998}.
%
%
%
These formulae depend on the (generalized) right and left eigenvectors of $J_0$, which we denote by $v_0$, $v_1$ and $w_1$, $w_0$ respectively. A direct calculation shows that
$$
v_0 = \frac{z_B}{\Tilde{\eps}_2\theta_2}
\begin{pmatrix}
    1 \\
    -1
\end{pmatrix}, \qquad 
v_1 = \frac{\Tilde{\eps}_2\theta_2}{z_B}
\begin{pmatrix}
    1 \\
    0
\end{pmatrix}, \qquad
w_1 = 
\frac{z_B}{\Tilde{\eps}_2\theta_2}
\begin{pmatrix}
    1 \\
    1
\end{pmatrix}, \qquad 
w_0 = \frac{\Tilde{\eps}_2\theta_2}{z_B}
\begin{pmatrix}
    0\\
    -1
\end{pmatrix} , $$
where the pre-factors have been chosen to ensure a normalisation which satisfies $\langle v_0,w_0 \rangle=\langle v_1,w_1\rangle=1$ and $\langle v_1,w_0\rangle=\langle v_0,w_1\rangle=0$. Under the linear change of coordinates defined by
\[
z_1 = \langle \tilde Z,w_0 \rangle=-\frac{z_{BT}}{\Tilde{\eps}_2\theta_2}\tilde Y , \qquad
z_2 = \langle \tilde Z,w_1 \rangle=\frac{z_{BT}}{\Tilde{\eps}_2\theta_2}(\tilde X+ \tilde Y) ,
\]
system \eqref{eq:B2_BT_system_origin} is transformed to
\begin{equation}\label{eq:B2_BT_system_origin_transformed}
    \begin{aligned} 
    z_1' &= G_1(z_1, z_2, \tilde \kappa) , \\
    z_2' &= G_2(z_1, z_2, \tilde \kappa) ,
    \end{aligned}
\end{equation}
where the functions $G_i$ are smooth; here we have opted to omit the explicit expressions, which are lengthy and algebraically messy.
Using these expressions and a little computer algebra (SymPy package in Python), we obtain the following expressions for the coefficient functions $a_{20}$, $b_{20}$ and $b_{11}$:
\[
\begin{split}
    a_{20}(\Tilde{\kappa}) &= \frac{\partial^2}{\partial z_1^2}G_1 \vert_{(z_1,z_2)=(0,0)}=\frac{2 z_{BT}^{2} \left(z_{BT} - 1\right)-2\Tilde{\kappa}_2 z_{BT}}{\Tilde{\eps}_2^{2} \theta_{2}^{2}}, \\
    b_{20}(\Tilde{\kappa}) &= \frac{\partial^2}{\partial z_1^2}G_2 \vert_{(z_1,z_2)=(0,0)}= \frac{2 z_{BT}^{3} \left(\Tilde{\eps}_2 \theta_2(\Tilde{\kappa}_1+2z_{BT}(1-z_{BT}))+\theta_1(\Tilde{\kappa}_2+z_{BT}(1-z_{BT}))\right)}{\Tilde{\eps}_2^{4} \theta_{2}^{4} \theta_1} , \\
    b_{11}(\Tilde{\kappa}) &= \frac{\partial^2}{\partial z_1 \partial z_2}G_2 \vert_{(z_1,z_2)=(0,0)}=\frac{z_{BT} \left(2 \Tilde{\eps}_2 \theta_{2} (\Tilde{\kappa}_1 -z_{BT}(z_{BT}-1))+\Tilde{\eps}_2\theta_2(2z_{BT}-1)^2-\theta_1z_{BT}(3z_{BT}-2) \right)}{\Tilde{\eps}_2^{2} \theta_{2}^{2}\theta_1}.
\end{split}
\]
After inserting \eqref{eq:x_B} for $z_{BT}$, a direct computation shows that 
(BT.1) and (BT.2) are satisfied with
$$a_{20}(0)+b_{11}(0)=\frac{2 \Tilde{\eps}_2^{2} \theta_{2}^{2} + \Tilde{\eps}_2 \theta_{1} \theta_{2}  + \theta_{1}^{2}}{\theta_1(2\Tilde{\eps}_2 \theta_2 +\theta_1)^3}>0 , \qquad
b_{20}(0)=\frac{2 \left(\Tilde{\eps}_2 \theta_{2} + \theta_{1}\right) }{\theta_1 \left(2 \Tilde{\eps}_2 \theta_{2} + \theta_{1}\right)^{4}} >0 ,$$
respectively.

It remains to check (BT.3), i.e., the regularity of the map $\Theta : (\tilde Z,\Tilde{\kappa}) \mapsto (F,\trace J, \det J)$ at $(\tilde Z,\Tilde{\kappa})=(0,0)$. A direct computation using computer alegbra (SymPy in Python again) shows that the determinant of the Jacobian matrix associated with this map at $(\tilde Z,\Tilde{\kappa})=(0,0)$ has the form $\det \Theta (0,0) = N / D$, where
\[
\begin{split}
    N :=& \Tilde{\eps}_2^{2} \theta_{2}^{2} \left(\Tilde{\eps}_2^{5} \theta_{2}^{5} + 5 \Tilde{\eps}_2^{4} \theta_{1} \theta_{2}^{4} + 10 \Tilde{\eps}_2^{3} \theta_{1}^{2} \theta_{2}^{3} + 10 \Tilde{\eps}_2^{2} \theta_{1}^{3} \theta_{2}^{2} + 5 \Tilde{\eps}_2 \theta_{1}^{4} \theta_{2} + \theta_{1}^{5}\right) > 0 , \\
    D :=& 512 \Tilde{\eps}_2^{9} \theta_{2}^{9} + 2304 \Tilde{\eps}_2^{8} \theta_{1} \theta_{2}^{8} + 4608 \Tilde{\eps}_2^{7} \theta_{1}^{2} \theta_{2}^{7} + 5376 \Tilde{\eps}_2^{6} \theta_{1}^{3} \theta_{2}^{6} + 4032 \Tilde{\eps}_2^{5} \theta_{1}^{4} \theta_{2}^{5} + 2016 \Tilde{\eps}_2^{4} \theta_{1}^{5} \theta_{2}^{4} + 672 \Tilde{\eps}_2^{3} \theta_{1}^{6} \theta_{2}^{3} \\
    &+ 144 \Tilde{\eps}_2^{2} \theta_{1}^{7} \theta_{2}^{2} + 18 \Tilde{\eps}_2 \theta_{1}^{8} \theta_{2} + \theta_{1}^{9} > 0 .
\end{split}
\]
Thus $\det \Theta (0,0) > 0$, which shows that (BT.3) is satisfied.
%

Since (BT.0)-(BT.3) are satisfied with $a_{20}(0)+b_{11}(0)>0$, \cite[Thm.~8.4]{Kuznetsov_1998} implies that system \eqref{eq:B2_SN__equation_proof} has a regular Bogdanov-Takens bifurcation at $(X,Y) = (X_{BT}, Y_{BT})$ when $(\kappa_1, \kappa_2) = (\kappa_{1,BT}, \kappa_{2,BT})$, and that the associated Hopf bifurcation is subcritical.
\end{proof}

Because the Bogdanov-Takens bifurcation identified in the limiting system \eqref{eq:B2_SN__equation_proof} is \textit{regular}, it persists as a regular Bogdanov-Takens bifurcation in system \eqref{eq:B2_SN_equation_proof_perturbed} when $0 < \eps_1 \ll 1$. This is formalised in the following result.

\begin{prop}
\label{prop:B2}
\textup{(Regular Bogdanov-Takens bifurcation in $B_2$)}
Fix $\beta_2 > \beta_1 > 0$ and $\tilde \eps_2 \in [\beta_1, \beta_2]$. 
There exists an $\eps_{1,0} > 0$ and $C^1$-function 
$(\kappa_{1,BT}, \kappa_{2,BT}) : [0, \eps_{1,0}) \to \R^2$ such that system \eqref{eq:B2_SN_equation_proof_perturbed} undergoes a Bogdanov-Takens bifurcation 
when $(\kappa_1, \kappa_2) = (\kappa_{1,BT}, \kappa_{2,BT})(\eps_1)$. The functions 
$\kappa_{1,BT}$ and $\kappa_{2,BT}$ satisfy
\[
    \kappa_{1,BT}(\eps_1) = 2 z_{BT} (1 - z_{BT}) + \mathcal O(\eps_1) , \qquad
    \kappa_{2,BT}(\eps_1) = 1 + z_{BT}^2 + \mathcal O(\eps_1) ,
\]
as $\eps_1 \to 0$. The associated Hopf bifurcation is subcritical.
  %
\end{prop}

\begin{proof}
    The local persistence of an equilibrium with an associated Jacobian matrix $J$ which has a double-zero eigenvalue follows from the implicit function theorem and the fact that the differential associated with the map $(\tilde Z, \tilde \kappa, s_2) \mapsto (F, \trace J, \det J)$ at $(\tilde Z, \tilde \kappa, \eps_1) = (0, 0, 0)$ has maximal rank, which is a consequence of the fact that the genericity condition (BT.3) is satisfied when $\eps_1 = 0$ (this was shown in the proof of Lemma \ref{lem:B2_BT} above).

    In order to show that a Bogdanov-Takens bifurcation occurs, it remains to check the genericity and non-degeneracy conditions (BT.0)-(BT.3). Since these are `non-zero conditions', they can be verified for all $0 \leq \eps_1 \ll 1$ sufficiently small by appealing to the fact that they are true when $\eps_1 = 0$. For example: because the Jacobian matrix satisfies $J = J_0 + \mathcal O(\eps_1)$ as $\eps_1 \to 0$ and $J_0 \neq \mathbb O_{2,2}$ is smooth, it follows that $J \neq \mathbb O_{2,2}$ for all $0 \leq \eps_1 \ll 1$ sufficiently small. Thus (BT.0) is satisfied. Similar arguments show that (BT.1)-(BT.2) are satisfied for all $0 \leq \eps_1 \ll 1$ sufficiently small, with the same signs. This means that the associated Hopf bifurcation is also subcritical. Finally, (BT.3) requires that the map $(\tilde Z, \tilde \kappa, \eps_1) \mapsto (\hat F, \trace J, \det J)$ is regular. Since regularity of the map $(\tilde Z, \tilde \kappa) \mapsto (\tilde Z, \trace J_0, \det J_0)$ is enough to guarantee maximal rank for all $0 \leq \eps_1 \ll 1$, the proof is complete.
\end{proof}

Proposition \ref{prop:B2} confirms that a regular Bogdanov-Takens bifurcation occurs in system \eqref{eq:B2_SN_equation_proof_perturbed}, for all $0 \leq \eps_1 \ll 1$. This implies the existence of regular saddle-node, subcritical Hopf and homoclinic bifurcations along codimension-1 branches in the $(\kappa_1, \kappa_2)$-bifurcation set, \textit{locally near} $(\kappa_{1,BT}(\eps_1), \kappa_{2,BT}(\eps_1))$. In order to complete the proof of Theorem \ref{thm:B2}, however, we still need to (i) characterise the saddle-node curve for system \eqref{eq:B2_SN_equation_proof_perturbed} outside of this neighbourhood, and (ii) show that this blows down to the desired result for system \eqref{eq:fund_prob_P1} in the original $(x,y)$-coordinates. We start with (i).

\begin{prop}
    \label{prop:B2_SN_Hopf}
\textup{(Regular saddle-node bifurcation in $B_2$)}
Fix $\beta_2 > \beta_1 > 0$, $\tilde \eps_2 \in [\beta_1, \beta_2]$, a compact subinterval $\mathcal I \subset (0,1)$ and $\mathcal J(\eps_1) = \{ \kappa_2 \in \mathcal I : \kappa_2 > \kappa_{2,BT}(\eps_1) \}$. 
There exists an $\eps_{1,0} > 0$ and a $C^1$-function $\kappa_{1,sn} : \mathcal I \times [0, \eps_{1,0}) \to \R^2$ 
such that system \eqref{eq:B2_SN_equation_proof_perturbed} undergoes a regular saddle-node bifurcation when $\kappa_1 = \kappa_{1,sn}(\kappa_2, \eps_1)$, 
unless $(\kappa_1, \kappa_2) = (\kappa_{1,BT}(\eps_1), \kappa_{2,BT}(\eps_1))$. 
We have that
\[
    \kappa_{1,sn}(\kappa_2, \eps_1) = 2 \kappa_2 - 2 + 2 \sqrt{1 - \kappa_2} + \mathcal O(\eps_1) 
\]
as $\eps_1 \to 0$. 
\end{prop}

\begin{proof}
    We start by showing that the limiting problem \eqref{eq:B2_SN__equation_proof} has a regular saddle-node bifurcation. Direct calculations show that $F(Z_{sn},Z_{sn},\kappa_{1,sn}(\kappa_2,0),\kappa_2)=0$, where $Z_{sn} = \sqrt{1 - \kappa_2} \in (0,1)$, $\kappa_{1,sn}(\kappa_2,0)$ is as in the statement of the proposition, and that the Jacobian matrix 
    \begin{equation}
        \label{eq:B2_Jacobian}
        J(Z_{sn},Z_{sn},\kappa_{1,sn}(\kappa_2,0),\kappa_2,0)=Z_{sn}(1-Z_{sn})\begin{pmatrix}
    \frac{1}{\theta_1}(1-2Z_{sn}) & \frac{1}{\theta_1}(1-2Z_{sn})\\
    \frac{-Z_{sn}}{\theta_2 \Tilde{\eps}_2} & \frac{-Z_{sn}}{\theta_2 \Tilde{\eps}_2}
    \end{pmatrix} ,
    \end{equation}
    has exactly one zero eigenvalue as long as $\kappa_2 \in \mathcal I \setminus \{ \kappa_{2,BT} \}$, i.e.~as long as one stays away from the Bogdanov-Takens point identified in Lemma \ref{lem:B2_BT}. In order to confirm that the loss of hyperbolicity along $\kappa_1 = \kappa_{1,sn}(\kappa_2,0)$ is due to a saddle-node bifurcation, it remains to check the following transversality and non-degeneracy conditions, which are obtained from \cite[formula (5.50)]{Kuznetsov_1998} after projection onto the relevant center manifold:
    \begin{enumerate}
        \item[(SN.1)] $w^T F_{\kappa_1}(Z_{sn},Z_{sn},\kappa_{1,sn}(\kappa_2,0),\kappa_2)\neq 0$, where $F_{\kappa_1}$ denotes the $2 \times 1$ vector of partial derivatives of components of $F$ with respect to $\kappa_1$;
        \item[(SN.2)] $w^T D^2_{(X,Y)} F(Z_{sn},Z_{sn},\kappa_{1,sn}(\kappa_2,0),\kappa_2)[v,v] \neq 0$, where $D^2_{(X,Y)} F[v,w]$ denotes the bilinear form which is defined entry-wise via $v^TD^2_{(X,Y)} F_iw$ and $D^2_{(X,Y)} F_i$, where the latter is the $2\times2$ Hessian matrix that contains all the second order derivatives of $F=(F_1,F_2)$ ($i \in \{1,2\}$).
    \end{enumerate}
for all $\kappa_2 \in \mathcal I \setminus \{ \kappa_{2,BT} \}$, where 
$$v=\begin{pmatrix}
    1\\
    -1
\end{pmatrix}, \qquad 
w=\begin{pmatrix}
    \frac{Z_{sn}}{\theta_2 \Tilde{\eps}_2}\\ 
    \frac{1}{\theta_1}(1-2Z_{sn})
\end{pmatrix},$$
are the right and left eigenvectors associated with the zero eigenvalue of the Jacobian respectively. Substituting all of the relevant formulae into the transversality and nondegeneracy conditions above, we find that
$$w^T  F_{\kappa_1}(Z_{sn},Z_{sn},\kappa_{1,sn}(\kappa_2,0),\kappa_2)=-\frac{1}{\theta_1\theta_2 \Tilde{\eps}_2}{Z_{sn}}^2(1-Z_{sn})\neq 0,$$
and
$$w^T D^2_{(X,Y)}  F(Z_{sn},Z_{sn},\kappa_{1,sn}(\kappa_2,0),\kappa_2)[v,v]= \frac{1}{\theta_1\theta_2 \Tilde{\eps}_2} 2Z_{sn}(1-Z_{sn}) \neq 0$$
for all $\kappa_2 \in (0,1)$. Thus (SN.1)-(SN.2) are satisfied, implying the existence of a saddle-node bifurcation when $\kappa_1 = \kappa_{1,sn}(\kappa_2,0)$ for all $\kappa_2 \in \mathcal (0,1) \setminus \{ \kappa_{2,BT} \}$. The existence of the $C^1$-function $\kappa_{1,sn}$ described in the proposition follows from the implicit function theorem, since the differential associated with the map $(Z, \kappa, \eps_1) \mapsto (\hat F, \det D_{(X,Y)} \hat F)$ at $(Z, \kappa, \eps_1) = (Z_{sn}, Z_{sn}, \kappa_{1,sn}(\kappa_2,0), \kappa_2, 0)$ has maximal rank, which implies that the equilibrium $(Z_{sn}, Z_{sn})$ with a single zero eigenvalue persists for $0<\eps_1 \ll 1$. Together with regular perturbation arguments which ensure that the non-zero conditions (SN.1)-(SN.2) remain true after sufficiently small perturbations, this completes the proof.
\end{proof}

 %

We are now in a position to complete the proof of Theorem \ref{thm:B2}.


\begin{proof}[Proof of Theorem \ref{thm:B2}]
Together, Propositions \ref{prop:B2} and \ref{prop:B2_SN_Hopf} show that for all $\eps_1 \in [0, \eps_{1,0})$ with $\eps_{1,0} > 0$ sufficiently small, the transformed system \eqref{eq:B2_SN_equation_proof_perturbed} has (i) a regular Bogdanov-Takens bifurcation, and (ii) a saddle-node bifurcation which admits of a parameterisation on larger regions in $(\kappa_1, \kappa_2)$-space. Since the coordinate change \eqref{B2_variable_change_BT} is smooth and smoothly invertible, corresponding results hold for system \eqref{eq:B2_spherical_scaling_chart}. In particular, since the coordinate transformation leaves $\kappa_1$ and $\kappa_2$ fixed, the functions $\kappa_{1,BT}$, $\kappa_{2,BT}$ and $\kappa_{1,sn}$ have the same asymptotic formulae as $\eps_1 \to 0$ in system \eqref{eq:B2_spherical_scaling_chart}. The results obtained in chart $\mathcal K_2$ (Propositions \ref{prop:B2} and \ref{prop:B2_SN_Hopf} in particular) apply almost verbatim after blowing down to the original system \eqref{eq:fund_prob_P1}, except that the blow-down transformation leads to statements for $\eps_1 \in (0, \eps_{1,0})$ (i.e.~without $\eps_1 = 0$).
\end{proof}

\begin{rem} \label{rem:implicit_hopf}
    Proposition \ref{prop:B2} shows that the Hopf bifurcation is subcritical for $\kappa_2$-values sufficiently close to $\kappa_{2,BT}$, however, we have not included a result analogous to Proposition \ref{prop:B2_SN_Hopf} which extends the parametrization of the Hopf curve. Direct calculations show that the Jacobian matrix $J$ associated with the limiting system \eqref{eq:B2_SN__equation_proof} satisfies the zero trace condition along a parameter space curve defined implicitly by
    \[
    \theta_1^{-1} (1 - X^-) (1 - 2Y^-) - \theta_2^{-1} \tilde \eps_2^{-1} Y^- (1 - Y^-) = 0 , 
    \]
    where $X^-$ and $Y^-$ denote the $(\kappa_1,\kappa_2)$-dependent coordinates of the equilibrium $q_{f/n}$ defined in \eqref{eq:eq_coords}. The determinant associated with $J$ can be shown to be positive along this curve as long as $\kappa_2 \in (3/4,1)$, however, we were unable to verify the usual transversality and non-degeneracy conditions that are required to prove the existence of a nondegenerate Hopf bifurcation. Numerical continuation in MatCont suggests that the branch of subcritical Hopf bifurcations which emerges from the Bogdanov-Takens point continues in a regular fashion as shown in Figure \ref{fig:B2_bifurcation_diagram}.
\end{rem}

\section{\texorpdfstring{Region $B_1$}{Region B1}} \label{sec:Region_B1}
We turn now to the geometric blow-up construction in $B_1$. We continue to work with system \eqref{eq:fund_prob_P1}, i.e.~in parameter chart $\mathcal P_1$, except that our primary focus now shifts to the singular geometry and dynamics as both $\eps_1 \to 0$ and $\tilde \eps_2 \to 0$ (as opposed to the $B_2$ analysis in the preceding section, where only $\eps_1 \to 0$). 

\subsection{Blow-up, singular geometry and dynamics}
\label{sec:B1_blow-up_geometry}


We can build upon the geometry established in Section \ref{sec:Region_B2}. We saw in Section \ref{sub:B2_geometry} that by starting with the extended system \eqref{eq:P1_extended} and applying a sequence of blow-up transformations which consisted of (i) a spherical blow-up of the intersection point $\Sigma_1 \cap \Sigma_2 \cap \{0\}$, followed by (ii) cylindrical blow-ups along $\Sigma_1^\pm \times \{0\}$ and $\Sigma_2^\pm \times \{0\}$, the loss of smoothness could be resolved for any fixed $\tilde \eps_2 > 0$. Letting $\tilde \eps_2 \to 0$, however, leads to additional loss of smoothness along the horizontal axes in the local coordinate charts $\mathcal K_2$, $\mathfrak K_2$ (defined in \eqref{eq:mathfrak_K2_coordinates}), and the counterpart to $\mathfrak K_2$ which would be associated with the cylindrical blow-up of $\Sigma_1^- \times \{0\}$ to the left of the blow-up sphere. Three more cylindrical blow-ups will be necessary to resolve this.

We start with the degeneracy arising on the blow-up sphere, which is visible in chart $\mathcal K_2$. Recall that the dynamics in chart $\mathcal K_2$ are governed by
\begin{equation}
\label{eq:B2_eqns_again}
\begin{aligned} 
    x_2'&=\frac{1}{\theta_1e^{\eps_1 x_2}}\left( \Hat{H}(x_2)+ \Hat{H}\left(\frac{y_2}{\Tilde{\eps}_2}\right) -2 \Hat{H}(x_2) \Hat{H}\left(\frac{y_2}{\Tilde{\eps}_2}\right) -\kappa_1 e^{\eps_1 x_2}\right) , \\
    y_2'&=\frac{1}{\theta_2e^{\eps_1 y_2}}\left( 1-  \Hat{H}(x_2) \Hat{H}\left(\frac{y_2}{\Tilde{\eps}_2}\right) -\kappa_2 e^{\eps_1 y_2}\right) ,
\end{aligned}
\end{equation}
where $0 \leq \eps_1 \ll 1$. This system has a PWS singular limit when $\tilde \eps_2 \to 0$, due to \eqref{eq:Hill_limit}, which can be written as
\begin{equation} \label{eq:B1_PWS_limit}
    \begin{aligned} 
    \begin{pmatrix}
        x_2' \\
        y_2'
    \end{pmatrix}=
    \begin{cases}
        Z^+(x_2,y_2,\eps_1), & y_2>0, \\ 
        Z^-(x_2,y_2,\eps_1), & y_2<0, \\ 
    \end{cases}
\end{aligned}
\end{equation}
where
\begin{equation} \label{eq:B1_Z+_Z-}
    \begin{aligned} 
        Z^+(x_2,y_2,\eps_1) :=
        \begin{pmatrix}
    \frac{1}{\theta_1e^{\eps_1 x_2}} (1 - \Hat{H}(x_2)  -\kappa_1 e^{\eps_1 x_2}) \\
    \frac{1}{\theta_2e^{\eps_1 y_2}}( 1-  \Hat{H}(x_2)  -\kappa_2 e^{\eps_1 y_2})
    \end{pmatrix}, \qquad 
        Z^-(x_2,y_2,\eps_1) :=
        \begin{pmatrix}
    \frac{1}{\theta_1e^{\eps_1 x_2}}(  \Hat{H}(x_2) - \kappa_1 e^{\eps_1 x_2})\\
    \frac{1}{\theta_2e^{\eps_1 y_2}}( 1  -\kappa_2 e^{\eps_1 y_2}) 
    \end{pmatrix},
\end{aligned}
\end{equation}
and the switching manifold is
\[
\Sigma_1 := \{(x_2,0) : x_2 \in \R \}.
\]
The main features of the limiting PWS system \eqref{eq:B1_PWS_limit} are summarised in the following result.

\begin{lem}
\label{lem:B1_pws}
    Consider the limiting PWS system \eqref{eq:B1_PWS_limit}, with $(\kappa_1, \kappa_2) \in \Lambda$. The following assertions are true:
    \begin{enumerate}
        \item[(i)] The half-line $\{ ( \ln ( (1 - \kappa_1) / \kappa_1), y_2) : y_2 > 0 \}$ is invariant under the flow induced by $Z^+|_{\eps_1 = 0}$, and $y_2' < 0$ along it;
        \item[(ii)] The half-line $\{ ( \ln ( \kappa_1 / (1 - \kappa_1)), y_2) : y_2 < 0 \}$ is invariant under the flow induced by $Z^-|_{\eps_1 = 0}$, and $y_2' > 0$ along it;
        \item[(iii)] There exists an $\eps_{1,0} > 0$ such that there is a regular invisible fold point $F : (\ln(1 - \kappa_2) / \kappa_2, 0) \in \Sigma_1$ for all $\eps_1 \in [0,\eps_{1,0})$.
    \end{enumerate}
\end{lem}

\begin{proof}
    Assertions (i)-(ii) are consequences of the fact that the limiting vector fields $Z^\pm(x_2,y_2,0)$ do not depend upon $y_2$, and can be checked directly. One can directly verify that $F$ is a regular invisible fold point by writing $z = (x_2, y_2)$, $f(z) = y_2$ and showing that the defining conditions (see e.g.~\cite{Kristiansen_2015}), namely
    \[
    Z^+f(F,\eps_1) = 0, \qquad
    Z^+(Z^+f)(F,\eps_1) < 0, \qquad
    Z^-f(F,\eps_1) > 0,
    \]
    where $\mathcal Z f(\cdot) := \langle \nabla f(\cdot),\mathcal Z(\cdot) \rangle$ denotes the Lie derivative of $f$ along the vector field $\mathcal Z$, are satisfied for all $\eps_1 \in [0,\eps_{1,0})$. In particular: a direct calculation shows that $Z^+f(F,\eps_1) = 0$ for all $\eps_1 \geq 0$, and that
    \[
    Z^+(Z^+f)(F,\eps_1) = -\frac{1}{\theta_1\theta_2}(1-\kappa_2)\kappa_2(\kappa_2-\kappa_1) + \mathcal O(\eps_1) , \qquad 
    Z^-f(F,\eps_1) = \frac{1}{\theta_2} (1 - \kappa_2) + \mathcal O(\eps_1),
    \]
    as $\eps_1 \to 0$. Fixing $(\kappa_1, \kappa_2) \in \Lambda$, it follows that there exists an $\eps_{1,0} > 0$ sufficiently small to ensure that these quantities are strictly negative and positive respectfully for all $\eps_1 \in [0, \eps_{1,0})$.
\end{proof}

Lemma \ref{lem:B1_pws} justifies the sketch of the PWS geometry and dynamics on the blow-up sphere shown in Figure \ref{fig:B1_piecewise_smooth}. Away from $\Sigma_1$, system \eqref{eq:B2_eqns_again} is regularly perturbed in both $\tilde \eps_2$ and $\eps_1$, so the PWS description is expected to provide a good description of the leading order dynamics (on compact subsets of the $(x_2, y_2)$-plane).

In order to determine the behaviour close to $\Sigma_1$, we permit an abuse of notation by dropping the subscripts on the variables and parameters, consider the extended system obtained by appending $\tilde \eps'=0$ to system \eqref{eq:B2_eqns_again}, which is degenerate along $\Sigma_1 \times \{0\}$, and apply a cylindrical blow-up of the form
\begin{equation}
\label{eq:cylindrical_blow-up_B1}
\eta \geq 0, \ (\bar y, \bar{\tilde \eps} ) \in \mathbb S^1 \mapsto 
\begin{cases}
    y = \eta \bar y, \\
    \tilde \eps = \eta \bar{\tilde \eps} .
\end{cases}
\end{equation}
A complete treatment would involve an analysis in the three coordinate charts defined via $\bar y = \pm 1$ and $\bar{\tilde \eps} = 1$. We shall present a detailed analysis in the scaling chart $\bar{\tilde \eps} = 1$ at the end of Section \ref{sec:B1_bifurcation_analysis} below. Except for a couple of minor things in $\bar y = 1$, the details of the dynamics have been omitted because (i) they are not necessary for proving Theorem \ref{thm:B1}, and (ii) the complete treatment would lengthen the manuscript significantly.


The remaining two cylindrical blow-ups, which should be applied to the degenerate lines which lie along the horizontal axes in chart $\mathfrak K_2$ and its counterpart on the left-hand side of the blow-up sphere, are homogeneous like \eqref{eq:cylindrical_blow-up_B1}. For brevity, we shall omit the details on the left blow-up, and focus on the blow-up on the right-hand side with an aim towards the identification of the stable node $q_n$. We consider the extended system obtained by appending $\tilde \eps_2' = 0$ to system \eqref{eq:B2_spherical_sigma2_+_cylinder_scaling}, drop the subscripts, apply a cylindrical blow-up along $\Sigma_1^+ := \{ (r, 0, \rho, 0) : r, \rho \geq 0 \}$ defined by
    \[
    \eta \geq 0, \ (\bar y, \bar{\tilde \eps}) \in \mathbb S^1 \mapsto
    \begin{cases}
        y = \eta \bar y , \\
        \tilde \eps = \eta \bar{\tilde \eps} ,
    \end{cases}
    \]
    and consider the following (desingularised) system in the rescaling chart with $y = \tilde \eps y_2$:
    \begin{equation}
        \label{eq:mathfrak_K2_right_eqns}
        \begin{split}
            r' &= \tilde \eps r \rho G\left(\frac{1}{\rho}, y_2, r \right) , \\
            y_2' &= \theta_2^{-1} e^{- r \rho \tilde \eps y_2} \left( 1 - \hat H\left( \frac{1}{\rho} \right) \hat H(y_2) - \kappa_2 e^{r \rho \tilde \eps y_2} \right) , \\
            \rho' &= - \tilde \eps \rho^2 G\left(\frac{1}{\rho}, y_2, r \right) .
        \end{split}
    \end{equation}
    System \eqref{eq:mathfrak_K2_right_eqns} is slow-fast with $0 < \tilde \eps \ll 1$. The layer problem \eqref{eq:mathfrak_K2_right_eqns}$|_{\tilde \eps = 0}$ has a 2-dimensional critical manifold
    \[
    \mathcal S^+_{1} := \left\{ (r, y_2, \rho) \in \R^+ \times \R \times \R^+ : 1 - \hat H \left( \frac{1}{\rho} \right) \hat H(y_2) - \kappa_2 = 0 \right\} ,
    \]
    where we permit a slight abuse of notation by reusing the label $\mathcal S_1^+$. 
    The following result will be used in the proof of Proposition \ref{prop:node}.

\begin{lem}
    \label{lem:node_B1}
    Consider system \eqref{eq:mathfrak_K2_right_eqns} with $(\kappa_1, \kappa_2) \in \Lambda$. The 1-dimensional curve $\mathcal S_{1}^+|_{\rho = 0}$ is an attracting and normally hyperbolic critical manifold for the planar system which governs the flow in $\{\rho = 0\}$. The reduced flow on $\mathcal S_{1}^+|_{\rho = 0}$ has a unique attracting equilibrium $q_{n}$ which is $\mathcal O(\eps_1)$-close to the point $(\ln(\kappa_2 / \kappa_1), \ln((1 - \kappa_2)/\kappa_2), 0)$.
\end{lem}

\begin{proof}
    For clarity in this proof, we reinstate the subscripts on the small parameters by writing $\eps = \eps_1$ and $\tilde \eps = \tilde \eps_2$ as before. The fact that $\mathcal S_{1}^+|_{\rho = 0}$ is an attracting and normally hyperbolic critical manifold within $\{\rho = 0\}$ can be verified directly. In order to determine the reduced flow on $\mathcal S_{1}^+|_{\rho = 0}$, we use $\eps_1 = \rho r$ to eliminate $\rho$ (this is valid when bounded away from the blow-up sphere, i.e.~when $r > 0$), and combine this with the fact that $\tilde \eps_2 = \eps_2 / \eps_1$ in order to obtain a planar slow-fast system of the form
    \[
    \begin{split}
        r' &= \tilde \eps_2 G\left( \frac{r}{\eps_1}, y_2, r \right), \\
        y_2' &= \theta_2^{-1} e^{- \tilde \eps_2 y_2} \left( 1 - \hat H\left(\frac{r}{\eps_1}\right) \hat H(y_2) - \kappa_2 e^{\tilde \eps_2 y_2} \right) ,
    \end{split}
    \]
    where $0 < \eps_1, \tilde \eps_2 \ll 1$ and we have restricted to $r > 0$. For each fixed $(\kappa_1, \kappa_2) \in \Lambda$, the layer problem ($\tilde \eps_2=0$) associated with this system has a normally hyperbolic and attracting critical manifold for all $0 < \eps_1 \ll 1$ sufficiently small. The reduced problem satisfies
    \[
    \dot r = \theta_1^{-1} e^{-r} ( \kappa_2 - \kappa_1 e^r + \mathcal O(\eps_1) )
    \]
    as $\eps_1 \to 0$. A straightforward application of the implicit function theorem shows that there is a unique equilibrium $\mathcal O(\eps_1)$-close to the limiting value $(\ln(\kappa_2 / \kappa_1), \ln((1 - \kappa_2)/\kappa_2), 0)$.
\end{proof}

Lemmas \ref{lem:B1_pws} and \ref{lem:node_B1}, together with the preceding discussion, justify the part of the geometric blow-up construction, singular geometry and dynamics to the right of the blow-up sphere shown in Figure \ref{fig:B1_blow-up}.

\begin{rem}
    There are, of course, gaps in our description. In addition to the geometric and dynamic features identified above, one can identify a local center-stable manifold which has its base along the attracting critical manifolds at the intersection of the blow-up cylinders on top of the blow-up sphere and right-hand blow-up cylinder. There are also lines of resonant saddle equilibria along the intersections between the blow-up cylinders and the underlying blow-up sphere or cylinders (except at the point $\mathcal Q$, which is degenerate). These are (by now) somewhat standard features in blow-up analyses of this kind, and can be analysed using a suite of techniques that have been developed in e.g.~\cite{Krupa_2001_Extend,Kosiuk2009}; we refer also to the recent book \cite{Maesschalck_2021} and the references therein.
\end{rem}

\subsection{Dynamics in the scaling chart and the proof of Theorem \ref{thm:B1}} \label{sec:B1_bifurcation_analysis}

After applying the cylindrical blow-up in \eqref{eq:cylindrical_blow-up_B1}, we obtain the following system in the scaling chart $\bar{\tilde\eps} = 1$ after an appropriate desingularisation, expressed in local coordinates $(x, y) = (x, \tilde \eps y_2)$:
\begin{equation} \label{eq:B1_sphere_cylinder_scaling_chart}
    \begin{aligned} 
    x'&=\frac{\tilde \eps}{\theta_1 e^{\eps x}}\left( \Hat{H}(x)+ \Hat{H}(y_2) -2 \Hat{H}(x) \Hat{H}(y_2) -\kappa_1 e^{\eps x}\right), \\
    y_2'&=\frac{1}{\theta_2e^{\eps \tilde \eps y_2}}\left( 1-  \Hat{H}(x) \Hat{H}(y_2) -\kappa_2 e^{\eps \tilde \eps y_2}\right) ,
\end{aligned}
\end{equation}
where $0 < \eps, \tilde \eps \ll 1$. We note that system \eqref{eq:B1_sphere_cylinder_scaling_chart} is singularly perturbed (i.e.~slow-fast) with respect to $0 < \tilde \eps \ll 1$, and regularly perturbed by $0 < \eps \ll 1$. The layer problem is obtained by letting $\tilde \eps \to 0$, and given by
\[
\begin{split}
    x' &= 0 , \\
    y_2' &= \frac{1}{\theta_2}(1-\Hat{H}(x)\Hat{H}(y_2)-\kappa_2) ,
\end{split}
\]
which does not depend on $\eps$. A direct calculation shows that the critical manifold can be written as
$$\mathcal{S} := \{(x,\varphi_0(x,\kappa_2)) \in \R^2 : x \in (x^-, \infty) \},$$
where $x^- := \ln((1-\kappa_2) / \kappa_2)$, the function $\varphi_0(x,\kappa_2)$ is smooth and uniquely determined by
\[
\Hat{H}(x) \Hat{H}(\varphi_0(x,\kappa_2)) = 1 - \kappa_2  ,
\]
and the non-trivial eigenvalue associated with the linearisation along $\mathcal S$ is given by
$$\lambda=-(1-\kappa_2)\left(1-\frac{1-\kappa_2}{\Hat{H}(x)}\right).$$
Since $\hat H(x) \in (0,1)$ for all $x \in \R$ and $\kappa_2 < 1$ for all $(\kappa_1, \kappa_2) \in \Lambda$, $\lambda < 0$ and so $\mathcal S$ is normally hyperbolic and attracting. Fenichel theory therefore implies for any fixed compact interval $\mathcal J \subset (x^-, \infty)$, that the compact submanifold $\mathcal S \cap \{ x \in \mathcal J \}$ persists as a nearby attracting slow manifold
\[
\mathcal S_{\tilde \eps} = \left\{ (x, \varphi(x, \kappa_1, \kappa_2, \eps, \tilde \eps)) : x \in \mathcal J \right\} 
\]
where $\varphi(x, \kappa_1, \kappa_2, \eps, \tilde \eps ) = \varphi_0(x, \kappa_2) + \tilde \eps \varphi_{rem}(x, \kappa_1, \kappa_2, \eps, \tilde \eps)$, for a remainder function $\varphi_{rem}$ which is sufficiently smooth for our purposes; see e.g.~\cite{Fenichel_1979,Jones_1995,Wiggins2013} for details. Restricting system \eqref{eq:B1_sphere_cylinder_scaling_chart} to $\mathcal S_{\tilde \eps}$ leads to the following equation for the slow flow:
\begin{equation}
    \label{eq:B1_slow_flow}
    \Dot{x} = \frac{1}{\theta_1e^{\eps x}}\left( \Hat{H}(x) + (1-2 \Hat{H}(x)) 
\hat H(\varphi(x, \kappa_1, \kappa_2, \eps, \tilde \eps)) - \kappa_1 e^{\eps x} \right) .
\end{equation}
This problem is well-defined and regularly perturbed with respect to both $0 \leq \eps, \tilde \eps \ll 1$.

\begin{prop}
\label{prop:B1}
    Fix $(\kappa_1, \kappa_2) \in \Lambda$, a compact subinterval $\mathcal I \subset (0,1)$ and consider equation \eqref{eq:B1_slow_flow}. There exists $\eps_0, \tilde \eps_0 > 0$ and a $C^1$-function $\kappa_{1,sn}(\kappa_2,\eps, \tilde \eps) : \mathcal I \times [0,\eps_0) \times [0,\tilde \eps_0) \to \R$ such that there is a saddle-node bifurcation for $\kappa_1=\kappa_{1,sn}(\kappa_2,\eps,\tilde \eps)$. Moreover,
    \[
    \kappa_{1,sn}(\kappa_2,0,0) = 2\kappa_2-2+2\sqrt{1-\kappa_2} .
    \]
\end{prop}



\begin{proof}
%
 %
   We denote the right-hand side of \eqref{eq:B1_slow_flow} by $F(x,\kappa_1,\kappa_2,\eps,\tilde \eps)$ and verify the defining conditions
    \begin{equation}
    \label{eq:SN_cond}
        F(x_{sn},\kappa_{1,sn}(\kappa_2,0,0),\kappa_2,0,0) = F_{x}(x_{sn},\kappa_{1,sn}(\kappa_2,0,0),\kappa_2,0,0) = 0 ,
    \end{equation}
    at a point $(x, \kappa_1, \kappa_2, \eps, \tilde \eps) = (x_{sn}, \kappa_{1,sn}(\kappa_2, 0, 0), \kappa_2, 0, 0)$ with $\kappa_2 \in (0,1)$, as well as the combined transversality/nondegeneracy condition that the matrix
    \begin{equation}
    \label{eq:SN_genericity}
    \begin{pmatrix}
        F_{x_2} & F_{\kappa_1} \\
        F_{x_2 x_2} & F_{x_2 \kappa_1} 
    \end{pmatrix} \bigg|_{(x_{sn}, \kappa_{1,sn}(\kappa_2, 0, 0), \kappa_2, 0, 0)}
    \end{equation}
    has maximal rank. Direct calculations show that \eqref{eq:SN_cond} is satisfied for the unique $x_{sn}$ which satisfies
    \[
    \Hat{H}(x_{sn}) = 1 - \kappa_2 + \frac{\kappa_1}{2}.
    \]
    In order to verify that the matrix \eqref{eq:SN_genericity} has maximal rank, it suffices to show that $F_{x x} F_{\kappa_1} (x_{sn}, \kappa_{1,sn}(\kappa_2, 0, 0), \kappa_2, 0, 0) \neq 0$ (since $F_{x} (x_{sn}, \kappa_{1,sn}(\kappa_2, 0, 0), \kappa_2, 0, 0) = 0$ due to \eqref{eq:SN_cond}). Direct calculations yield
    \[
    \begin{split}
        F_{x x}(x_{sn}, \kappa_{1,sn}(\kappa_2, 0, 0), \kappa_2, 0, 0) &= \frac{1}{\theta_1} \left( \kappa_2-\frac{\kappa_1}{2} \right)^2 \left( 1-\kappa_2+\frac{\kappa_1}{2}+\frac{1-\kappa_2}{1-\kappa_2+\frac{\kappa_1}{2}} \right) > 0, \\ 
        F_{\kappa_1} (x_{sn}, \kappa_{1,sn}(\kappa_2, 0, 0), \kappa_2, 0, 0) &= -\frac{1}{\theta_1} < 0,  
    \end{split}
    \]
    where the former inequality holds for all $(\kappa_1, \kappa_2) \in \Lambda$, and therefore also along $\kappa_1 = \kappa_{1,sn}(\kappa_2,0,0)$ for all $\kappa_2 \in (0,1)$.
    
    The preceding arguments show that the reduced problem \eqref{eq:B1_slow_flow}$|_{\eps = \tilde \eps = 0}$ has a regular saddle-node bifurcation when $\kappa_1 = \kappa_{1,sn}(\kappa_2,0,0)$, for all $\kappa_2 \in (0,1)$. The persistence of the saddle-node bifurcation for $0 \leq \eps, \tilde \eps \ll 1$ follows from the implicit function theorem and regular perturbation theory. The details are similar to those presented in the proofs of Propositions \ref{prop:B2} and \ref{prop:B2_SN_Hopf} above, so we omit them here for brevity.
\end{proof}

We are now in a position to prove Theorem \ref{thm:B1}.


\begin{proof}[Proof of Theorem \ref{thm:B1}]
   Theorem \ref{thm:B1} follows directly from Proposition \ref{prop:B1} after composing the blow-down maps associated with spherical and cylindrical blow-up transformations that were applied in order to obtain system \eqref{eq:B1_sphere_cylinder_scaling_chart}. The blow-down forces the restriction to $\eps = \eps_1 > 0$ and $\tilde \eps = \tilde \eps_2 > 0$ in Theorem \ref{thm:B1}.  
\end{proof}

We conclude this section with some brief observations that are relevant for understanding the global dynamics. In particular, we are interested in a connection from the endpoint of $\mathcal{S}_2^+$ to the saddle equilibrium $q_s$.
This is easiest to view in chart $\bar y=1$ $(y=\eta_1,\, \Tilde{\eps}=\eta_1 \hat \eps)$ of the cylindrical blow-up \eqref{eq:cylindrical_blow-up_B1} applied to \eqref{eq:B2_eqns_again}, where the dynamics is governed by
\begin{equation} \label{eq:B1_sphere_cylinder_upper_chart}
    \begin{aligned} 
    x_2'&=\frac{\eta_1}{\theta_1e^{\eps x_2}}\left( \Hat{H}(x_2)+ \Hat{H}(\frac{1}{\hat \eps}) -2 \Hat{H}(x_2) \Hat{H}(\frac{1}{\hat \eps}) -\kappa_1 e^{\eps x_2}\right)\\
    \Hat{\eps}'&=-\frac{\hat \eps}{\theta_2e^{\eps \eta_1 }}\left( 1-  \Hat{H}(x_2) \Hat{H}(\frac{1}{\hat \eps}) -\kappa_2 e^{\eps \eta_1}\right)\\
    \eta_1'&=\frac{\eta_1}{\theta_2e^{\eps \eta_1 }}\left( 1-  \Hat{H}(x_2) \Hat{H}(\frac{1}{\hat \eps}) -\kappa_2 e^{\eps \eta_1}\right).
\end{aligned}
\end{equation}
In terms of $X:=\Hat{H}(x_2)$ and $Y:=\Hat{H}(\frac{1}{\Hat{\eps}})$, the `extension' of $\mathcal S_2^+$ from the blow-up cylinder onto the sphere is found to be contained within the invariant line defined by $X = 1 - \kappa_1$, and $\tilde \eps = 0$; see Figure \ref{fig:B1_heteroclinic}. 
%


\begin{figure}[t!]
    \centering
    \begin{subfigure}[t]{0.3\textwidth}
        \centering
            \includegraphics[width=1\linewidth]{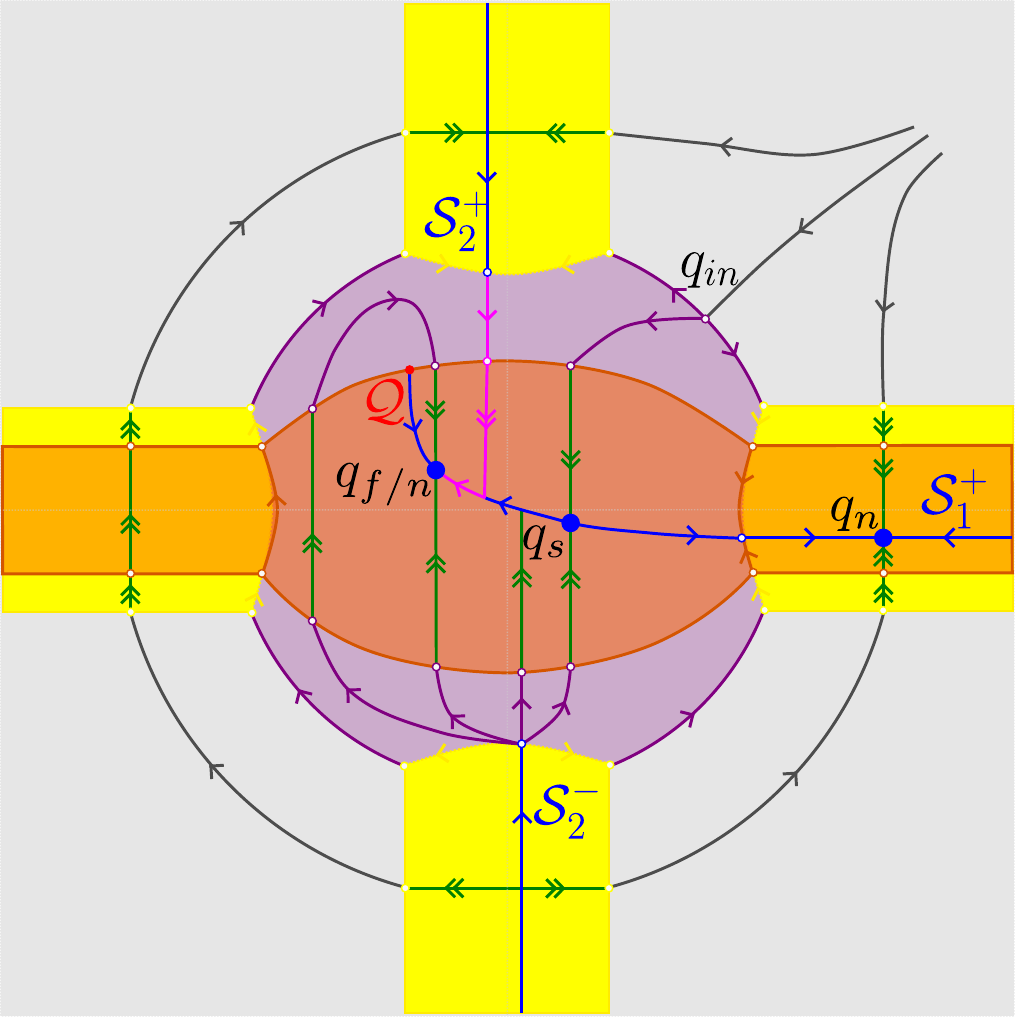}
    \caption{Region II.}
    \label{fig:B1_II}
    \end{subfigure}%
    \hfill
    \begin{subfigure}[t]{0.3\textwidth}
        \centering
            \includegraphics[width=1\linewidth]{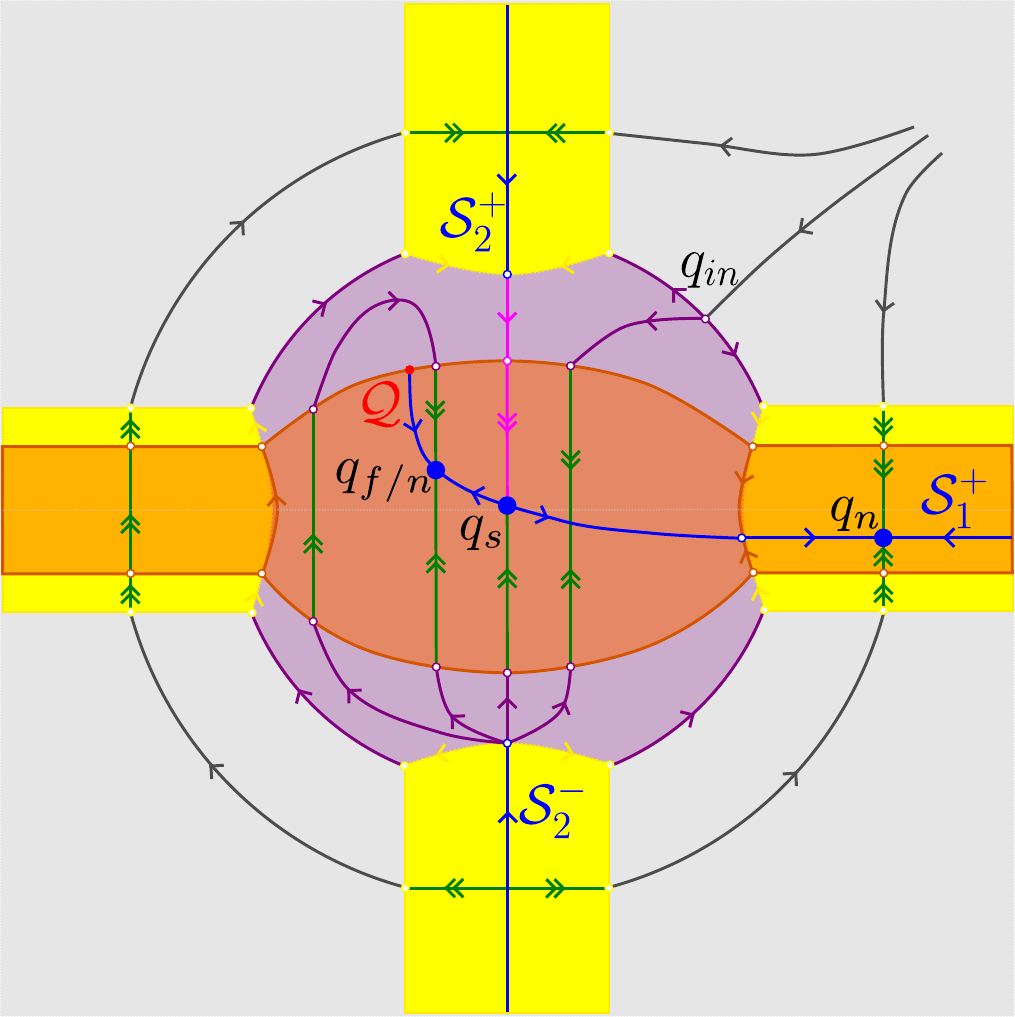}
    \caption{Border between region II and III in Figure \ref{fig:B1_bifurcation_diagram}, i.e., $\kappa_1=1/2$ and $\kappa_2 \in (3/4,1)$.}
    \label{fig:B1_II_III}
    \end{subfigure}
    \hfill
    \begin{subfigure}[t]{0.3\textwidth}
        \centering
        \includegraphics[width=1\linewidth]{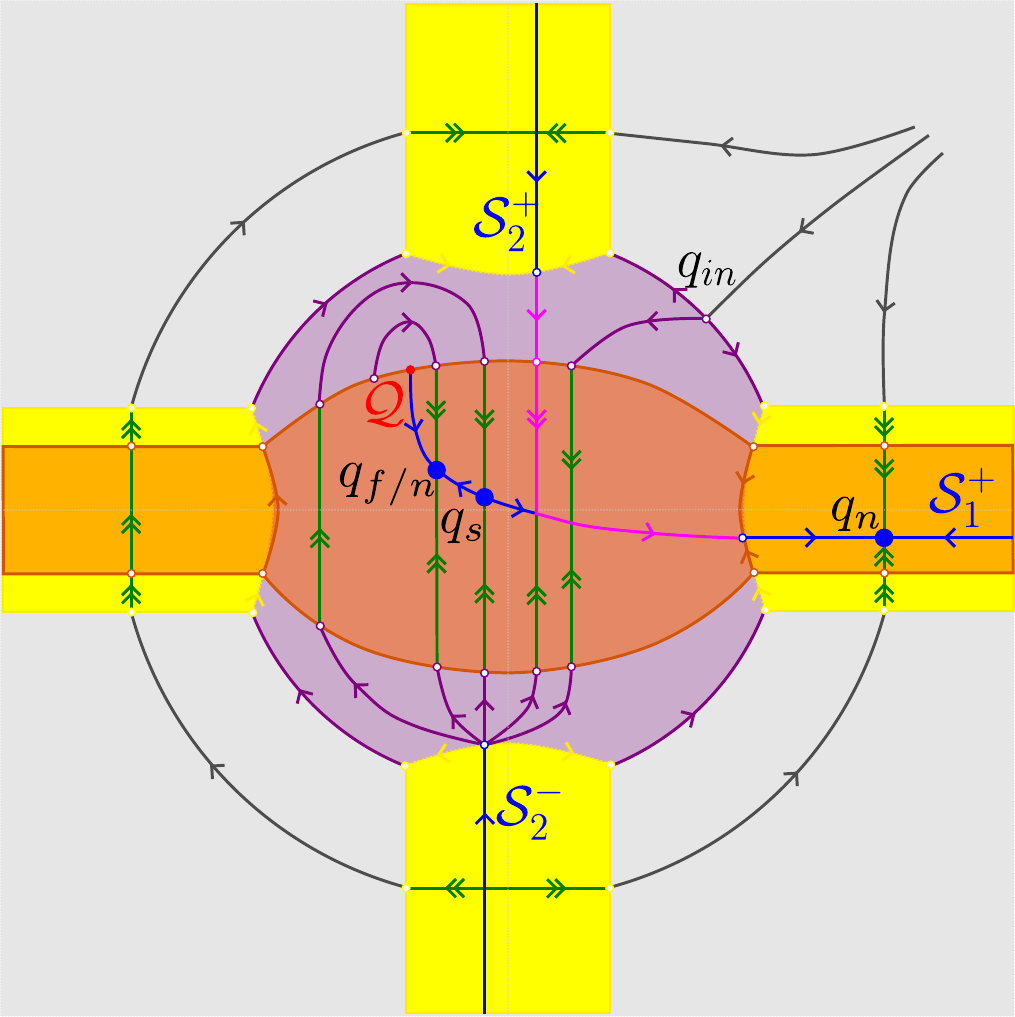}
    \caption{Region III.}
    \label{fig:B1_III}
    \end{subfigure}
    \caption{Selected phase portraits of \eqref{eq:B1_sphere_cylinder_scaling_chart} emphasizing the importance for the global dynamics of the transition between region II and III in the $B_1$ bifurcation diagram (vertical dotted black line  in Figure \ref{fig:B1_bifurcation_diagram}), as indicated in Remark \ref{rem:main_global_dynamics}. 
    In (a): Phase portrait in region II. The stable manifold of $q_s$, denoted $w^s(q_s)$, acts as a separatrix and dissects the state space. Solutions with initial conditions to the right of $w^s(q_s)$ are (on the singular level) attracted to $\mathcal{S}_1^+$ and follow the slow flow towards $q_n$. Solutions with initial conditions to the left of $w^s(q_s)$ are attracted to $\mathcal{S}_2^+$ and follow its extension on the sphere (magenta) towards $q_{f/n}$. 
    In (b): The critical case on the border line between region II and III where a heteroclinic connection (magenta) between the (singular) extension of $\mathcal{S}_2^+$ and $q_s$ exists (in fact $\mathcal{S}_2^+$ and its extension is part of the stable manifold $w^s(q_s)$).
    In (c): Phase portrait in region III. The extension (magenta) of $\mathcal{S}_2^+$ on the sphere (and therefore all solutions attracted by $\mathcal{S}_2^+$) converge to $q_n$. The basin of attraction of $q_{f/n}$ shrinks to the region enclosed by $w^s(q_s)$ which emanates from $\mathcal{S}_2^-$.}
    \label{fig:B1_heteroclinic}
\end{figure}

The coordinates of the two equilibria on the critical manifold $\mathcal{S}$ are given by \eqref{eq:eq_coords}. In particular, there is a heteroclinic connection between the endpoint of $\mathcal{S}_2^+$ and $q_s$ ($q_{f/n}$) when $1-\kappa_1=X^+$ ($1-\kappa_1=X^-$). Solving these two expressions leads to 
$$\kappa_2-\frac{3\kappa_1}{2}=\pm \sqrt{(1-\kappa_2+\frac{\kappa_1}{2})^2+\kappa_2-1}$$
with solutions given by 
$$\kappa_1=\kappa_2,\quad \kappa_1=\frac{1}{2}.$$
The former is part of the boundary of $\Lambda$ and will therefore be ignored. The latter is feasible if $\kappa_2-\frac{3\kappa_1}{2}>0$ ($\kappa_2-\frac{3\kappa_1}{2}<0$) which for $\kappa_1=\frac{1}{2}$ reduces to $\kappa_2>\frac{3}{4}$ ($\kappa_2<\frac{3}{4}$). These calculations justify the observations in Remark \ref{rem:heteroclinic} which pertain to the boundary between regions II and III in the bifurcation diagram shown in Figure \ref{fig:B1_bifurcation_diagram}. In particular, the basins of attraction associated with $q_n$ and $q_{f/n}$ are expected to change significantly as $\kappa_1$ is increased over a neighbourhood of $\kappa_1 = 1/2$ when $\kappa_2 \in (3/4, 1)$, see Figure \ref{fig:B1_heteroclinic}.

\section{\texorpdfstring{Region $B_3$}{Region B3}} \label{sec:Region_B3}
It remains to consider the region $B_3$. In order to do so, we work with system \eqref{eq:fund_prob_P2}, i.e.~in the parameter chart $\mathcal P_2$, and consider both $0 <  \tilde \eps_1, \eps_2 \ll 1$ as small parameters. We start by considering the geometry and dynamics away from $I = \Sigma_1 \cap \Sigma_2$, the (smoothly perturbed) PWS dynamics near $I$, and the extension of locally invariant manifolds within the switching layers associated with $\Sigma_2^\pm$ into a neighbourhood about $I$. Following this, further geometric blow-ups in both parameter and variable space will be applied in order to prove Theorems \ref{thm:B3_BT} and \ref{thm:B3}.



\subsection{Geometric blow-up, regularised PWS dynamics and invariant manifolds}

As with the analysis in $B_1$, we can build upon the $B_2$ analysis presented in Section \ref{sec:Region_B2}. We saw in Section \ref{sub:B2_geometry} that by starting with the extended system \eqref{eq:P1_extended} and applying a sequence of blow-up transformations which consisted of (i) a spherical blow-up of the intersection point $\Sigma_1 \cap \Sigma_2 \cap \{0\}$, followed by (ii) cylindrical blow-ups along $\Sigma_1^\pm \times \{0\}$ and $\Sigma_2^\pm \times \{0\}$, 
the loss of smoothness could be resolved for any fixed $\tilde \eps_2 > 0$. An analogous result can be proven for any fixed $\tilde \eps_1 > 0$ in parameter chart $\mathcal P_1$, after applying the same sequence of blow-ups to the extended system obtained by appending $\dot \eps_2 = 0$ to system \eqref{eq:fund_prob_P2}. 
As $\tilde \eps_1 \to 0$, however, there is an additional loss of smoothness along the vertical axes in chart $\mathcal K_2$ and the `birds eye charts' associated with the cylindrical blow-ups of $\Sigma_2^\pm$. This is reminiscent of, but distinct from, the loss of smoothness along horizontal (as opposed to vertical) axes when we let $\tilde \eps_2 \to 0$ in the $B_1$ analysis. Naturally, at least three more cylindrical blow-ups will be necessary to resolve this.

As noted above, we start with the extended system obtained by appending $\dot \eps_2 = 0$ to system \eqref{eq:fund_prob_P2}, i.e.~
\begin{equation}
\label{eq:fund_extended_P2}
    \begin{split}
    \Dot{x}&=H(x,\tilde \eps_1 \eps_2,\theta_1)+H(y,\eps_2,\theta_2)-2 H(x,\tilde \eps_1 \eps_2,\theta_1) H(y,\eps_2,\theta_2)-\kappa_1 \theta_1^{-1} x , \\
        \Dot{y}&=1-H(x,\tilde \eps_1 \eps_2,\theta_1) H(y,\eps_2,\theta_2)- \kappa_2 \theta_2^{-1} y , \\
        \dot \eps_2 &= 0 ,
    \end{split}
\end{equation}
and apply a spherical blow-up which is analogous to \eqref{eq:sphere_exponential_blow-up}, namely
\begin{equation}
    \label{eq:spherical_blow-up_B3}
    r \geq 0, \  (\bar x, \bar y, \bar \eps_2) \in \mathbb S^2 \mapsto
    \begin{cases}
        x = \theta_1e^{r \Bar{x}} , \\
        y = \theta_2e^{r \Bar{y}} , \\
        \eps_2 = r \Bar{\eps}_2 .
    \end{cases}
\end{equation}
We permit a slight abuse of notation by reusing our notation for the directional charts as follows:
\[
\mathcal K_1^\pm : \bar x = \pm 1, \qquad 
\mathcal K_2 : \bar \eps_2 = 1, \qquad 
\mathcal K_3^\pm : \bar y = \pm 1.
\]
We will mostly be interested in the geometry and dynamics in the scaling chart $\mathcal K_2$. Before moving onto $\mathcal K_2$, however, we shall prove a lemma that will allow us to complete the proof of Proposition \ref{prop:node}. We state it for the following system in chart $\mathfrak K_2$, which is obtained after blowing up the degenerate line $\Sigma_1^+$ in chart $\mathcal K_1^+$ ($\bar x = 1$) via a blow-up analogous to \eqref{eq:B2_Sigma_2^+_blow-up_transformation} (except with $\eps_2 = r_1 \eps_{21}$ in $\mathcal K_1^+$ and $\eps_{21} = \rho \bar \eps_{21}$), and considering the resulting (desingularised) equations in the $\bar \eps_{21} = 1$ chart:
\begin{equation}
\label{eq:K12_system}
    \begin{split}
        r_1' &= r_1 \rho_2 G\left( \frac{1}{\rho_2 \tilde \eps_1}, y_{12}, r_1 \right) , \\
        y_{12}' &= \theta_2^{-1} e^{-r_1 \rho_2 y_{12}} \left( 1 - \hat H\left( \frac{1}{\rho_2 \tilde \eps_1} \right) \hat H(y_{12}) - \kappa_2 e^{r_1 \rho_2 y_{12}} \right) , \\
        \rho_2' &= - \rho_2^2 G\left( \frac{1}{\rho_2 \tilde \eps_1}, y_{12}, r_1 \right)  .
\end{split}
\end{equation}

\begin{lem}
\label{lem:q_node_B3}
    Fix $(\kappa_1, \kappa_2) \in \Lambda$ and consider system \eqref{eq:K12_system}. The following assertions are true:
    \begin{enumerate}
        \item[(i)] The line
        \[
        \mathcal{S}_1^+ := \left\{ \left(r_1, \ln \frac{1-\kappa_2}{\kappa_2}, 0 \right) : r_1 \geq 0 \right\} 
        \]
        defines a critical manifold which is normally hyperbolic and attracting when considered as a critical manifold for the planar system which governs the flow in $\{ \rho_2 = 0 \}$.
        \item[(ii)] The reduced flow on $\mathcal S_1^+$ has a unique, asymptotically stable equilibrium $q_n : ( \ln (\kappa_2 / \kappa_1 ), \ln((1-\kappa_2)/\kappa_2), 0)$.
    \end{enumerate}
\end{lem}

\begin{proof}
    The proof is similar to the proofs that we presented for Lemmas \ref{lem:B2_S_2^+} and \ref{lem:node_B1}, so we present fewer details. Assertion (i) follows from a direct calculation. Note that the computations to check assertion (i) are done in chart $\mathcal{K}_1^+$ before the secondary blow-up of the degenerate line $\Sigma_1^+$. 
    In order to prove Assertion (ii) we notice that $\eps_2 = r_1 \eps_{21} = r_1 \rho_2 \ll 1$ is a conserved quantity. This allows us to reduce to a planar slow-fast system with $0 < \eps_2 \ll 1$ when $r_1 > 0$. This system has a normally hyperbolic and attracting critical manifold $\{ (r_1, y_{12}) \in \R^+ \times \R : \hat H(y_{12}) = 1-\kappa_2 \}$, and the reduced flow on this set is governed by
    \[
    \dot r_1 = \theta_1^{-1} e^{-r_1} ( \kappa_2 - \kappa_1 e^{r_1} ) .
    \]
    Assertion (ii) follows from the fact that this equation has a unique, stable equilibrium at $r_1 = \ln (\kappa_2 / \kappa_1) > 0$ (note that $\kappa_2 > \kappa_1$ for all $(\kappa_1, \kappa_2) \in \Lambda$).
\end{proof}

Lemma \ref{lem:q_node_B3} justifies the geometry and dynamics on the right-hand side of the blow-up sphere, and on the blow-up cylinder to the right in Figure \ref{fig:B3_PWS}. Moreover, Assertion (ii) is the final ingredient in the proof of Proposition \ref{prop:node}: 

\begin{proof}[Proof of Proposition \ref{prop:node}]
    Lemmas \ref{lem:B2_S_2^+}, \ref{lem:node_B1}, and \ref{lem:q_node_B3} prove the existence of a stable node in the singular limit after blow-up in $B_2$, $B_1$ and $B_3$ respectively, and Fenichel theory ensures the persistence of this equilibrium for sufficiently small $\eps_1$ (in the case of Lemma \ref{lem:node_B1}) and $\eps_2$ (in the case of Lemmas \ref{lem:B2_S_2^+} and \ref{lem:q_node_B3}). Fixing $\beta_1$ and $\beta_2$ sufficiently small and large respectively ensures that the regions $B_i$ overlap, and blowing down yields the result.
\end{proof}

We turn now to the dynamics on the blow-up sphere. After a suitable rescaling of time, the equations in the scaling chart $\mathcal K_2$ associated with the spherical blow-up \eqref{eq:spherical_blow-up_B3} applied to \eqref{eq:fund_extended_P2} are given by
\begin{equation}\label{eq:B3_spherical_scaling_chart}
    \begin{aligned} 
    x_2'& = \frac{1}{\theta_1e^{\eps_2 x_2}}\left( \Hat{H}\left(\frac{x_2}{\Tilde{\eps}_1}\right)+ \Hat{H}(y_2) -2 \Hat{H}\left(\frac{x_2}{\Tilde{\eps}_1}\right) \Hat{H}(y_2) -\kappa_1 e^{\eps_2x_2}\right), \\
    y_2'&=\frac{1}{\theta_2e^{\eps_2 y_2}}\left( 1 - \Hat{H}\left(\frac{x_2}{\Tilde{\eps}_1}\right) \Hat{H}(y_2) -\kappa_2 e^{\eps_2y_2}\right) ,
\end{aligned}
\end{equation}
where $0 \leq \eps_2 \ll 1$ and $0 < \tilde \eps_1 \ll 1$. Due to \eqref{eq:Hill_function}, this system has a PWS singular limit when $\tilde \eps_1 \to 0$. We permit a slight abuse of notation by re-using the $Z^\pm$ notation used in the $B_1$ analysis in order to write this system as
\begin{equation} \label{eq:B3_PWS_limit}
    \begin{aligned} 
    z'&=\begin{pmatrix}
        x_2' \\
        y_2'
    \end{pmatrix} =
    \begin{cases}
        Z^+(x_2,y_2,\eps_2), & x_2>0, \\ 
        Z^-(x_2,y_2,\eps_2), & x_2<0,  
    \end{cases}\\
\end{aligned}
\end{equation}
where
\begin{equation} \label{eq:B3_Z+_Z-}
    \begin{aligned} 
        Z^+(x_2,y_2,\eps_2) :=
        \begin{pmatrix}
        \frac{1}{\theta_1e^{\eps_2 x_2}} (1 - \Hat{H}(y_2)  -\kappa_1 e^{\eps_2 x_2})\\
        \frac{1}{\theta_2e^{\eps_2 y_2}}( 1-  \Hat{H}(y_2)  -\kappa_2 e^{\eps_2 y_2})\\
        \end{pmatrix}, \qquad 
        Z^-(x_2,y_2,\eps_2) :=
        \begin{pmatrix}
        \frac{1}{\theta_1e^{\eps_2 x_2}}(  \Hat{H}(y_2)  -\kappa_1 e^{\eps_2 x_2}) \\
        \frac{1}{\theta_2e^{\eps_2 y_2}}( 1  -\kappa_2 e^{\eps_2 y_2}) \\
        \end{pmatrix},
    \end{aligned}
\end{equation}
and the switching manifold is
\[
\Sigma_{2}^1 := \{(0,y_2) : y_2 \in \R\}.
\]
Among other things, the next result shows that system \eqref{eq:B3_spherical_scaling_chart} features a \textit{regularised visible-invisible two-fold}.

\begin{lem} \label{lem:B3_PWS_sphere}
    Consider the limiting PWS system \eqref{eq:B3_PWS_limit} with $(\kappa_1, \kappa_2) \in \Lambda$. There exists an $\eps_{2,0} > 0$ such that the following assertions are true:
    \begin{enumerate}
        \item[(i)] The half-line $\{ (x_2, \ln \frac{1-\kappa_2}{\kappa_2} ) : x_2 > 0 \}$ is invariant under the flow induced by $Z^+|_{\eps_2 = 0}$, and $x_2' > 0$ along it;
        \item[(ii)] If $\kappa_1 \neq \frac{1}{2}$, then for all $\eps_2 \in [0,\eps_{2,0})$, there is a visible fold point $F_V : (0, \ln \frac{1-\kappa_1}{\kappa_1} ) \in \Sigma_2^1$, with tangency between $Z^+$ and $\Sigma_2^1$;
        \item[(iii)] If $\kappa_1 \neq \frac{1}{2}$, then for all $\eps_2 \in [0,\eps_{2,0})$, there is an invisible fold point $F_I : (0, \ln \frac{\kappa_1}{1-\kappa_1}) \in \Sigma_2^1$, with tangency between $Z^-$ and $\Sigma_2^1$;
        \item[(iv)] For $\kappa_1 = \frac{1}{2}$ and for all $\eps_2 \in [0,\eps_{2,0})$, there is a visible-invisible two-fold at $F_V = F_I : (0,0) \in \Sigma_2^1$.
    \end{enumerate}
\end{lem}
\begin{proof}
    Assertion (i) can be verified directly, and follows from the fact that $Z^+|_{\eps_2 = 0}$ is independent of $x_2$. Tangencies between $Z^+, Z^-$ and $\Sigma_2^1 = \{ (x_2, y_2) \in \R^2 : f(x_2,y_2) := x_2 = 0\}$ occur when
    \[
    Z^+f(0,y_2,\kappa_1,\kappa_2,\eps_2) = \theta_1^{-1} \left( 1 - \hat H(y_2) - \kappa_1 \right) = 0 , \qquad 
    Z^-f(0,y_2,\kappa_1,\kappa_2,\eps_2) = \theta_1^{-1} \left( \hat H(y_2) - \kappa_1 \right) = 0 ,
    \]
    i.e.~at the points $F_V$ and $F_I$ respectively. Note that the tangencies occur for all $(\kappa_1, \kappa_2) \in \Lambda$ and $\eps_2 \geq 0$. Moreover, direct calculations show that
    \[
    \begin{split}
        Z^+(Z^+f)(F_V, \kappa_1, \kappa_2, \eps_2) &= -\frac{1}{\theta_1\theta_2} \kappa_1 (1-\kappa_1)(\kappa_1-\kappa_2) + \mathcal O(\eps_2) , \\
        Z^-(Z^-f)(F_I, \kappa_1, \kappa_2, \eps_2) &= \frac{1}{\theta_1\theta_2}\kappa_1(1-\kappa_1)(1-\kappa_2) + \mathcal O(\eps_2) ,
    \end{split}
    \]
    as $\eps_2 \to 0$. Since $(1-\kappa_1)(\kappa_1-\kappa_2) < 0$ for all $(\kappa_1, \kappa_2) \in \Lambda$, it follows that for each fixed $(\kappa_1, \kappa_2) \in \Lambda$, there is an $\eps_{2,0} > 0$ such that $Z^+(Z^+f)(F_V, \kappa_1, \kappa_2, \eps_2) > 0$ and $Z^-(Z^-f)(F_I, \kappa_1, \kappa_2, \eps_2) > 0$ for all $\eps_2 \in [0,\eps_{2,0})$. This proves Assertions (ii)-(iii).

    It remains to prove Assertion (iv). Let $y_{2,V}(\kappa_1) := \ln ((1 - \kappa_1) / \kappa_1)$ and $y_{2,I}(\kappa_1) := \ln ( \kappa_1 / (1 - \kappa_1))$ denote the $y_2$-components of $F_V$ and $F_I$ respectively. Since $y_{2,V}(1/2) = y_{2,I}(1/2) = 0$, $F_V = F_I = \{(0,0)\}$ when $\kappa_1 = 1/2$, i.e.~the limiting PWS system \eqref{eq:B3_PWS_limit} has a visible-invisible two-fold singularity when $\kappa_1 = 1/2$. The associated transversality condition is satisfied since
    \[
    \frac{d}{d \kappa_1} (y_{2,V} - y_{2,I}) \bigg|_{\kappa_1 = 1/2} = -8 \neq 0 ,
    \]
    thereby proving Assertion (iv).
\end{proof}

Lemma \ref{lem:B3_PWS_sphere} justifies the sketch of the PWS geometry and dynamics on the blow-up sphere shown in Figure \ref{fig:B3_PWS}. Away from $\Sigma_2^1$, system \eqref{eq:B3_spherical_scaling_chart} is regularly perturbed in both $\tilde \eps_1$ and $\eps_2$, so the PWS description is expected to provide a good description of the leading order dynamics on compact subsets of the $(x_2, y_2)$-plane. The presence of fold and two-fold singularities in the limiting PWS system \eqref{eq:B3_PWS_limit}, implies the presence of \textit{regularised} fold and two-fold singularities for all $0 < \tilde \eps_1 \ll 1$ in system \eqref{eq:B3_spherical_scaling_chart}. A normal form associated with the regularised visible-invisible two-fold has recently been analysed in detail using geometric blow-up techniques in \cite{Kristiansen_2023} (we also refer again to \cite{Bonet2018,Kristiansen_2015} for earlier and closely related studies). This allows us to prove a result on the existence of canard-type cycles in system \eqref{eq:B3_spherical_scaling_chart}.

In order to formulate our result, we need to construct a family of \textit{singular} canard cycles 
$\Gamma_0^{y_{2,0}}(\eps_2)$ in the limiting PWS problem \eqref{eq:B3_PWS_limit} which is parameterised by the vertical coordinate $y_2 = y_{2,0}$; see Figure \ref{fig:B3_PWS_sing_cycles}. 
To this end, we notice that for each sufficiently small $\eps_2 \geq 0$ and $y_{2,0} < 0$ such that the point $(0,y_{2,0}) \in \Sigma_2^1$ lies between $q_s$ and $(0,0)$ (the location of the two-fold when $\kappa_1 = 1/2$), the forward flow of $(0,y_{2,0})$ under $Z^-(x_2,y_2,\eps_2)$ when $\kappa_1 \approx 1/2$ and $\kappa_2 \in (3/4, 1)$ has a first return point $(0, \xi(y_{2,0}, \eps_2))$. Such a construction is always possible when $y_{2,0} < 0$ and $|\kappa_1 - 1/2|$ are sufficiently small (we refer again to \cite{Kristiansen_2023} for further details on an analogous construction). Thus for each $\eps_2 \geq 0$ sufficiently small, $\kappa_2 \in (3/4,1)$ and $\kappa_1 \approx 1/2$, candidate singular orbits can be defined via
\begin{equation}
\label{eq:sing_cycles}
    \Gamma_0^{y_{2,0}}(\eps_2) := \{ (0, y_2) : y_2 \in [y_{2,0}, \xi(y_{2,0}, \eps_2)] \} \cup \mathcal O_0^{y_{2,0}}(\eps_2) ,
\end{equation}
where the left-most component is the sliding component in $\Sigma_2^1$ and the right-most component $\mathcal O_0^{y_{2,0}}(\eps_2)$ is the (simply connected) orbit segment of $Z^-(x_2,y_2,\eps_2)$ which connects the points $(0, y_{2,0})$ and $(0, \xi(y_{2,0}, \eps_2))$. 
Now we can formulate the following result.

\begin{figure}
    \centering
    \includegraphics[width=0.4\linewidth]{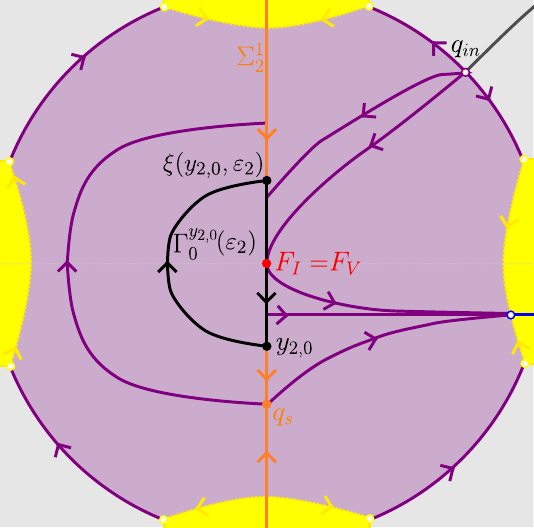}
    \caption{A representative PWS singular cycle $\Gamma_0^{y_{2,0}}(\eps_2)$, as defined in \eqref{eq:sing_cycles}. For each fixed and sufficiently small $\eps_2 \geq 0$, one can construct a family of cycles parameterised by the minimal vertical coordinate $y_{2,0}$, where $(0,y_{2,0}) \in \Sigma_2^1$. Each cycle has two components: a sliding component along $\Sigma_2^1$ and a regular component $\mathcal O_0^{y_{2,0}}(\eps_2)$ induced by the flow of the vector field $Z^-(x_2, y_2, \eps_2)$ when $\kappa_1 = 1/2$ and $\kappa_2 \in (3/4, 1)$, which connects to the first return point $(0, \xi(y_{2,0}, \eps_2)) \in \Sigma_2^1$. This is analogous to the construction in \cite{Kristiansen_2023}.}
    \label{fig:B3_PWS_sing_cycles}
\end{figure}

\begin{prop} \label{prop:B3_canard_uniqueness}
    Fix $\kappa_2 \in (3/4,1)$ and $y_{2,0} < 0$ sufficiently small. Then there exist $\tilde \eps_{1,0} > 0$, $\eps_{2,0}>0$ such that for each $\eps_2 \in [0,\eps_{2,0})$, there is a smooth function $\kappa_{1,c}^{y_{2,0}} : [0, \tilde \eps_{1,0}) \to \R$ with $\kappa_{1,c}^{y_{2,0}}(0) = 1/2$ such that for all $\tilde \eps_1 \in (0, \tilde \eps_{1,0})$, system \eqref{eq:B3_spherical_scaling_chart}$|_{\kappa_1 = \kappa_{1,c}^{y_{2,0}}(\tilde \eps_1)}$ has a periodic orbit $\Gamma_{\tilde \eps_1}^{y_{2,0}}(\eps_2)$ that is isolated, hyperbolic and Hausdorff close to $\Gamma_0^{y_{2,0}}(\eps_2)$ as $\tilde \eps_1 \to 0$.
   %
\end{prop}

\begin{proof}
    We need to check the conditions for \cite[Theorem 3.1]{Kristiansen_2023}. In order to do so, we make a preliminary transformation $(x_2, y_2) = (v, -u)$ (i.e.~a counterclockwise rotation by $\pi / 2$), and use the fact that $\hat H(-u) = 1 - \hat H(u)$, in order to arrive at a transformed system which satisfies
    \begin{equation}
        \label{eq:regularised_pws}
        \begin{split}
            u'&=\frac{1}{\theta_2}\left( -1+ \Hat{H} \left(\frac{v}{\Tilde{\eps}_1} \right) - \Hat{H}(u) \Hat{H}\left(\frac{v}{\Tilde{\eps}_1} \right) +\kappa_2 + \mathcal{O}(\eps_2) \right) ,\\
            v'&= \frac{1}{\theta_1}\left( 1 - \Hat{H}(u) - \Hat{H} \left( \frac{v}{\Tilde{\eps}_1} \right) + 2  \Hat{H}(u) \Hat{H} \left( \frac{v}{\Tilde{\eps}_1} \right) -\kappa_1 + \mathcal{O}(\eps_2) \right) ,
        \end{split}
    \end{equation}
    as $\eps_2 \to 0$. We note that the $\mathcal O(\eps_2)$ terms are smooth. System \eqref{eq:regularised_pws} is in the relevant general form \cite[eq.~(2.1)]{Kristiansen_2023}. In particular, the PWS singular limit when $\tilde \eps_1 \to 0$ can be written as
    \begin{equation} \label{eq:B3_canard_PWS_limit}
        \begin{aligned} 
        z'  &=
        \begin{pmatrix}
            u' \\
            v'
        \end{pmatrix}
        =
        \begin{cases}
            Z^+(u,v,\kappa_1,\kappa_2,\eps_2), & v>0 , \\
            Z^-(u,v,\kappa_1,\kappa_2,\eps_2), & v<0 ,
        \end{cases}
    \end{aligned}
    \end{equation}
    where the switching manifold is $\{ v = 0 \}$ and the vector fields $Z^\pm := (U^\pm, V^\pm)$ satisfy
    \begin{equation} \label{eq:B3_canard_Z+_Z-}
        \begin{aligned} 
            Z^+(u,v,\kappa_1,\kappa_2,\eps_2)&=\begin{pmatrix}
            U^+(u,v,\kappa_1,\kappa_2,\eps_2)\\
            V^+(u,v,\kappa_1,\kappa_2,\eps_2)
        \end{pmatrix}=\begin{pmatrix}
        \frac{1}{\theta_2} (- \Hat{H}(u)  +\kappa_2) + \mathcal O(\eps_2) \\
        \frac{1}{\theta_1}(\Hat{H}(u)  -\kappa_1) + \mathcal O(\eps_2) 
        \end{pmatrix} , \\
          Z^-(u,v,\kappa_1,\kappa_2,\eps_2)&=\begin{pmatrix}
            U^-(u,v,\kappa_1,\kappa_2,\eps_2)\\
            V^-(u,v,\kappa_1,\kappa_2,\eps_2)
        \end{pmatrix}=\begin{pmatrix}
    \frac{1}{\theta_2}(  \kappa_2-1 ) + \mathcal O(\eps_2) \\
    \frac{1}{\theta_1}( 1 - \Hat{H}(u) - \kappa_1) + \mathcal O(\eps_2) \\
    \end{pmatrix} ,
    \end{aligned}
    \end{equation}
    as $\eps_2 \to 0$. System \eqref{eq:regularised_pws} can be viewed as a regularised PWS system with regularisation function $\hat H$. More precisely, it can be written as
    $$\begin{pmatrix}
        u' \\
        v'
    \end{pmatrix} = \hat H\left( \frac{v}{\tilde \eps_1} \right) Z^+(u, v, \kappa_1, \kappa_2, \eps_2) + \left(1 - \hat H\left( \frac{v}{\tilde \eps_1} \right) \right) Z^-(u, v, \kappa_1, \kappa_2, \eps_2).$$
    Now we have to check the following assumptions, which are directly implied by the corresponding assumptions from \cite{Kristiansen_2023}:
    \begin{enumerate}
        \item[(A1)] $\lim_{z \to -\infty} \hat H(z) = 0$ and $\lim_{z \to \infty} \hat H(z) = 1$;
        \item[(A2)] $\hat H'(z) > 0$ for all $z \in \R$;
        \item[(A3)] $\hat H(z)$ is smooth as $z \to \pm \infty$;
        \item[(A4)] There exists a neighbourhood of $(u,v) = (0,0)$ for which the regularised visible-invisible two-fold is type $VI_3$, i.e., the stable sliding region ($u<0$) and the unstable sliding region ($u>0$) are connected by the flow of the sliding vector field $U_{sl}>0$ in a neighbourhood of the origin. 
    \end{enumerate}
    (A1)-(A3) are immediate, so it remains to check (A4). A direct calculation using \cite[equation (2.5)]{Kristiansen_2023} shows that the sliding vector field $U_{sl}(u, \kappa_1, \kappa_2, \eps_2)$ on $\{ v = 0 \}$ satisfies
    \begin{equation} \label{eq:B3_sliding}
        u' = U_{sl}\left(u, \frac{1}{2}, \kappa_2, \eps_2\right) = 
        \frac{-(\Hat{H}(u)-\frac{1}{2})(\Hat{H}(u)+1-\kappa_2)}{\theta_2(2\Hat{H}(u)-1)} + \mathcal O(\eps_2) 
    \end{equation}
    as $\eps_2 \to 0$. A direct calculation reveals the existence of a unique attracting equilibrium corresponding to $q_s$ at $u = u^\ast(\kappa_2, \eps_2)$, with asymptotics $u^\ast(\kappa_2, \eps_2) = \ln((2\kappa_2 - 1) / (2 - 2 \kappa_2)) + \mathcal O(\eps_2)$ as $\eps_2 \to 0$. For each fixed $\kappa_2 \in (3/4, 1)$, there exists an $\eps_{2,0} > 0$ and a constant $\Omega > 0$ such that $u^*(\kappa_2, \eps_2) > \Omega > 0$ for all $\eps_2 \in [0,\eps_{2,0})$. Thus, there is a neighbourhood about $(0,0)$ on which $U_{sl}(u, 1/2, \kappa_2, \eps_2) > 0$ and (A4) is satisfied.

    Given (A1)-(A4), \cite[Thm.~3.1]{Kristiansen_2023} asserts the existence of $k \in \mathbb N_+$ canard-type limit cycles, where $k-1$ is the number of simple zeros of a particular \textit{slow divergence integral} (eqn.~(3.1) in \cite{Kristiansen_2023}). In our case, the limiting slow divergence integral when $\eps_2 = 0$ simplifies to
    \[
    I(u_0,0) = \int \limits_{-u_0}^{u_0} \frac{\Hat{H}(u)-\frac{1}{2}}{2\kappa_2-1-\Hat{H}(u)}\, du,
    \]
    where $u_0 \in (0, \Omega]$ and we have included a zero in the second argument to indicate that $I(u_0,0) = \lim_{\eps_2 \to 0} I(u_0, \eps_2)$, where $I(u_0, \eps_2)$ denotes the slow divergence integral for $\eps_2 \geq 0$. Since the domain of the integral $I(u_0,0)$ is symmetric, we can show $I(u_0,0)>0$ on $(0,\Omega]$ by a simple comparison argument, i.e.,
    $$f(u)>-f(-u),\, u \in (0,\Omega] \implies I(u_0,0) = \int \limits_{-u_0}^{u_0} f(u)\, du>0 ,$$
    where here
    \begin{equation} \label{eq:B3_denominator}
        f(u)=\frac{\Hat{H}(u)-\frac{1}{2}}{2\kappa_2-1-\Hat{H}(u)}, \qquad f(-u)=\frac{\frac{1}{2}-\Hat{H}(u)}{2\kappa_2-2+\Hat{H}(u)} ,
    \end{equation}
    and we have used $\Hat{H}(-u)=1-\Hat{H}(u)$. From the choice of $\Omega$ it follows that the two denominators in \eqref{eq:B3_denominator} are positive for $u \in (0,\Omega]$ and $\kappa_2 \in (3/4,1)$. A straightforward computation then simplifies $f(u)>-f(-u)$ to $\Hat{H}(u)>1/2$, which is true for $u \in (0,\Omega]$.
    The comparison argument above ensures that $I(u_0,0)$ is strictly positive in $(0,\Omega]$ and we conclude that $k=1$.
    Since \eqref{eq:regularised_pws} and therefore \eqref{eq:B3_canard_PWS_limit} depend regularly on $\eps_2$ this implies that the inequality $I(u_0,\eps_2)>0$ is fulfilled also for $\eps_2 \in [0,\eps_{2,0})$ for $\eps_{2,0}>0$ small enough.
    
    Therefore, it follows from \cite[Thm~3.1]{Kristiansen_2023} that for each $\eps_2 \in [0,\eps_{2,0})$ and $u_0 \in (0,\Omega]$, there exists a smooth function $\kappa_{1,c}^{u_0}(\Tilde{\eps}_1)$ with $\kappa_{1,c}^{u_0}(0)=1/2$ such that \eqref{eq:regularised_pws} has a periodic orbit $\Gamma_{\Tilde{\eps}_1}^{u_0}(\eps_2)$ for $0<\Tilde{\eps}_1 \ll 1$. The above-mentioned result also guarantees that the periodic orbit $\Gamma_{\Tilde{\eps}_1}^{u_0}(\eps_2)$ is isolated, hyperbolic and Hausdorff close to the canard cycle $\Gamma_{0}^{u_0}(\eps_2)$ which corresponds to $\Gamma_{0}^{-y_{2,0}}(\eps_2)$ after applying the rotation $(x_2, y_2) = (v, - u)$. Proposition \ref{prop:B3_canard_uniqueness} itself follows after inverting the coordinate rotation.
\end{proof}

Proposition \ref{prop:B3_canard_uniqueness} provides us with quite a bit of detail about the dynamics when $\kappa_2 \in (3/4, 1)$ and $\kappa_1 = \tfrac{1}{2} + \mathcal O(\tilde \eps_1)$, i.e.~in the scaling regime (S3). However, we need to do more work if we would like to understand (i) the global geometry and dynamics, and (ii) the behaviour in scaling regime (S2) close to $(\kappa_1, \kappa_2) = (1/2, 3/4)$. In order to address this, we need more detailed information about the behaviour close to $\Sigma_2^1$. To that end, we need to apply more blow-ups.

We start by considering the extended system obtained by appending $\tilde \eps_1'=0$ to system \eqref{eq:B3_spherical_scaling_chart}, which is degenerate along $\Sigma_2^1 \times \{0\}$. We then drop the subscripts to simplify the notation, and apply a cylindrical blow-up of the form
\begin{equation}
\label{eq:cylindrical_blow-up_B3}
\eta \geq 0, \ (\bar x, \bar{\tilde \eps} ) \in \mathbb S^1 \mapsto 
\begin{cases}
    x = \eta \bar x, \\
    \tilde \eps = \eta \bar{\tilde \eps}.
\end{cases}
\end{equation}
Note that we have permitted another slight abuse of notation by reusing the radial variable $\eta$ (as in \eqref{eq:cylindrical_blow-up_B1}). A complete treatment would involve an analysis in the three coordinate charts defined via $\bar x = \pm 1$ and $\bar{\tilde \eps} = 1$. We shall present a detailed analysis in the scaling chart $\bar{\tilde \eps} = 1$; the details of the dynamics in the $\bar x = \pm 1$ charts are omitted for brevity, and because they are not necessary for proving Theorem \ref{thm:B3}. Further details of the normal form analysis in these charts can also be found in \cite{Kristiansen_2023}.

We write the local coordinates in $\bar{\tilde \eps} = 1$ as $x = \tilde \eps x_2$. After a suitable desingularisation (multiplication by $\tilde \eps$), the equations are given by
\begin{equation}\label{eq:B3_spherical_cylinder_scaling_chart}
    \begin{aligned} 
    x_2' &= \frac{1}{\theta_1 e^{\tilde \eps \eps x_2}}\left( \Hat{H}(x_2)+ \Hat{H}(y) - 2 \Hat{H}(x_2) \Hat{H}(y) -\kappa_1 e^{\tilde \eps \eps x_2}\right) , \\
    y'&=\frac{\tilde \eps}{\theta_2e^{\eps y}}\left( 1 - \Hat{H}(x_2) \Hat{H}(y) -\kappa_2 e^{\eps y}\right) ,
\end{aligned}
\end{equation}
with $0 < \tilde \eps, \eps \ll 1$ (recall that we dropped the subscripts; in terms of our earlier notation, $\tilde \eps = \tilde \eps_1$ and $\eps = \eps_2$). This system is slow-fast with respect to the limit $\tilde \eps \to 0$, for all $\eps \geq 0$. The associated layer problem, obtained by setting $\tilde \eps = 0$, is given by
\begin{equation}
\label{eq:B3_layer_problem}
\begin{split}
x_2' &= \frac{1}{\theta_1}\left( \Hat{H}(x_2)+ \Hat{H}(y) -2 \Hat{H}(x_2) \Hat{H}(y) -\kappa_1\right) , \\
y_2' &= 0,
\end{split}
\end{equation}
which does not depend on $\eps$. The critical manifold has the form
\begin{equation} \label{eq:B3_critical manifold}
    \mathcal{S} := \{(x_2,y) \in \R^2 : G(x_2,y,\kappa_1) = 0\} ,
\end{equation}
where $G(x_2, y, \kappa_1) :=  \Hat{H}(x_2)+ \Hat{H}(y) - 2 \Hat{H}(x_2) \Hat{H}(y) -\kappa_1$. In general, $\mathcal S$ is given by the union of two curves which are non-intersecting unless $\kappa_1 = 1/2$ (which is precisely the critical value associated with the regularised two-fold identified above). We consider the cases $\kappa_1 \neq 1/2$ and $\kappa_1 = 1/2$ in turn. The geometry is sketched for $\kappa_1 < 1/2$, $\kappa_1 = 1/2$ and $\kappa_1 > 1/2$ in Figure \ref{fig:B3_kappa1_geometry}. 

When $\kappa_1 \neq 1/2$, $\mathcal S$ can be written as the disjoint union $\mathcal S = \mathcal S^a \cup \mathcal S^r$, where
\[
\mathcal S^a := \{ (x_2, y) \in \mathcal S : y > 0 \} , \qquad
\mathcal S^r := \{ (x_2, y) \in \mathcal S : y < 0 \} .
\]
The superscripts "a/r" stand for "attracting/repelling" respectively. This is justified by the fact that the non-trivial eigenvalue associated with the linearisation along $\mathcal{S}$ is given by
\[
\lambda = \frac{1}{\theta_1}\Hat{H}(x_2)(1-\Hat{H}(x_2))(1-2\Hat{H}(y)) ,
\]
which implies that $\mathcal S^a$ ($\mathcal S^r$) is normally hyperbolic and attracting (repelling). When $\kappa_1 = 1/2$, $\mathcal S$ takes the more degenerate form $\mathcal S = \mathcal S^r \cup \mathcal C \cup \mathcal S^a$, where
\[
\mathcal S^r := \{ (0, y) : y < 0 \} , \qquad 
\mathcal C := \{ (x_2, 0) : x_2 \in \R \}, \qquad
\mathcal S^a := \{ (0, y) : y > 0 \} .
\]
Direct calculations show that the critical manifold $\mathcal S^{a,+}$ ($\mathcal S^{r,-}$) is normally hyperbolic and attracting (repelling), and that $\mathcal C$ is degenerate (the non-trivial multiplier associated with the linearisation along $\mathcal C$ is identically zero).

We obtain the following result.

\begin{lem} \label{lem:B3_slow-manifold}
    Fix $(\kappa_1, \kappa_2) \in \Lambda$, $\eps \geq 0$ and $\delta \in (0,1)$. There exists an $\tilde \eps_0>0$ such that for all $\tilde \eps \in (0, \tilde \eps_0)$, system \eqref{eq:B3_spherical_cylinder_scaling_chart} has normally attracting/repelling slow manifolds $\mathcal{S}^{a/r}_{\tilde \eps}$ of the form
    \[
    \mathcal S_{\tilde \eps}^a := \left\{ (\varphi^a(y, \kappa_1, \kappa_2, \tilde \eps, \eps) , y) : y \in [\delta, 1/\delta] \right\} , \qquad
    \mathcal S_{\tilde \eps}^r := \left\{ (\varphi^r(y, \kappa_1, \kappa_2, \tilde \eps, \eps) , y) : y \in [-1/\delta, -\delta] \right\} ,
    \]
    where the functions $\varphi^{a/r}$ are $C^r$-smooth and satisfy $G(\varphi^{a/r}(y, \kappa_1, \kappa_2, \tilde \eps, \eps), y, \kappa_1) = \mathcal O(\tilde \eps)$ as $\tilde \eps \to 0$. Moreover,
    \[
    \varphi^{a/r}(y, 1/2, \kappa_2, \tilde \eps, \eps) = \mathcal O(\tilde \eps)
    \]
    as $\tilde \eps \to 0$.
\end{lem}

\begin{proof}
    This follows from Fenichel theory \cite{Fenichel_1979,Jones_1995,Wiggins2013} and the calculations which preceded the statement of the Lemma.
\end{proof}

The Fenichel slow manifolds $\mathcal S_{\tilde \eps}^{a/r}$ described by Lemma \ref{lem:B3_slow-manifold} blow-down to submanifolds of the locally invariant manifolds described by Lemma \ref{lem:B3_invariant_manifolds}, a proof of which is given 
below.

\begin{proof}[Proof of Lemma \ref{lem:B3_invariant_manifolds}]
   Lemma \ref{lem:B3_slow-manifold} describes the slow manifolds $\mathcal S^{a/r}_{\tilde \eps}$ on compact subsets in the scaling chart which covers the top of the blow-up cylinder on top of the blow-up sphere in Figures \ref{fig:B3_kappa1_geometry} and \ref{fig:B3_kappa1_0.5_kappa2_varied}, but it does not describe their connection to the manifolds denoted by $\mathcal S_2^\pm$ in these figures, which lie on the blow-up cylinders above and below the blow-up sphere. In order to complete the proof of Lemma \ref{lem:B3_slow-manifold}, one must (i) identify $\mathcal S^+_2$ and $\mathcal S^-_2$ as attracting and repelling slow manifolds respectively in suitable coordinate charts, and (ii) show that $\mathcal S_2^+$ connects to $\mathcal S^a_{\tilde \eps}$ and that $\mathcal S_2^-$ connects to $\mathcal S^r_{\tilde \eps}$. For the attracting branches $\mathcal S^+_2$ and $\mathcal S^a_{\tilde \eps}$, both (i) and (ii) can be shown using the following set of equations, which are obtained after writing system \eqref{eq:fund_extended_P2} in chart $\mathcal K_3^+$ coordinates 
   $$x=\theta_1e^{r_3x_3},\quad y=\theta_2e^{r_3},\quad \eps_2=r_3\eps_{23},$$ 
   associated with the spherical blow-up \eqref{eq:spherical_blow-up_B3}, and subsequently applying a cylindrical blow-up 
   \begin{equation}
\label{eq:cylindrical_blow-up_B3_upper}
\rho \geq 0, \ (\bar x_3, \bar{ \eps}_{23} ) \in \mathbb S^1 \mapsto 
\begin{cases}
    x_3 = \rho_2 \bar x_{3}, \\
    \eps_{23} = \rho_2 \bar{ \eps}_{23},
\end{cases}
\end{equation}
to resolve the degeneracy corresponding to $\Sigma_2^+$. The (desingularised) equations in the $\bar \eps_{23} = 1$ chart, with coordinates $(x_{32}, r_3, \rho_2)$ given by $x_3=\rho_2x_{32}$ and $\eps_{23}=\rho_2$ are still degenerate along $\{ (0, r_3, \rho_2) : r_3 \geq 0 \}$ when $\tilde \eps_1 = 0$ due to a persistent loss of smoothness; see Figure \ref{fig:B3_PWS}. This is resolved by a cylindrical blow-up applied to the relevant extended system (including $\tilde \eps_1' = 0$) of the form
   \[
   \eta \geq 0 , \ (\bar x_{32}, \bar{\tilde \eps}_1) \in \mathbb S^1 \mapsto
   \begin{cases}
       x_{32} = \eta \bar x_{32} , \\
       \tilde \eps_1 = \eta \bar{\eps}_1 .
   \end{cases}
   \]
   Considering the scaling chart $\bar{\tilde \eps}_1=1$ with coordinates $(x_{322}, r_3, \rho_2)$, where $x_{32}=\Tilde{\eps}_1x_{322}$, and dropping the subscripts, leads to the following (desingularised) system of equations:
\[
\begin{split}
    x'&=\frac{1}{\theta_1e^{r\Tilde{\eps}\rho x}}[\Hat{H}(x)+\Hat{H}(1/\rho)-2\Hat{H}(x)\Hat{H}(1/\rho)-\kappa_1e^{r\Tilde{\eps} \rho x}] , \\
    r'&=\frac{r\rho\Tilde{\eps}}{\theta_2 e^r}[1-\Hat{H}(x)\Hat{H}(1/\rho)-\kappa_2e^r] , \\
    \rho'&=-\frac{\rho^2\Tilde{\eps}}{\theta_2 e^r}[1-\Hat{H}(x)\Hat{H}(1/\rho)-\kappa_2e^r],
\end{split}
\]
where $0 \leq \Tilde{\eps}\ll 1$ is a singular perturbation parameter. The layer problem is given by
$$x'=\frac{1}{\theta_1}[\Hat{H}(x)+\Hat{H}(1/\rho)-2\Hat{H}(x)\Hat{H}(1/\rho)-\kappa_1].$$
The critical manifold is given by
\begin{equation} \label{eq:B3_critical manifold_upper}
    \mathcal{S}_0 := \{(x,r,\rho) \in \R^3 : \Hat{H}(x)+ \Hat{H}(1/\rho) - 2 \Hat{H}(x) \Hat{H}(1/\rho) -\kappa_1 = 0\} ,
\end{equation}
and is attracting for all $\rho \geq 0$ (the non-trivial eigenvalue is given by $\lambda = -\theta_1^{-1} \hat H(x) ( 1 - \hat H(x))(1-2\hat H(1/\rho))) < 0$ for all $\rho \geq 0$). 
Within the invariant plane $\{\rho=0\}$ we recover the attracting critical manifold $\mathcal{S}_2^+=\{\Hat{H}(x)=1-\kappa_1\}$. Fenichel theory implies that this perturbs to the slow manifold required to prove (i). For $\rho>0$ we recover the critical manifold geometry of \eqref{eq:B3_critical manifold} in the current coordinates for all $r\geq 0$. The two one-dimensional critical manifold branches $\mathcal{S}_2^+$ and $\mathcal{S}^a$ above can be identified in the current coordinate chart within the two-dimensional attracting critical manifold $\mathcal{S}_0$. In particular, we have $\mathcal{S}_2^+=\mathcal{S}_0|_{\rho=0}$ and $\mathcal{S}^a=\mathcal{S}_0|_{r=0}$, which connect at $(\ln( (1-\kappa_1)/\kappa_1),0,0)$. Therefore, 
Fenichel theory implies (ii) and justifies Figures \ref{fig:B3_kappa1_geometry} and \ref{fig:B3_kappa1_0.5_kappa2_varied}. 
The argument for the repelling branches $\mathcal S^-_2$ and $\mathcal S^r_{\tilde \eps}$ is analogous, so we omit it for brevity. 
\end{proof}

\begin{rem} \label{rem:B3_non-hyperbolic_kappa1_neq_0.5}
    Certain features in Figures \ref{fig:B3_kappa1_geometry} and \ref{fig:B3_kappa1_0.5_kappa2_varied} have been included without proof for brevity. Here we note the presence of two degenerate points $\mathcal Q_I$ and $Q_V$ in particular, which are correlated with exponential/beyond all orders tangencies between the fast foliation and the extension of $\mathcal S^{a/r}$ to the edge of the blow-up cylinder when $|x_2| \to \infty$ (see Figure \ref{fig:B3_kappa1_geometry}). We expect the blow-up analyses associated with these points to be quite involved, for reasons analogous to those given in Remark \ref{rem:Q_degeneracy} above. We omit this analysis for brevity, and because the local dynamics near $\mathcal Q_I$ and $\mathcal Q_V$ is not important in the proof of our main results.
\end{rem}

\subsection{Saddle-node bifurcation and reduced dynamics when \texorpdfstring{$\kappa_1 \neq \frac{1}{2}$}{kappa1 not 0.5}} \label{sec:B3_kappa_1neq0.5}


The qualitative geometry and dynamics of the layer problem \eqref{eq:B3_layer_problem} is the same as in the blow-up analysis of the regularised two-fold presented in \cite{Kristiansen_2023}, however, there is an important difference in the reduced problem. A direct calculation shows that the reduced flow on $\mathcal S^{a/r}$ 
is governed by
\begin{equation}
\label{eq:B3_reduced}
    \dot y = \theta_2^{-1} e^{-\eps y} \left[ 1 - \frac{\kappa_1 - \hat H(y)}{1 - 2 \hat H(y)} \hat H(y) - \kappa_2 e^{\eps y} \right] , \qquad y \neq 0 ,
\end{equation}
which is a regular $\mathcal O(\eps)$ perturbation of
\begin{equation}
\label{eq:B3_reduced_limit}
    \dot y = \theta_2^{-1} \left[ 1 - \frac{\kappa_1 - \hat H(y)}{1 - 2 \hat H(y)} \hat H(y) - \kappa_2 \right] , \qquad y \neq 0 ,
\end{equation}
on compact domains for $y$. 
%
We obtain a similar expression for the equilibria -- when they exist -- as in region $B_2$:
\begin{equation}
\label{eq:B3_eqs}
    \Hat{H}^{\pm}(y)=1-\kappa_2+\frac{\kappa_1}{2}\pm \sqrt{(1-\kappa_2+\frac{\kappa_1}{2})^2+\kappa_2 -1},
\end{equation}
cf equation \eqref{eq:eq_coords}. Indeed, the equilibria corresponding to solutions of this equation can be viewed as the `extensions' of the equilibria $q_s$ and $q_{f/n}$ identified in $B_2$ in Section \ref{sec:Region_B2} up to $\tilde \eps \to 0$. We shall therefore denote them here by
%
%
$q_s : (x^+, y^+)$ and $q_{f/n} :(x^-, y^-)$, where $x^\pm$ and $y^\pm$ are uniquely determined by the equations $G(\varphi^{a/r}(y, \kappa_1, \kappa_2, 0, 0), y, \kappa_1) = 0$ and \eqref{eq:B3_eqs} respectively.

\begin{rem}
    \label{rem:eq_relations}
    Explicit formulae for $(x^\pm, y^\pm)$ exist, but they are lengthy and not necessary in what follows. One can use these formulae to show that $x^-<x^+$ and $y^+<y^-$. In the following we shall use these relations without proof.
\end{rem}

The following result describes the creation/destruction of these equilibria in a saddle-node bifurcation under variation in $(\kappa_1, \kappa_2)$, away from the degenerate point $(\kappa_1, \kappa_2) = (1/2, 3/4)$. In order to state it, we write $\mathcal I := \mathcal I^- \cup \mathcal I^+$, where $\mathcal I^\pm \subset \R$ are compact intervals satisfying $\mathcal I^- \subset (0,3/4)$ and $\mathcal I^+ \subset (3/4, 1)$.

\begin{prop} \label{prop:B3_saddle-node}
    Fix $\mathcal I$ (defined above) and consider system \eqref{eq:B3_spherical_cylinder_scaling_chart}. There exists $\tilde \eps_0, \eps_0 > 0$ and a $C^1$-smooth function $\kappa_{1,sn}:\mathcal I \times [0,\tilde \eps_0) \times [0,\eps_0) \to \R$ such that there is a saddle-node bifurcation when $\kappa_1 = \kappa_{1,sn}(\kappa_2,\tilde \eps,\eps)$. Moreover,
    \[
    \kappa_{1,sn}(\kappa_2,0,0) = 2\kappa_2-2+2\sqrt{1-\kappa_2} 
    \]
    for all $\kappa_2 \in \mathcal I$.
%
\end{prop}

\begin{proof}
The proof is similar to the proof of Theorem \ref{thm:B1}: after verifying the existence of a regular saddle-node bifurcation in the limiting reduced problem \eqref{eq:B3_reduced_limit}, the persistence as a saddle-node bifurcation in system \eqref{eq:B3_spherical_cylinder_scaling_chart} can be proven using the implicit function theorem (to perturb in $\eps$) and Fenichel theory (to perturb in $\tilde \eps$). The latter argument relying on Fenichel theory only applies when $\kappa_1$ is bounded away from $\kappa_1 = 1/2$; when $\kappa_1 = 1/2$ the equilibrium lies on the non-hyperbolic critical manifold $\mathcal C$. This is the reason that the saddle-node curve is only identified for $\kappa_1 = \kappa_{1,sn}(\kappa_2, \tilde \eps, \eps)$ with $\kappa_2 \in \mathcal I$ (alternatively, one can show that the genericity condition \eqref{eq:SN_genericity} for a one-dimensional saddle-node bifurcation breaks down when $\kappa_1 = 1/2$ in system \eqref{eq:B3_reduced_limit}).
\end{proof}

Proposition \ref{prop:B3_saddle-node} justifies the bifurcation diagram in Figure \ref{fig:B3_bifurcation_diagram_singular}, which is sketched in the limiting case with $\tilde \eps = \eps = 0$. 
The saddle-node curve shown in orange divides $\Lambda$ into two open regions. Let $\bar \kappa_{1,sn}(\kappa_2)$ denote the extension of $\kappa_{1,sn}(\kappa_2, 0, 0)$ over $\kappa_2 \in (0,1)$. In the left region denoted by I, we have $\kappa_1 < \bar \kappa_{1,sn}(\kappa_2) < 1/2$. Here there are no equilibria and the orientation of the critical manifolds $\mathcal S^{a/r}$ is as shown in Figure \ref{fig:B3_kappa1_less_0,5}. On the right, when $\kappa_1 > \bar \kappa_{1,sn}(\kappa_2)$, we distinguish three different `sub-regions' depending on whether $\kappa_1 \gtrless 1/2$ and $\kappa_2 \gtrless 3/4$, in order to account for the fact that (i) the relative orientation of the critical manifolds $\mathcal S^{a/r}$ varies in an important way as $\kappa_1$ is varied over $1/2$, and (ii) the equilibria $q_s$ and $q_{f/n}$ are found to lie on different branches depending on $(\kappa_1, \kappa_2)$. The relevant sub-regions are defined as follows:
\[
\begin{split}
    IIi &:= \left\{ (\kappa_1, \kappa_2) \in \Lambda : \kappa_1 > 1/2 \right\} , \\
    IIii &:= \left\{ (\kappa_1, \kappa_2) \in \Lambda : \kappa_1 \in (\bar \kappa_{1,sn}(\kappa_2), 1/2), \kappa_2 < 3/4 \right\} ,\\
    V &:= \left\{ (\kappa_1, \kappa_2) \in \Lambda : \kappa_1 \in (\bar \kappa_{1,sn}(\kappa_2), 1/2), \kappa_2 > 3/4 \right\} , \\
\end{split}
\]
see Figure \ref{fig:B3_bifurcation_diagram_singular} (we refrain from introducing ``region III'' and ``region IV" at this point because we anticipate additional regions close to $\kappa_1 = 1/2$). Using equation \eqref{eq:B3_eqs} and the relations in Remark \ref{rem:eq_relations}, we infer that 
\[
0 < \Hat{H}^{\pm}(y) < \frac{1}{2} \iff y^{\pm} < 0 
\]
in V. This implies that both $q_s, q_{f/n} \in \mathcal S^r$. Similarly,
\[
\frac{1}{2} < \Hat{H}^{\pm}(y) < 1 \iff y^{\pm} >0
\]
in IIii, which in this case implies that $q_s, q_{f/n} \in \mathcal S^a$. In IIi we have that
\[
\Hat{H}(y^-) < \frac{1}{2} < \Hat{H}(y^+) \iff y^+<0<y^-,
\]
which implies that $q_{f/n} \in \mathcal S^a$ and $q_s \in \mathcal S^r$. 
Figure \ref{fig:B3_area_I_II_III_regular} shows the phase portrait in singular limit, for $(\kappa_1,\kappa_2)$ values in regions V, IIi and IIii (the phase portrait in region I, which is not shown, is similar to that in region V or IIii except that there are no equilibria on either $\mathcal S^{a/r}$).

\begin{figure}[h!]
    \centering
    \begin{subfigure}{0.3\textwidth}
         \includegraphics[width=\linewidth]{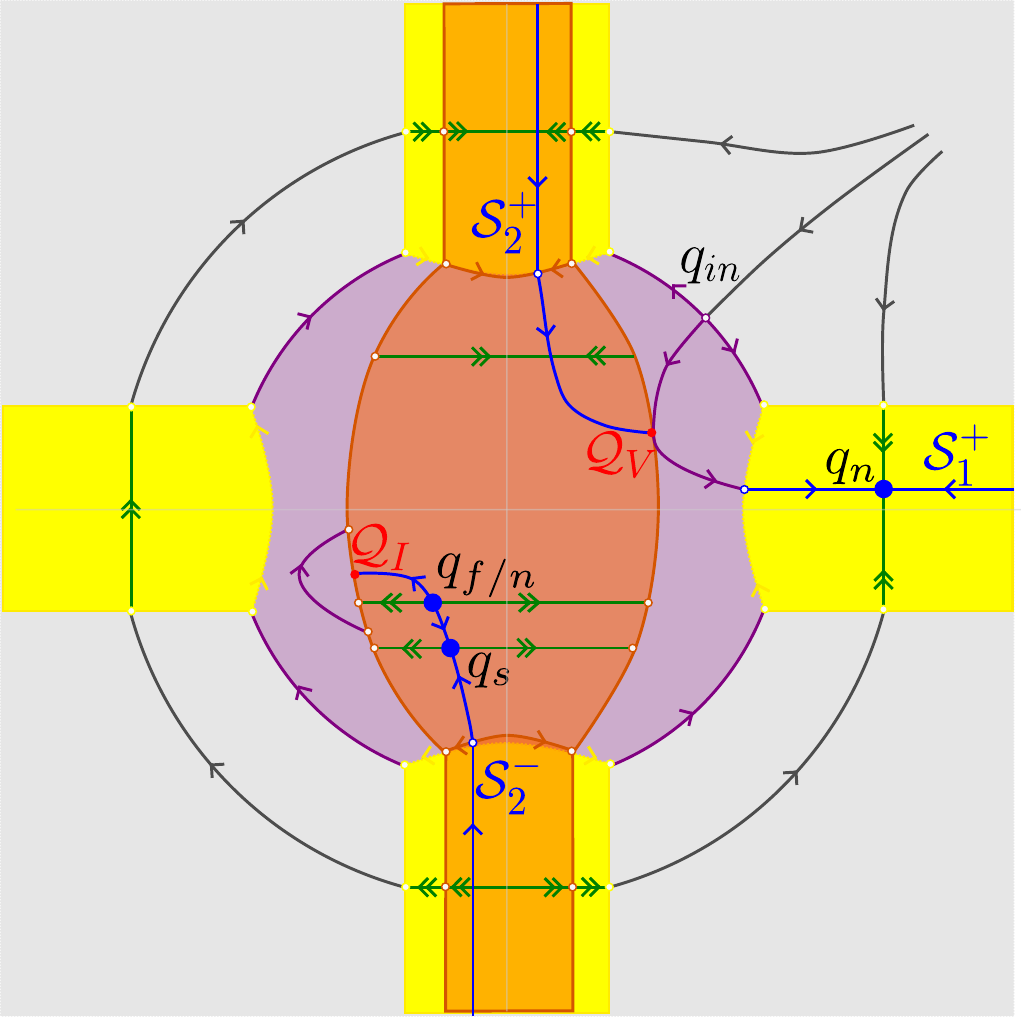}
    \caption{Region V.}
    \label{fig:B3_V}
    \end{subfigure}
    \hfill
    \begin{subfigure}{0.3\textwidth}
         \includegraphics[width=\linewidth]{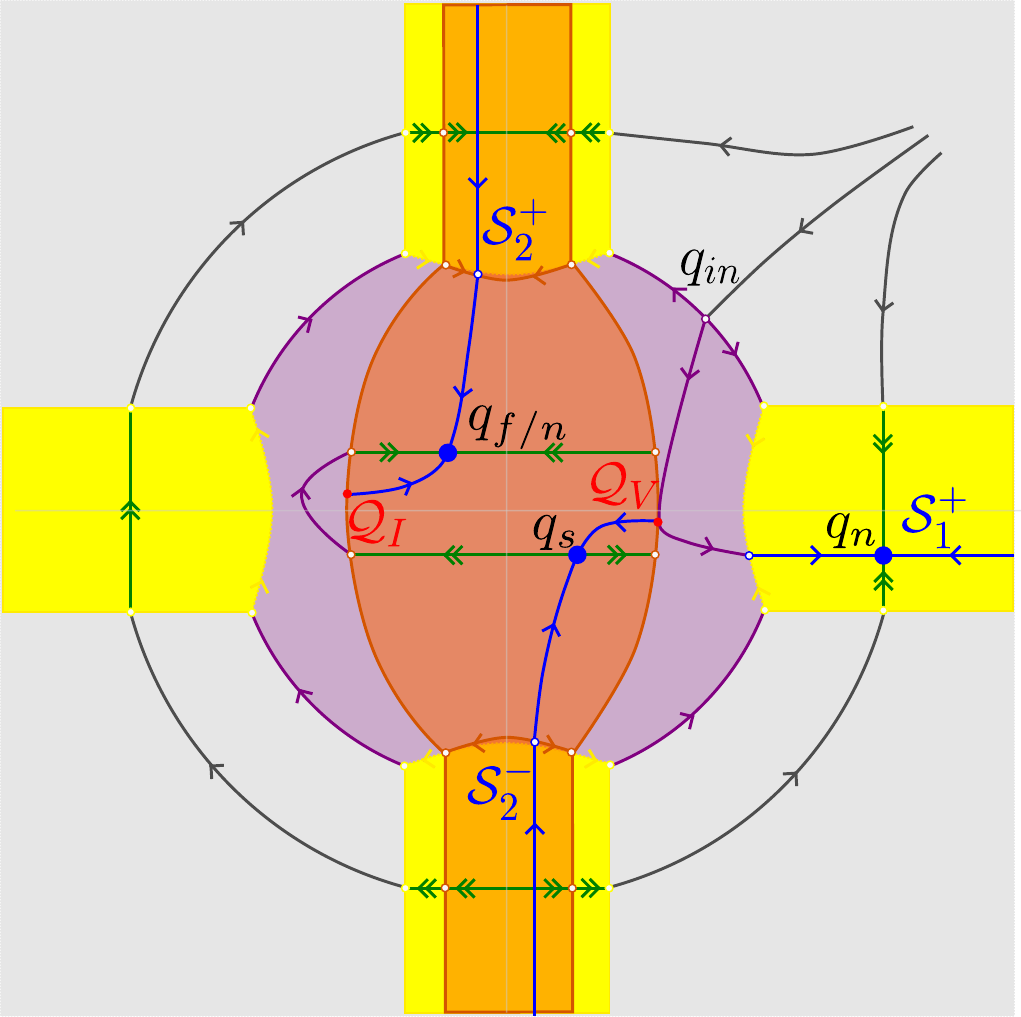}
    \caption{Region IIi.}
    \label{fig:B3_IIi}
    \end{subfigure}
    \hfill
    \begin{subfigure}{0.3\textwidth}
         \includegraphics[width=\linewidth]{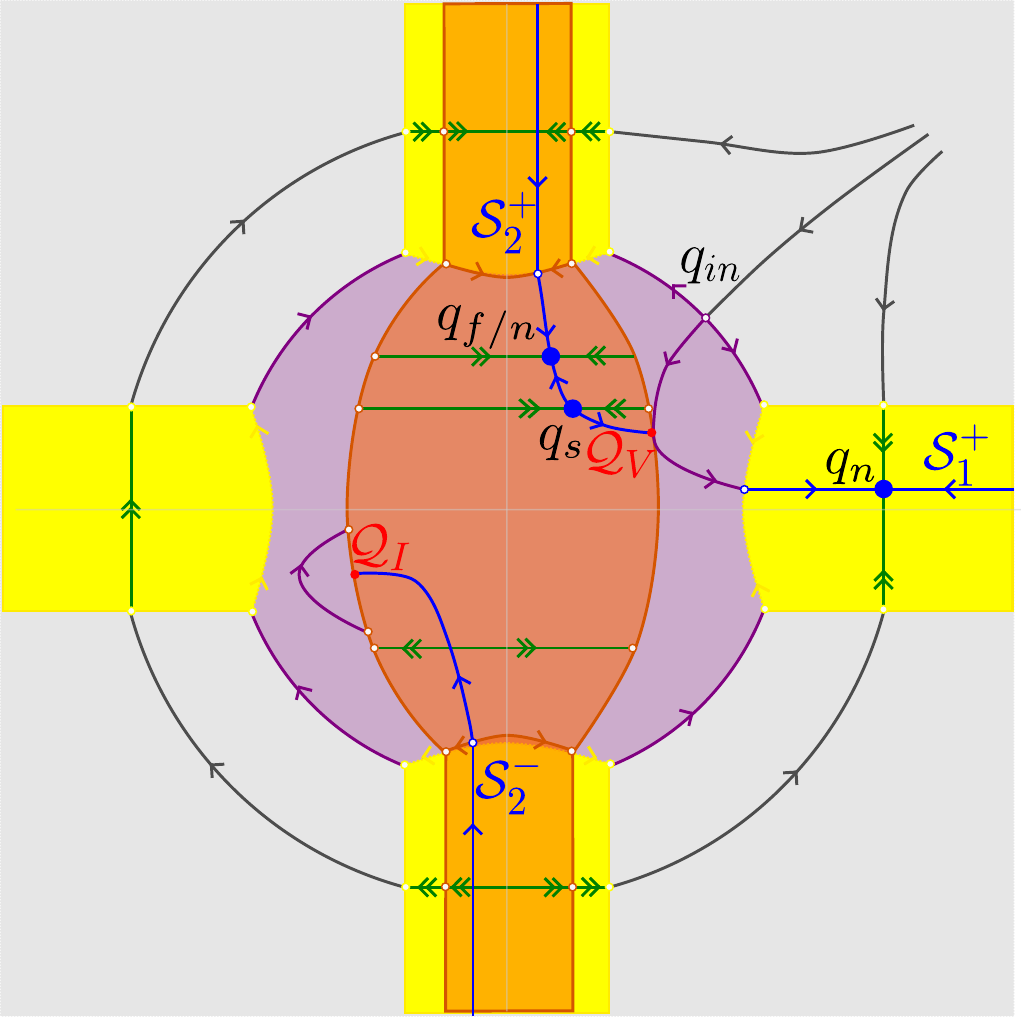}
    \caption{Region IIii.}
    \label{fig:B3_IIii}
    \end{subfigure}
    \caption{Phase portraits in the parameter regions V, IIi and IIii shown in (a), (b) and (c), respectively. As seen already in Figure \ref{fig:B3_kappa1_0.5_kappa2_varied}, the critical manifolds $\mathcal S^{a/r}$ reverse their orientation as $\kappa_1$ is varied over $1/2$; this accounts for the different geometry in region IIi (subfigure (b) and the only case shown here with $\kappa_1 > 1/2$). Here was also shown that the location and type of the equilibria $q_s$ and $q_{f/n}$ depends upon the region: $q_s, q_{f/n} \in \mathcal S^r$ in region V (a), $q_s \in \mathcal S^r$ and $q_{f/n} \in \mathcal S^a$ in region IIi (b), and $q_s, q_{f/n} \in \mathcal S^a$ in region IIii (c).}
    \label{fig:B3_area_I_II_III_regular}
\end{figure}


The (extended) saddle-node curve along $\kappa_1 = \bar \kappa_{1,sn}(\kappa_2)$ shown in Figure \ref{fig:B3_bifurcation_diagram_singular} is degenerate at the red point $(\kappa_1, \kappa_2) = (1/2,3/4)$ which lies at the turning point, which we have (somewhat preemptively) labeled "BT" because it lies at the limiting location of the Bogdanov-Takens point identified in the $B_2$ analysis (consider letting $\tilde \eps_1 = 1 / \tilde \eps_2 \to 0$ in equation \eqref{eq:BT_pars_B2}). In order to gain a better understanding of the dynamics and bifurcation structure near $(\kappa_1, \kappa_2) = (1/2, 3/4)$ and close to the line segment $\{(1/2, \kappa_2) : \kappa_2 \in (3/4,1) \}$ which emanates from this point, we need to apply subsequent blow-ups.

\subsubsection{Singular Bogdanov-Takens bifurcation}
\label{sec:B3_BT_blow-up}

We start by dropping the subscripts in system \eqref{eq:B3_spherical_cylinder_scaling_chart} and blowing up the degenerate point $(\kappa_1,\kappa_2,\tilde \eps)=(\frac{1}{2},\frac{3}{4},0)$ in parameter space. 
We apply a weighted spherical blow-up defined via
\begin{equation} \label{eq:B3_parameter_BU_spherical}
    \sigma \geq 0, \ (\Bar{\kappa}_1,\Bar{\kappa}_2,\bar{\tilde \eps}) \in \mathbb{S}^2 \mapsto 
    \begin{cases}
        \kappa_1=\frac{1}{2}+\sigma^2 \Bar{\kappa}_1, \\
        \kappa_2=\frac{3}{4}+\sigma \Bar{\kappa}_2, \\
        \tilde \eps = \sigma \bar{\tilde \eps} ,
    \end{cases}
\end{equation}
and consider the dynamics in the `scaling chart' $\bar{\tilde \eps} = 1$, with local coordinates denoted by
\[
\kappa_1 = \frac{1}{2} + \tilde \eps^2 \gamma_1, \qquad \kappa_2 = \frac{3}{4} + \tilde \eps \gamma_2,
\]
see Figure \ref{fig:B3_spherical_parameter_blow-up}.
In this chart, system \eqref{eq:B3_spherical_cylinder_scaling_chart} 
becomes
\begin{equation} \label{eq:B3_parameter_scaling_chart}
\begin{split}
    x' &= \mathcal F(x,y,\gamma_1,\tilde \eps,\eps) , \\
    y'&= \tilde \eps \mathcal G(x,y,\gamma_2,\tilde \eps,\eps) ,
\end{split}
\end{equation}
where $\mathcal F, \mathcal G$ satisfy
\[
\begin{split}
    \mathcal F(x,y,\gamma_1,\tilde \eps,\eps) &= \frac{1}{\theta_1}[\Hat{H}(x)+\Hat{H}(y)-2\Hat{H}(x)\Hat{H}(y)-\frac{1}{2}-\tilde \eps^2\gamma_1]+\mathcal{O}(\tilde \eps \eps) , \\
    \mathcal G(x,y,\gamma_2,\tilde \eps,\eps) &= \frac{1}{\theta_2}[1-\Hat{H}(x)\Hat{H}(y)-\frac{3}{4}-\tilde \eps\gamma_2+\mathcal{O}(\eps)] ,
\end{split}
\]
as $\eps \to 0$. 
\begin{figure}[t!]
    \centering
    \begin{subfigure}{0.49\textwidth}
         \includegraphics[width=1\linewidth]{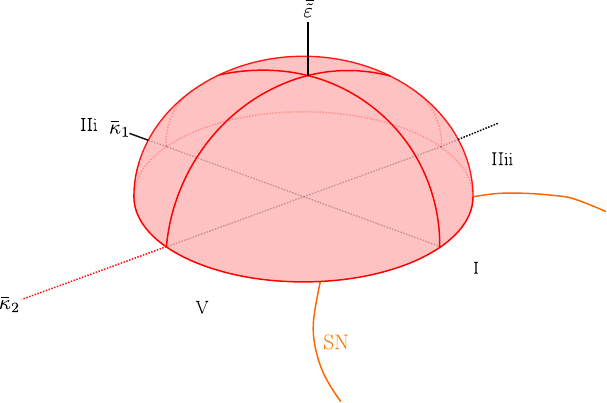}
    \caption{}
    \label{fig:B3_spherical_parameter_blow-up}
    \end{subfigure}
    \hfill
    \begin{subfigure}{0.49\textwidth}
        \includegraphics[width=1\linewidth]{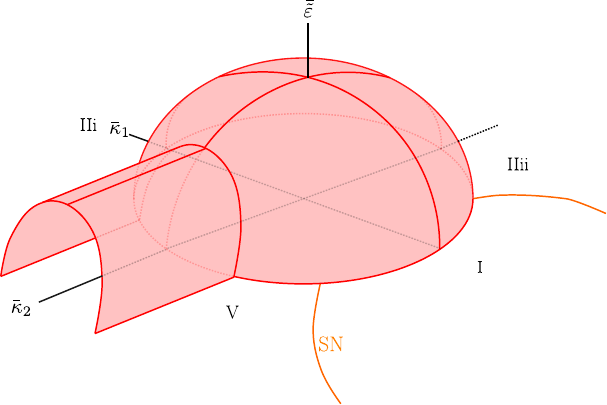}
    \caption{}
    \label{fig:B3_parameter_blow-up_cylinder_attached}
    \end{subfigure}
    \caption{(a) Spherical parameter blow-up of the degenerate point $BT: (\kappa_1,\kappa_2, \tilde \eps) =(1/2,3/4,0)$ (red point in Figure \ref{fig:B3_bifurcation_diagram_singular}) via \eqref{eq:B3_parameter_BU_spherical}. (b) Subsequent cylindrical blow-up of the degenerate line segment $\{(1/2,\kappa_2,0), \, \kappa_2 \in (3/4,1)\}$ (dotted-red in Figure \ref{fig:parameter_blow-up_2D_sequence}) via \eqref{eq:B3_parameter_BU_cylinder}.}
    \label{fig:B3_parameter_sequence}
\end{figure}
This system slow-fast with respect to $0 < \tilde \eps \ll 1$, and regularly perturbed by $0 < \eps \ll 1$. A direct calculation shows that $\mathcal F(0, y, \gamma_1, 0, \eps) = \mathcal F(x, 0, \gamma_1, 0, \eps) = 0$, and that the critical manifold is given by $\mathcal S = \mathcal S^r \cup \mathcal C \cup \mathcal S^a$ as before; we refer again to Figure \ref{fig:B3_kappa1=1/2}. 
Another direct calculation shows that the reduced problem on the normally hyperbolic part $\mathcal{S}^r \cup \mathcal S^a$ is given by 
\[
\Dot{y} = \frac{1}{4} - \frac{\Hat{H}(y)}{2} , \qquad y \neq 0 .
\]
The extension of this flow to $y \in \R$ has a unique equilibrium at $y = 0$, i.e.~at $(0,0) \in \mathcal C$. In order to understand the dynamics near $(0,0)$, we apply the following spherical blow-up to the extended system obtained after appending $\tilde \eps' = 0$ to system \eqref{eq:B3_parameter_scaling_chart} in the variable/state space:
\begin{equation}
\label{eq:B3_BT_BU}
    \rho \geq 0 , \ (\Bar{x},\Bar{y},\Bar{\tilde \eps}) \in \mathbb{S}^2 \mapsto 
    \begin{cases}
        x=\rho \Bar{x}, \\
        y=\rho \Bar{y}, \\ 
        \tilde \eps=\rho \Bar{\tilde \eps} .
    \end{cases}
\end{equation}
We shall work with the (desingularised) equations in the scaling chart $\bar{\tilde \eps} = 1$ with coordinates
$$x = \tilde \eps x_2, \qquad y = \tilde \eps y_2, $$ 
which can be shown to be given by
\begin{equation} \label{eq:B3_BT_blow-up_scaling}
    \begin{aligned} 
    x_2' & =\frac{1}{\theta_1} \left( -\frac{x_2 y_2}{8}-\gamma_1 \right) +\mathcal{O}(\tilde \eps, \eps) ,\\
    y_2' &= \frac{1}{\theta_2} \left( -\frac{x_2 + y_2}{8} -\gamma_2 \right)+\mathcal{O}(\tilde \eps^2, \eps) .
\end{aligned}
\end{equation}
Note that we have used the fact that $\Hat{H}(z)=\tfrac{1}{2}+\tfrac{z}{4}+\mathcal{O}(z^3)$ as $z \to 0$ in order to obtain these equations. The upshot of all of this is that system \eqref{eq:B3_BT_blow-up_scaling} is a \textit{regular perturbation} of the limiting system
\begin{equation}\label{eq:B3_BT_blow-up_scaling_rho=0}
    \begin{aligned} 
    x_2 '&= \frac{1}{\theta_1}\left( -\frac{x_2 y_2}{8} - \gamma_1 \right) , \\
    y_2' &= \frac{1}{\theta_2}\left( -\frac{x_2 + y_2}{8} - \gamma_2 \right) ,
\end{aligned}
\end{equation}
which we can analyse using standard bifurcation theory.

\begin{prop}
\label{prop:BT_B3}
    Consider system \eqref{eq:B3_BT_BU} and fix $M > 0$. There exists $\tilde \eps_0, \eps_0 > 0$ such that for all $\tilde \eps \in [0,\tilde \eps_0)$ and $\eps \in [0,\eps_0)$, the following assertions are true:
    \begin{itemize}
        \item[(i)] There are $C^1$-smooth functions $\gamma_{i,BT} : [0,\tilde \eps_0) \times [0,\eps_0) \to \R$, $i \in \{1,2\}$, such that there is a regular Bogdanov-Takens bifurcation when $(\gamma_1, \gamma_2) = (\gamma_{1,BT}(\tilde \eps, \eps), \gamma_{2,BT}(\tilde \eps, \eps))$, where
        \[
        \left(\gamma_{1,BT}(0, 0), \gamma_{2,BT}(0, 0) \right) = \left( - \frac{\theta_1^2}{8\theta_2^2}, \frac{\theta_1}{4\theta_2} \right) .
        \]
        \item[(ii)] There is a $C^1$-smooth function $\gamma_{1,sn} : [-M,M] \times [0,\Tilde \eps_0) \times [0, \eps_0)$ such that there is a regular saddle-node bifurcation when $\gamma_1 = \gamma_{1,sn}(\gamma_2, \tilde \eps, \eps)$ unless $\gamma_2 = \gamma_{2,BT}(\tilde \eps, \eps)$. We have
        \[
        \gamma_{1,sn}(\gamma_2, 0, 0) = - 2 \gamma_2^2 , \qquad \gamma_2 \neq \gamma_{2,BT}(0,0).
        \]
        \item[(iii)] There is a $C^1$-smooth function $\gamma_{1,h} : [-M,M] \times [0,\tilde \eps_0) \times [0, \eps_0)$ such that there is a subcritical Hopf bifurcation when $\gamma_1 = \gamma_{1,h}(\gamma_2, \tilde \eps, \eps)$ and $\gamma_2 > \gamma_{2,BT}(\tilde \eps, \eps)$. We have
        \[
        \gamma_{1,h}(\gamma_2, 0, 0) = \frac{\theta_1}{\theta_2} \left( \frac{\theta_1}{8 \theta_2} - \gamma_2 \right) .
        \]
    \end{itemize}
\end{prop}

\begin{proof}
We prove the corresponding (but simpler) statement for the limiting system \eqref{eq:B3_BT_blow-up_scaling_rho=0}; the perturbation arguments which guarantee persistence for all $0 \leq \tilde \eps, \eps \ll 1$ are similar to those presented in earlier proofs, e.g.~the proof of Proposition \ref{prop:B2}, so we omit them here for brevity. 

System \eqref{eq:B3_BT_blow-up_scaling_rho=0} has equilibria at
\begin{align*}
    (x_2^1 , y_2^1) &= \left(-4\gamma_2 - 4\sqrt{\gamma_2^2+\frac{\gamma_1}{2}},-4\gamma_2 + 4\sqrt{\gamma_2^2+\frac{\gamma_1}{2}}\right) ,\\
    (x_2^2 , y_2^2) &= \left(-4\gamma_2 + 4\sqrt{\gamma_2^2+\frac{\gamma_1}{2}},-4\gamma_2 - 4\sqrt{\gamma_2^2+\frac{\gamma_1}{2}}\right) ,
\end{align*}
whenever $\gamma_2^2 + \gamma_1 / 2 \geq 0$. The Jacobian of \eqref{eq:B3_BT_blow-up_scaling_rho=0} is given by
$$J=\begin{pmatrix}
    -\frac{y}{8 \theta_1} & - \frac{x}{8\theta_1}\\
    -\frac{1}{8\theta_2} & -\frac{1}{8\theta_2}
\end{pmatrix}.$$
A direct calculation shows that the Jacobian matrix $J$ associated with the linearisation at $(x_2^{1,2}, y_2^{1,2})$ satisfies $\trace J = \det J = 0$ if an only if $x_2^{1,2} = y_2^{1,2} = z_{2,BT}$, where $z_{2,BT} = -\theta_1 / \theta_2$ and
\[
\gamma_1 = \gamma_{1,BT} = - \frac{\theta_1^2}{8 \theta_2^2}, \qquad
\gamma_2 = \gamma_{2,BT} = \frac{\theta_1}{4 \theta_2} .
\]
This suggests the presence of a Bogdanov-Takens bifurcation at $(x_2, y_2) = (z_{2,BT}, z_{2,BT})$ when $(\gamma_1, \gamma_2) = (\gamma_{1,BT},\gamma_{2,BT})$.
%
In order to prove this, we first translate the proposed Bogdanov-Takens point to the origin via
$$X=x_2-z_{2,BT}, \qquad Y=y_2-z_{2,BT}, \qquad \Tilde{\gamma}_1=\gamma_1-\gamma_{1,BT}, \qquad \Tilde{\gamma}_2=\gamma_2-\gamma_{2,BT},$$
which transforms system \eqref{eq:B3_BT_blow-up_scaling_rho=0} into
\begin{equation}\label{eq:B3_BT_blow-up_scaling_BT_transformed}
\begin{aligned} 
    X'&=\frac{1}{\theta_1}\left(-\frac{XY}{8}+\frac{\theta_1}{8\theta_2}(X+Y)-\Tilde{\gamma}_1 \right) , \\ 
    Y'&=\frac{1}{\theta_2}\left( -\frac{X+Y}{8} -\Tilde{\gamma}_2\right). 
\end{aligned}
\end{equation}
System \eqref{eq:B3_BT_blow-up_scaling_BT_transformed} has an equilibrium at $X = Y = \tilde \gamma_1 = \tilde \gamma_2 = 0$, and the Jacobian matrix associated with the linearisation at this point is given by
\begin{equation} \label{eq:B3_singular_BT_Jacobian_at_equilibrium}
   J_0 = \frac{1}{8\theta_2} \begin{pmatrix}
        1 & 1\\
        -1 & -1
    \end{pmatrix} . 
\end{equation}
As expected, $\trace J_0 = \det J_0 = 0$, so $J_0$ has two zero eigenvalues.

In order to confirm that system \eqref{eq:B3_BT_blow-up_scaling_BT_transformed} has a regular Bogdanov-Takens bifurcation, we need to check the conditions (BT.0)-(BT.3) for the application of \cite[Theorem 8.4]{Kuznetsov_1998}; we refer back to the proof of Proposition \ref{lem:B2_BT} for details. (BT.0) is immediate; equation \eqref{eq:B3_singular_BT_Jacobian_at_equilibrium} shows that $J_0 \neq \mathbb O_{2,2}$. In order to check (BT.1)-(BT.2), we need the (generalized) right and left eigenvectors of $J_0$, which we denote again by $v_0, v_1$ and $w_1, w_0$ respectively. These are given by 
\[
v_0=\begin{pmatrix}
    1\\-1
\end{pmatrix}, \qquad 
v_1=\begin{pmatrix}
    1\\0
\end{pmatrix}, \qquad
w_1=\begin{pmatrix}
    1\\1
\end{pmatrix}, \qquad 
w_0=\begin{pmatrix}
    0\\-1
\end{pmatrix}, \qquad 
\]
and satisfy the normalisation property $\langle v_0,w_0 \rangle=\langle v_1,w_1\rangle=1$ and $\langle v_1,w_0\rangle=\langle v_0,w_1\rangle=0$.
The linear coordinate change
\[
z_1 = \langle Z,w_0 \rangle = -Y , \qquad
z_2 = \langle Z,w_1 \rangle = X + Y ,
\]
puts system \eqref{eq:B3_BT_blow-up_scaling_BT_transformed} into the following local normal form:
\begin{equation}\label{eq:B3_BT_blow-up_scaling_BT_transformed_normal_form}
    \begin{aligned} 
    z_1'&=\frac{1}{\theta_2}\left(\frac{z_2}{8}+\Tilde{\gamma}_2\right) , \\
    z_2'&=\frac{1}{\theta_1}\left(\frac{z_1(z_2+z_1)}{8}-\Tilde{\gamma}_1-\frac{\theta_1 \Tilde{\gamma}_2}{\theta_2} \right) .
\end{aligned}
\end{equation}
We may write this as $(z_1', z_2') = (F_1(z_1, z_2, \tilde \gamma) , F_2(z_1, z_2, \tilde \gamma))$, where $\tilde \gamma := (\tilde \gamma_1, \tilde \gamma_2)$. A direct calculation shows that
\[
a_{20}(\Tilde{\gamma}) = \frac{\partial^2 F_1}{\partial z_1^2} (0,0,\tilde \gamma) = 0, \qquad 
b_{11}(\Tilde{\gamma}) = \frac{\partial^2 F_2}{\partial z_1 \partial z_2} (0,0,\tilde \gamma) = \frac{1}{8\theta_1}, \qquad
b_{20}(\Tilde{\gamma}) = \frac{\partial^2 F_2}{\partial z_1^2} (0,0,\tilde \gamma) = \frac{1}{4 \theta_1} .
\]
Therefore,
\begin{equation}
\label{eq:BT_coefficients}
    a_{20}(0) + b_{11}(0) = \frac{1}{8\theta_1} > 0 , \qquad
    b_{20}(0) = \frac{1}{4 \theta_1} > 0 ,
\end{equation}
which shows that (BT.1)-(BT.2) are satisfied. The former inequality in particular shows that the associated Hopf bifurcation is subcritical in a neighbourhood of the Bogdanov-Takens point.

It remains to check (BT.3), i.e., the regularity of the map $(X,Y,\tilde{\gamma}) \mapsto (f_1,f_2,\trace J_0, \det J_0)$ at $(0,0,0)$, where the functions $f_i$ are defined via the right-hand side of \eqref{eq:B3_BT_blow-up_scaling_BT_transformed}, i.e.~via $(X', Y') = (f_1(X, Y, \tilde \gamma), f_2(X, Y, \tilde \gamma))$. This can be verified with a direct calculation that shows that the determinant associated with the differential of this map has the form $-1 / 8^3 \theta_1^3\theta_2^2 \neq 0$. Thus (BT.3) is satisfied, which shows that system \eqref{eq:B3_BT_blow-up_scaling_BT_transformed} has a regular Bogdanov-Takens bifurcation at the origin when $\tilde \gamma = (0, 0)$. This shows that the limiting system \eqref{eq:B3_BT_blow-up_scaling_rho=0} has a regular Bogdanov-Takens bifurcation when $(\gamma_1, \gamma_2) = (\gamma_{1,BT}, \gamma_{2,BT})$. Perturbation arguments analogous to those which were given in earlier proofs imply Assertion (i).

Assertions (ii)-(iii) can also be proven for the limiting system \eqref{eq:B3_BT_blow-up_scaling_rho=0}, and subsequently shown to persist for $0 \leq \tilde \eps, \eps \ll 1$. The leading order expressions for the saddle-node and Hopf curves $\gamma_{1,sn}(\gamma_2,0,0)$ and $\gamma_{1,h}(\gamma_2,0,0)$ are obtained directly as solutions to $\det J|_{(x^{1,2}_2, y^{1,2}_2)} = 0$ and $\trace J|_{(x_2^1,y_2^1)} = 0$ respectively. Here $J$ is the Jacobian of \eqref{eq:B3_BT_blow-up_scaling_rho=0}.
Note that $\det J|_{(x_2^-,y_2^+)}>0$ if and only if $\gamma_2>\gamma_{2,BT}(0,0)$, which restricts the Hopf curve $\gamma_{1,h}(\gamma_2,0,0)$.

It remains to check the genericity conditions associated with the saddle-node and Hopf bifurcations. The saddle-node conditions (SN.1)-(SN.2) are the same as in Proposition \ref{prop:B2_SN_Hopf}. To verify (SN.1)-(SN.2), we denote the right-hand side of system \eqref{eq:B3_BT_blow-up_scaling_rho=0} by $F(x_2, y_2, \gamma_1, \gamma_2)$ and note that the saddle-node bifurcation occurs when $x^{1,2}_2 = y^{1,2}_2 = z_{sn} = - 4 \gamma_2$. A direct calculation shows that $J|_{(z_{sn},z_{sn})}$ has a zero eigenvalue with associated left and right eigenvectors $w = (1, 4 \gamma_2 \theta_2 / \theta_1)^T$ and $v = (1, -1)$ respectively. 
Using these expressions, we find that (SN.1) is satisfied with
$$w^TF_{\gamma_1}(z_{SN},z_{SN},\gamma_1,\gamma_{2,sn})=-\frac{1}{\theta_1}\neq 0,$$
and that (SN.2) is satisfied with
$$w^TD^2F(z_{SN},z_{SN},\gamma_1,\gamma_{2,sn})[v,v]=\frac{1}{4\theta_1}\neq 0.$$
This proves that the limiting system \eqref{eq:B3_BT_blow-up_scaling_rho=0} undergoes a saddle-node bifurcation when $\gamma_1 = \gamma_{1,sn} = -2 \gamma_2^2$. Perturbation arguments similar to those applied in earlier proofs yield Assertion (ii).

The transversality condition (H.1) for the Hopf bifurcation is satisfied with 
$$\frac{\partial \trace J|_{(x_2^1,y_2^1)}}{\partial \gamma_1}(\gamma_{1,h}(\gamma_2,0,0))=\frac{1}{|\gamma_2-\frac{\theta_1}{4\theta_2}|}\neq 0$$
and defined for all $\gamma_2 \neq \gamma_{2,BT}(0,0)$. Thus, it remains to show that the first Lyapunov coefficient $l_1$ satisfies $l_1 > 0$. Although the first inequality in \eqref{eq:BT_coefficients} is sufficient to ensure that the Hopf bifurcations are subcritical in a neighbourhood of the Bogdanov-Takens point, a separate calculation is needed to prove subcriticality for $\gamma_2$ values bounded above $\gamma_{2,BT}(0,0) =\theta_1/ 4 \theta_2$. We calculated the first Lyapunov coefficient directly (after a linear translation and a subsequent basis change) using \cite[formula (3.4.11)]{Guckenheimer_1983}, and found that
\[
l_1 = \frac{\theta_{2}}{128 \theta_{1} \left(4 \gamma_{2} \theta_{2} - \theta_{1}\right)} .
\]
This confirms that the Hopf bifurcation is subcritical for all $\gamma_2 >\gamma_{2,BT}(0,0) = \theta_1 / 4 \theta_2$. Perturbation arguments similar to those applied in earlier proofs imply Assertion (iii).
%
\end{proof}


Proposition \ref{prop:BT_B3} justifies the two-parameter bifurcation set which appears on the blow-up sphere in the (blown-up) bifurcation diagram in Figure \ref{fig:main_results} and can be used to prove Theorem \ref{thm:B3_BT}.

\begin{proof}[Proof of Theorem \ref{thm:B3_BT}]
    Theorem \ref{thm:B3_BT} follows directly from Proposition \ref{prop:BT_B3} after applying the relevant blow-down transformation, which forces the restriction to $(\tilde\eps_1, \eps_2) \in (0,\tilde \eps_{1,0}) \times (0, \eps_{2,0})$.
\end{proof}


We conclude this section with some remarks on the global dynamics. In particular, as with our analyses in $B_1$ and $B_2$ above, we can identify a heteroclinic-like connection of the continuation of $\mathcal{S}_2^+$ and one of the equilibria of \eqref{eq:B3_BT_blow-up_scaling_rho=0}. We notice that setting $\gamma_1=0$ in \eqref{eq:B3_BT_blow-up_scaling_rho=0} leads to
\[
x_2'|_{x_2 = 0} = 0 , \qquad 
y_2'|_{x_2 = 0} = \frac{1}{\theta_2}\left( -\frac{y_2}{8} - \gamma_2 \right),
\]
which shows that the $y_2$-axis and (singular) extension of the extension of the critical manifolds $\mathcal{S}_2^{\pm}$ for $\gamma_1=1/2$) is invariant. Furthermore, we can check that for $\gamma_1=0$ the equilibria of \eqref{eq:B3_BT_blow-up_scaling_rho=0} simplify to 
$$(x_2^1,y_2^1)=(-8\gamma_2,0),\qquad (x_2^2,y_2^2)=(0,-8\gamma_2).$$
The eigenvalues of the Jacobian $J$ evaluated at the equilibria are given by 
$$\lambda_1^1=-\frac{1}{16}+\frac{\sqrt{1-32\gamma_2}}{16}, \ \lambda_2^1=-\frac{1}{16}-\frac{\sqrt{1-32\gamma_2}}{16}; \quad \lambda_1^2=\gamma_2,\ \lambda_2^2=-\frac{1}{8},$$
respectively. We conclude that the equilibrium $(x_2^2,y_2^2)$ which lies on the invariant $y_2$-axis corresponds to $q_s$ ($q_{f/n})$) for $\gamma_2>0$ ($\gamma_2<0$).
This implies a heteroclinic connection between the endpoint of $\mathcal{S}_2^+$ and $q_s$ ($q_{f/n}$) when $\gamma_2>0$ ($\gamma_2<0$).

\begin{rem} \label{rem:B3_heteroclinic_connection}
The preceding observations suggest the existence of a function $\gamma_{1,het}(\gamma_2,\Tilde{\eps}_1,\eps_2)$ with $\gamma_{1,het}(\gamma_2,0,0)=0$ such that the heteroclinic perturbs to a kind of `separatrix' for $\gamma_1=\gamma_{1,het}(\gamma_2,\Tilde{\eps}_1,\eps_2)$ along which the slow manifold, which perturbs from $\mathcal{S}_2^+$, coincides with the stable manifold of $q_s$ if $\gamma_2>0$. These observations, together with Proposition \ref{prop:B3_Hopf_canard} Assertion (ii) below, justify our arguments for the global dynamics in Remark \ref{rem:main_global_dynamics}. 
They also explain the extension of the canard branch for $\kappa_1 = \kappa_{1,c}$ onto the blow-up sphere and up to the turning point on the saddle-node curve (SN) shown in Figure \ref{fig:B3_blown-up_bifurcation_set} below. Since we do not consider the perturbation of the heteroclinic connection identified above, however, the expected behaviour after perturbation shown in Figure \ref{fig:B3_blown-up_bifurcation_set} is only a conjecture.  
\end{rem}

\subsubsection{Singular Hopf bifurcation and canards} 
\label{sec:B3_Hopf_and_canard}


It remains to prove Theorem \ref{thm:B3}. In order to do so, we need to resolve the degeneracy which occurs along the line segment $\kappa_1 = 1/2$, $\kappa_2 > 3/4$ and $\tilde \eps = 0$ of the bifurcation set shown in Figure \ref{fig:B3_spherical_parameter_blow-up}. This is necessary in order to proof Assertions (ii)-(iii) in Theorem \ref{thm:B3}, and will require one final blow-up in parameters, followed by additional blow-ups in the phase space. Since $\kappa_2 > \frac{3}{4}$, we shall work in the coordinate chart $\Bar{\kappa}_2=1$ associated with the blow-up transformation defined in \eqref{eq:B3_parameter_BU_spherical}. We denote the relevant parameters in this chart by 
\[
\tilde \eps = \sigma_3 \tilde \eps_3 ,\qquad 
\kappa_1 = \frac{1}{2} + \sigma_3^2 \kappa_{1,3}, \qquad
\kappa_2 = \frac{3}{4} + \sigma_3.
\]
Applying this transformation to system \eqref{eq:B3_spherical_cylinder_scaling_chart} (after dropping the subscripts there) leads to
\begin{equation} \label{eq:B3_parameter_upper_chart}
\begin{split}
    x' &= \overline{\mathcal F}(x,y,\kappa_{1,3}, \sigma_3, \tilde \eps_3) , \\
    y' &= \sigma_3 \tilde \eps_3 \overline{\mathcal G}(x,y,\kappa_{1,3}, \sigma_3, \tilde \eps_3) ,
\end{split}
\end{equation}
where $\overline{\mathcal F}, \overline{\mathcal G}$ are smooth functions which satisfy
\[
\begin{split}
    \overline{\mathcal F}(x,y,\kappa_{1,3}, \sigma_3, \tilde \eps_3) &= \frac{1}{\theta_1}[\Hat{H}(x)+\Hat{H}(y)-2\Hat{H}(x)\Hat{H}(y)-\frac{1}{2}-\sigma_3^2\kappa_{1,3}] + \mathcal{O}(\sigma_3 \tilde \eps_3 \eps) , \\
    \overline{\mathcal G}(x,y,\kappa_{1,3}, \sigma_3, \tilde \eps_3) &= \frac{1}{\theta_2}[1-\Hat{H}(x)\Hat{H}(y)-\frac{3}{4}-\sigma_3+\mathcal{O}(\eps)] .
\end{split}
\]
The layer problem \eqref{eq:B3_parameter_upper_chart}$|_{\tilde \eps_3 = 0}$ features the same critical manifold $\mathcal S^r \cup \mathcal C \cup \mathcal S^a$ when $\kappa_{1,3} = 0$, for all $\sigma_3 \geq 0$. This motivates the successive cylindrical blow-up in parameters which is sketched in Figure \ref{fig:B3_parameter_blow-up_cylinder_attached}, which is defined via
\begin{equation}
    \label{eq:B3_parameter_BU_cylinder}
    \omega \geq 0, \ (\bar \kappa_{1,3}, \bar{\tilde \eps}_3) \in \mathbb S^1 \mapsto 
    \begin{cases}
        \kappa_{1,3} = \omega \bar \kappa_{1,3} , \\
        \tilde \eps_3 = \omega^2 \bar{\tilde \eps}_3 .
    \end{cases}
\end{equation}
We shall focus entirely on the dynamics in the $\bar{\tilde \eps}_3 = 1$ chart, where we write $(\kappa_{1,3}, \tilde \eps_3) = (\omega_2 \kappa_{1,3,2}, \omega_2^2)$. Writing out the equations in this chart, dropping the subscripts, resolving the loss of hyperbolicity in variable space along $\mathcal C$ for $\omega=0$ via the blow-up transformation
\begin{equation}
\label{eq:y_blow-up}
    \rho \geq 0, \ (\bar y, \bar \omega) \in \mathbb S^1 \mapsto 
    \begin{cases}
        y = \rho \bar y, \\
        \omega = \rho \bar \omega ,
    \end{cases}
\end{equation}
applied to an extended system with $\omega' = 0$ appended and deriving the equations in the associated scaling chart $\bar \omega = 1$ where $y = \omega y_2$, we arrive at the (desingularised) system
\begin{equation} \label{eq:B3_parameter_upper_chart_cylinder}
\begin{split}
     x' &= \frac{1}{\theta_1}\left( \frac{y_2}{4}(1-2\Hat{H}(x)) - \sigma_3^2\kappa_{1,3,2} \right) + \mathcal{O}(\omega^2, \sigma_3 \eps) , \\
     y_{2}' &= \frac{\sigma_3}{\theta_2}\left( \frac{1}{4} - \sigma_3 - \frac{\Hat{H}(x)}{2} - \omega \Hat{H}(x)\frac{y_2}{4} +\mathcal{O}(\omega^3, \eps \omega) \right) ,
\end{split}
\end{equation}
which we consider with $0 \leq \omega, \eps \ll 1$, $\sigma_3 \geq 0$ and $\kappa_{1,3,2} \in \R$. Finally, for the sake of simplicity, we restrict our analysis to $\sigma_3 > 0$ and apply one last rescaling defined by
\begin{equation}
\label{eq:rescaling}
    y_2 = \sigma_3^{1/2} \Tilde{y}_2, \qquad 
    \kappa_{1,3,2} = \sigma_3^{-3/2} \gamma_1,
\end{equation}
together with a time rescaling which amounts to division of the right-hand side by $\sigma_3^{1/2}$ to obtain
\begin{equation} \label{eq:B3_kappa1=0.5_cylinder}
\begin{aligned}
         x' &= \frac{1}{\theta_1}\left( (1-2\Hat{H}(x)) \frac{\tilde y_{2}}{4} - \gamma_1 \right)  +\mathcal{O}(\omega^2,\eps) , \\
         \tilde y_{2}' &= \frac{1}{\theta_2}\left( 1 - \kappa_2  - \frac{\Hat{H}(x)}{2} - \frac{\omega \sqrt{\kappa_2 - 3/4} }{4} \Hat{H}(x) \tilde y_{2} \right) + \mathcal{O}(\omega^3, \eps \omega) ,
\end{aligned}
\end{equation}
where we also used $\kappa_2 = \sigma_3 + 3/4$.

\begin{rem}
\label{rem:restriction}
    By restricting to $\sigma_3 > 0$, we are essentially ignoring the `matching problem' associated with the connection between the blow-up sphere and cylinder in $(\kappa_1, \kappa_2, \tilde \eps)$-parameter space along $\sigma_3 = \omega = 0$. 
    Because of this, system \eqref{eq:B3_kappa1=0.5_cylinder} could be obtained directly from system \eqref{eq:B3_spherical_cylinder_scaling_chart} by a cylindrical parameter blow-up with $\kappa_1 = 1/2 + \sqrt{\Tilde{\eps}} \gamma_1$ together with an associated rescaling chart to resolve the loss of normal hyperbolicity along $\mathcal{C}$ given by $y=\sqrt{\Tilde{\eps}}y_2$. We opted for the less direct derivation of system \eqref{eq:B3_kappa1=0.5_cylinder} above because it makes it clear how one should proceed if the intention is to provide a complete treatment; in this case, one must also consider system \eqref{eq:B3_parameter_upper_chart_cylinder} in a neighbourhood of $\sigma_3 = \omega = 0$ (where it must be considered as a slow-fast system with $0 \leq \sigma_3 \ll 1$).
\end{rem}

System \eqref{eq:B3_kappa1=0.5_cylinder} is a regular perturbation of the following limiting problem with $\omega = \eps = 0$:
\begin{equation} \label{eq:B3_kappa1=0.5_cylinder_rho=0}
\begin{aligned}
    x'&=\frac{1}{\theta_1}\left( \frac{\tilde y_{2}}{4}(1-2\Hat{H}(x))-\gamma_1 \right), \\
    \tilde y_{2}'&=\frac{1}{\theta_2}\left( 1- \kappa_2  - \frac{\Hat{H}(x)}{2} \right).
\end{aligned}
\end{equation}
The geometry and dynamics of this system is described in the following result. 

\begin{lem} \label{lem:B3_phaseportrait_kappa2not0.75}
    Fix $\kappa_2 \in (1/2,1) \setminus \{3/4\}$. Then system \eqref{eq:B3_kappa1=0.5_cylinder_rho=0} has the following properties:
    \begin{enumerate}
        \item[(i)] There is a unique equilibrium
        \begin{equation} \label{eq:B3_center_coordinates}
            q: \left( \ln \frac{2-2\kappa_2}{2\kappa_2-1}, \frac{4 \gamma_1}{4\kappa_2-3} \right).
        \end{equation}
        \item[(ii)] If $\kappa_2<\frac{3}{4}$, then $q$ is of saddle type. If $\kappa_2 > 3/4$ and $\gamma_1 > 0$ ($\gamma_1 < 0$), $q$ is an attracting (repelling) focus/node.
        \item[(iii)] If $\gamma_1=0$, then trajectories are contained within constant level sets of the function
        $$
        \mathcal H(x, \tilde y_2) := \frac{\tilde y_{2}^2}{2} + 4(\kappa_2-1) x + 2(3-4\kappa_2) \ln \vert e^x-1 \vert .
        $$
        For $\kappa_2>\frac{3}{4}$ and $\gamma_1=0$, $q$ is a nonlinear center. 
    \end{enumerate}
\end{lem}

\begin{proof}
    This follows from direct calculations, so we omit the details for brevity.
\end{proof}

\begin{figure}[t!]
        \centering
        \includegraphics[width=0.5\linewidth]{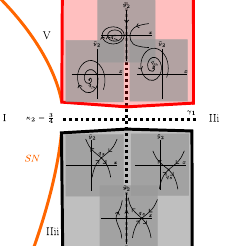}
    \caption{Phase portraits of system \eqref{eq:B3_kappa1=0.5_cylinder_rho=0}, which describes the dynamics on the blow-up cylinder which replaces the non-hyperbolic line $\mathcal C$, in different regions of the blown-up bifurcation set. The shaded red and grey regions are correlated with the scaling region (S3), where $\gamma_1 \sim 0$ (i.e.~$\kappa_1 \sim 1/2$) and $\kappa_2 \neq 3/4$ as $\tilde \eps_1 \to 0$. There is a saddle-type equilibrium $q_s$ when $\kappa_2 < 3/4$ (shaded grey region), and a focus/node-type equilibrium $q_{f/n}$ which undergoes a (vertical) Hopf bifurcation when $\kappa_2 > 3/4$ (shaded red region) and $\gamma_1$ is varied over $\gamma_1 = 0$; we refer to Lemma \ref{lem:B3_phaseportrait_kappa2not0.75} for details, and to Figure \ref{fig:B3_singular_canard_1} for a more global phase portrait in the special case with $\gamma_1 = 0$ and $\kappa_2 > 3/4$.}
    
    \label{fig:B3_blow-up_kappa1=0.5}
\end{figure}

\begin{rem}
    Lemma \ref{lem:B3_phaseportrait_kappa2not0.75} includes results for $\kappa_2 \in (1/2, 3/4)$, even though the sequence of transformations which led from system \eqref{eq:B3_spherical_cylinder_scaling_chart} to the limiting system \eqref{eq:B3_kappa1=0.5_cylinder_rho=0} requires that $\kappa_2 > 3/4$ because $\kappa_2 = 3/4 + \sigma_3$ with $\sigma_3 \geq 0$ in the $\bar \kappa_2 = 1$ chart associated with the blow-up \eqref{eq:B3_parameter_BU_spherical}. These results are valid because the same system can be obtained by an analogous sequence of transformations which start from the $\bar \kappa_2 = -1$ chart instead. We have omitted the details here for the sake of brevity. Alternatively, since $\kappa_2 \neq 3/4$, one could apply a cylindrical blow-up directly, as described in Remark \ref{rem:restriction}.
\end{rem}

A number of representative phase portraits described by Lemma \ref{lem:B3_phaseportrait_kappa2not0.75} are shown in Figure \ref{fig:B3_blow-up_kappa1=0.5}. The phase portrait described in Assertion (iii), when $\kappa_2 > \frac{3}{4}$ and $\gamma_1=0$, 
%
%
determines the dynamics on the `central' blow-up cylinder in shaded red in Figure \ref{fig:B3_singular_canard_1}, and closely resembles the phase portrait that is known to arise within the scaling chart associated with the unfolding of singular Hopf and canard-related phenomena that has been studied in, e.g.~\cite{dumortier1996canard, Maesschalck_2021, Krupa_2001_Canard}. This is not surprising, given that singular Hopf and canard phenomena have been identified in the study of regularised two-fold singularities in \cite{Bonet2018,Kristiansen_2015,Kristiansen_2023}. Using the same techniques that have been applied in all of these works, we can verify the existence of singular Hopf and canard solutions in system \eqref{eq:B3_kappa1=0.5_cylinder}. This is formalised in the following result.


\begin{prop} \label{prop:B3_Hopf_canard}
    Fix $\kappa_2 \in (3/4,1)$. There exists $\omega_0, \eps_0 > 0$ such that for all $\omega \in [0,\omega_0)$ and $\eps \in [0,\eps_0)$, system \eqref{eq:B3_kappa1=0.5_cylinder} has an equilibrium $q_{f/n}(\omega,\eps)$ which converges to \eqref{eq:B3_center_coordinates} as $(\omega, \eps) \to (0,0)$. Moreover, the following assertions are true:
    \begin{enumerate}
        \item[(i)] There is a $C^1$-smooth function $\gamma_{1,h} : [0, \omega_0) \times [0, \eps_0) \to \R$ with $\gamma_{1,h}(0,0) = 0$, for which
        $q_{f/n}(\omega, \eps)$ undergoes a (possibly degenerate) Hopf bifurcation when $\gamma_1 = \gamma_{1,h}(\omega, \eps)$;
        \item[(ii)] there is a $C^1$-smooth function $\gamma_{1,c} : [0, \omega_0) \times [0, \eps_0) \to \R$ with $\gamma_{1,c}(0,0) = 0$, for which there is a canard-like intersection of the forward (backward) extension of the attracting (repelling) slow manifolds $\mathcal S^a_{\tilde \eps}$ ($\mathcal S^r_{\tilde \eps}$) described in Lemma \ref{lem:B3_slow-manifold} when $\gamma_1 = \gamma_{1,c}(\omega, \eps)$. 
    \end{enumerate}
\end{prop}

\begin{rem}
\label{rem:Hopf_proof}
    The qualifier ``possibly degenerate'' has been included in our references to the singular Hopf bifurcation because we have not computed the first Lyapunov coefficient, which should be nonzero for a nondegenerate Hopf bifurcation. Although we were able to prove that the bifurcation is subcritical in the system \eqref{eq:B3_kappa1=0.5_cylinder}$|_{\eps = 0}$ with $0 < \omega \ll 1$ using \cite[Theorem 2.6]{Chow_1994} (see also \cite[Theorem 3.1]{Krupa_2001_Canard}), we were unable to prove subcriticality in an entire neighbourhood of $(\omega, \eps) = (0,0)$. We believe that subcriticality can be proven after (i) blowing up the point $(\omega, \eps) = (0,0)$ to a quarter circle, and (ii) proving the desired result in the relevant directional charts.
\end{rem}

\begin{proof}
Direct calculations show that the Jacobian $J(q)$ associated with the limiting system \eqref{eq:B3_kappa1=0.5_cylinder_rho=0}, evaluated at $q$, satisfies
\[
\det J(q) = \frac{1}{8} \theta_1^{-1} \theta_2^{-1} \hat H'(q) (4 \kappa_2 - 3) > 0
\]
for all $\kappa_2 \in (3/4, 1)$. We also find that the trace
\[
\mathcal G(\gamma_1) := \textup{Tr} J (q) = - 2 \theta_1^{-1} \frac{\gamma_1 \hat H'(q_x)}{4 \kappa_2 - 3} .
\]
satisfies $\mathcal G(0) = 0$ and
\[
\frac{d \mathcal G}{d \gamma_1} (0) = - 2 \theta_1^{-1} \frac{\hat H'(q_x)}{4 \kappa_2 - 3}  < 0 ,
\]
for all $\kappa_2 \in (3/4,1)$, thereby confirming that $q$ undergoes a (possibly degenerate) Hopf bifurcation in the limiting system \eqref{eq:B3_kappa1=0.5_cylinder_rho=0} when $\gamma_1 = 0$. Combining this with regular perturbation arguments and the implicit function theorem leads to Assertion (i).

\

Assertion (ii) can be proven using an extension of the Melnikov method which applies to problems on non-compact domains which has been formulated for planar slow-fast systems in \cite{krupa_extending_transcritical}, and in arbitrary (finite) dimensions in \cite{Wechselberger2002}. Firstly, one needs to extend the Fenichel slow manifolds $\mathcal S^{a/r}_{\tilde \eps}$ into the scaling chart $\bar \omega = 1$ associated with the blow-up \eqref{eq:y_blow-up}. This should be done in the $\bar y = \pm 1$ charts. The desingularised equations in $\bar y = 1$ (where $y = \rho_1$ and $\omega = \rho_1 \omega_1)$ are given by
\[
\begin{split}
    x' &= \theta_1^{-1} \left[ \frac{1}{4} \left( 1 - 2 \hat H(x) \right) - \omega_1 \gamma_1 + \mathcal O(\rho_1^2 \omega_1^3, \eps \omega_1) \right] , \\
    \rho_1' &= \theta_2^{-1} \rho_1 \omega_1^2 \sigma_3 \left[ \frac{1}{4} - \sigma_3 - \frac{1}{2} \hat H(x) - \frac{1}{4} \rho_1 \hat H(x) + \mathcal O(\rho_1^3, \omega_1^3, \eps \rho_1 \omega_1) \right] , \\
    \omega_1' &= - \theta_2^{-1} \omega_1^3 \sigma_3 \left[ \frac{1}{4} - \sigma_3 - \frac{1}{2} \hat H(x) - \frac{1}{4} \rho_1 \hat H(x) + \mathcal O(\rho_1^3, \omega_1^3, \eps \rho_1 \omega_1) \right] .
\end{split}
\]
The center manifold theorem implies the existence of a local 2-dimensional and attracting center manifold $\mathcal M^a$ in a neighbourhood about $p_a : (0,0,0)$ (the linearisation at the origin has eigenvalues $-1/8, 0, 0$), which can be used to extend $\mathcal S^a_{\tilde \eps}$ into the scaling chart $\bar \omega = 1$; see e.g.~\cite{Krupa_2001_Extend,krupa_extending_transcritical,dumortier1996canard,Maesschalck_2021} for details. An analogous argument shows the existence of a local 2-dimensional and repelling center manifold $\mathcal M^r$ based at $p_r : (0,0,0)$ in the $\bar y = -1$ chart (where $y = -\rho_3$ and $\omega = \rho_3 \omega_3$), which can be used to extend $\mathcal S^r$ into the scaling chart $\bar \omega = 1$.

Combining the observations above with the fact that the limiting system \eqref{eq:B3_kappa1=0.5_cylinder_rho=0}$|_{\gamma_1 = 0}$ has a unique solution $\Theta(t) = (0,-\theta_2^{-1} ( \kappa_2 - 3/4) t)$ which is forward (backward) asymptotic to $p_r$ ($p_a$) when $\kappa_2 \in (3/4, 1)$ shows that Assumptions 1-3 in \cite{Wechselberger2002} are satisfied (we have permitted a slight abuse of notation by denoting the independent variable in system \eqref{eq:B3_kappa1=0.5_cylinder_rho=0} by $t$). This allows one to define a signed distance function
\[
\mathcal D(\omega, \eps, \gamma_1) := x^r(\omega, \eps, \gamma_1) - x^a(\omega, \eps, \gamma_1) ,
\]
where $x^{a/r}$ denotes the $x$-component of the intersection of the extensions of $\mathcal M^{a/r}$ (and therefore also $\mathcal S^{a/r}_{\tilde \eps}$) under the forward/backward flow induced by system \eqref{eq:B3_kappa1=0.5_cylinder} with $\{y = 0\}$ (regular perturbation theory guarantees that intersections exist for sufficiently small $0 < \omega, \eps \ll 1$). Notice that $\mathcal D(0,0,0) = 0$. Our aim in the following is to show that an intersection persists for $0 < \omega, \eps \ll 1$. Following \cite{Wechselberger2002}, our signed distance function $\mathcal D(\omega, \eps, \gamma_1)$ can be written as an expansion of the form
\[
\mathcal D(\omega, \eps, \gamma_1) = \mathcal D_\omega(0,0,0) \omega + \mathcal D_\eps(0,0,0) \eps + \mathcal D_{\gamma_1}(0,0,0) \gamma_1 + \mathcal O(|(\omega, \eps, \gamma_1)|^2) .
\]
The coefficients $\mathcal D_\alpha(0,0,0)$, where $\alpha \in \{\omega, \eps, \gamma_1\}$, are Melnikov integrals of the form
\[
\mathcal D_\alpha(0,0,0) = \int_{-\infty}^\infty \left\langle \psi(t), \frac{\partial f}{\partial \alpha}(\Theta(t), 0, 0) \right\rangle dt , 
\]
where $\psi(t) = (e^{-at^2 / 16} , 0)^T$, where $a: = \theta_1^{-1} \theta_2^{-1} ( \kappa_2 - 3/4) > 0$, is a bounded solution to the adjoint variational equation
\[
\psi'(t) = - 
\begin{pmatrix}
    - a t / 8 & 1 / (8\theta_2) \\
    0 & 0
\end{pmatrix}
\psi(t) 
\]
associated with the limiting system \eqref{eq:B3_kappa1=0.5_cylinder_rho=0}$|_{\gamma_1 = 0}$ and evaluated along $\Theta(t) = (0, - a t)$. Evaluating $\mathcal D_{\gamma_1}(0,0,0)$ in particular leads to
\[
\mathcal D_{\gamma_1}(0,0,0) = - 4 \theta_1^{-1} \sqrt{\frac{\pi}{a}} < 0 .
\]
Since $\mathcal D_{\gamma_1}(0,0,0) \neq 0$, the implicit function theorem implies the existence of the $C^1$-function $\gamma_{1,c}$ in Assertion (ii).
\end{proof}

Proposition \ref{prop:B3_Hopf_canard} confirms the presence of a singular Hopf bifurcation and canard solutions in system \eqref{eq:B3_kappa1=0.5_cylinder}. In addition to that, Proposition \ref{prop:B3_canard_uniqueness} asserts the existence of large amplitude canard cycles in system \eqref{eq:B3_spherical_scaling_chart}, and Lemma \ref{lem:B3_phaseportrait_kappa2not0.75} shows that the left-hand side of the blow-up cylinder shown in red in Figure \ref{fig:B3_singular_canard_1} is foliated by periodic orbits when $\gamma_1 = 0$. All of these features support the claim that system \eqref{eq:B3_spherical_scaling_chart} features a canard explosion in which the small amplitude unstable cycle that is born in the Hopf bifurcation, rapidly grows under exponentially small variations of $\gamma_1$ into a large amplitude canard cycle. Ultimately we expect (but do not prove) that the cycle terminates in a canard-homoclinic orbit which perturbs from a singular canard-homoclinic cycle $\Gamma_4$, which is anchored at $q_s$ and can be constructed as shown in Figure \ref{fig:B3_PWS_sing_cycles}. This is similar to the \textit{incomplete canard explosions} described in \cite{Wechselberger2015}, which also arose when a return mechanism was combined with the unfolding of a singular Bogdanov-Takens bifurcation (the local normal form for the `generic' singular Bogdanov-Takens bifurcation was treated in \cite{deMaesschalck2011}). Altogether, the preceding analyses support the claim that the (blown-up) bifurcation set is as shown in Figure \ref{fig:B3_blown-up_bifurcation_set}.

\begin{figure}[t!]
    \centering
    \includegraphics[width=0.5\linewidth]{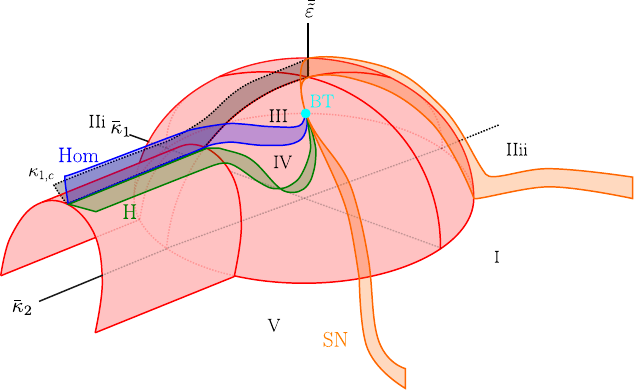}
    \caption{Conjectured $B_3$ bifurcation set in the blown-up $(\bar \kappa_1, \kappa_2, \bar{\tilde \eps})$-space, cf. earlier figures which only show the limiting bifurcation set when $\tilde \eps = 0$. Surfaces corresponding to saddle-node, subcritical Hopf, homoclinic and canard solutions are shown in orange, green, blue and grey respectively. Further details are given in the text.}
    \label{fig:B3_blown-up_bifurcation_set}
\end{figure}

We are now in a position to complete the proof of Theorem \ref{thm:B3}.
\begin{proof}[Proof of Theorem \ref{thm:B3}]
    Theorem \ref{thm:B3} follows directly from the results of this section after applying the relevant blow-down transformations. In particular:
\begin{itemize}
    \item Assertion (i) follows from 
    Proposition \ref{prop:B3_saddle-node}; 
    \item Assertion (ii) follows from 
    Proposition \ref{prop:B3_Hopf_canard} assertion (i);
    \item Assertion (iii) follows from Proposition 
    \ref{prop:B3_Hopf_canard} assertion (ii).
\end{itemize}
In each case, the relevant blow-down transformation is only defined for $\eps_1, \eps_2 > 0$. This leads to the restriction to $(\eps_1, \eps_2) \in (0, \eps_{1,0}) \times (0, \eps_{2,0})$ in each of the assertions (i)-(iii).
\end{proof}

\section{Summary and outlook} \label{sec:Outlook}
Plahte and Kjøglum's \textit{fundamental problem}, as given by system \eqref{eq:fund_prob}, has been analysed by a number of authors because it serves as a low-dimensional testbed for the study of a larger class of ODE-based models for gene regulatory dynamics; we refer again to \cite{Simon_thesis,Machina2011,Edwards2015,Plahte_GRN_2005,Quee2021} for existing studies. While there is already a significant literature devoted to the study of GRN dynamics in the (higher-dimensional counterpart to the) asymptotic regime $0 < \eps_1, \eps_2 \ll 1$, including the above-cited work of Plahte and Kjøglum, analytical studies have primarily involved an equal steepness assumption; $\eps_1 = \eps_2$ in the case of system \eqref{eq:fund_prob} and $\eps_1 = \cdots = \eps_n$ in the case of higher dimensional networks with $n$ steepness parameters $\eps_i$. In this work, we have presented a comprehensive analytical treatment of system \eqref{eq:fund_prob} with $0 < \eps_1, \eps_2 \ll 1$, over a relatively large region of the parameter space for which $(\kappa_1, \kappa_2) = (\alpha_1 \theta_1, \alpha_2 \theta_2) \in \Lambda$. We showed that distinct singular limits are obtained when the limit $(\eps_1, \eps_2) \to (0,0)$ is taken from within $B_1$, $B_2$ and $B_3$. This led to different bifurcation sets, i.e.~to different qualitative dynamics, and it is this subtle but important fact that we believe is the most significant implication for ODE-based models for GRN dynamics more generally.

Our findings strongly suggest that in practice, for a given ODE-based model for GRN dynamics, \textit{it is the relative size of the steepness parameters which determines the `correct' singular limit}. Another significant point in practice relates to the number of effective dimensions at play: our analyses and results in $B_2$ showed that the qualitative dynamics near the intersection $\Sigma_1 \cap \Sigma_2$ are effectively 2-dimensional; this is expected and perfectly consistent with existing studies which assume an equal steepness property. In both $B_1$ and $B_3$, however, the fact that $\eps_1 / \eps_2$ is either small or large leads to slow-fast dynamics near $\Sigma_1 \cap \Sigma_2$, which means that the effective dynamics there are organised (at least in large part) by 1-dimensional layer and reduced problems. In essence, the presence of steepness parameters on different asymptotic orders leads to additional time-scales, and this leads to lower dimensional reduced problems because the dimension of the invariant manifold that is expected to organise the slowest time-scale is equal to the number of slow variables on that time-scale. 
This may be significant in high-dimensional models which have only a small number of `smallest' switching parameters $\eps_i$. Indeed, if dimension-reduction is the goal, then observations such as these suggest that having $\eps_i$ values which span several orders of magnitude is \textit{desirable} for analytical and modelling purposes.

GRNs aside, one of our main intentions in this work has been to showcase and extend an analytical approach to the study of singular perturbation problems with multiple small parameters which has been developed and applied in e.g.~\cite{deMaesschalck2011,baumgartner_robertson_2025}. 
The most important aspect of this approach, the steps of which were summarised towards the end of Section \ref{sec:Intro}, is the use of a preliminary blow-up in the small parameter space. By blowing up the origin $(\eps_1, \eps_2) = (0,0)$ to a quarter-circle via \eqref{eq:parameter_blow-up} as shown in Figure \ref{fig:parameter_blow-up}, we were able to `tease apart' the different singular limits in each $B_i$ by working in parameter charts $\mathcal P_1$, $\mathcal P_2$ with small parameters that are `better suited' to local analyses in particular regions $B_i$. Later on, we applied additional blow-ups in the $(\kappa_1, \kappa_2, \tilde \eps_2)$-parameter space in order to desingularise and characterise the singular bifurcation set in $B_3$; the blow-up sequence is shown in Figure \ref{fig:B3_parameter_sequence}. Although we did not present all the details here (we refer again to Remark \ref{rem:restriction}), the $B_3$ analysis presented in Section \ref{sec:Region_B3} suggested that blow-up approaches of this kind may be very useful as a tool for connecting bifurcation curves across different asymptotic regimes. Looking forward, we are hopeful that the methodology presented herein will be useful for rigorous dynamical systems based analyses of a wide range of problems with multiple small parameters.

\section{Acknowledgements} \label{sec:Acknowledgements}
This work was inspired by the geometric blow-up analysis of Plahte and Kjøglum's fundamental problem presented in Frieder Simon's MSc thesis \cite{Simon_thesis}, which was supervised by Peter Szmolyan. We thank Peter Szmolyan in particular for introducing us to the problem and for helpful discussions throughout the project.
LB also thanks Adelaide University and the second author for their outstanding hospitality during a research stay and TU Wien for their financial support. SJ was supported by European Union Marie Skłodowska-Curie Postdoctoral Fellowship grant 101103827 in the early stages of the project, during which time he was based at TU Wien.

\printbibliography[heading=bibintoc,
title={References}]


\end{document}